\documentclass[12pt]{elsarticle}
\usepackage{amsmath, amssymb, amsthm}
\usepackage{graphicx,amssymb,amsfonts,amsbsy,amsmath,amsthm}
\usepackage{mathtools}
\usepackage{hyperref}
\usepackage[utf8]{inputenc}
\usepackage{floatrow}
\usepackage{float}
\usepackage{subfig}
\usepackage{amssymb}
\usepackage{amsthm}
\usepackage{epstopdf}
\usepackage{anysize}
\usepackage{color,soul}
\usepackage{graphicx}
\usepackage{multirow}
\newtheorem{theorem}{Theorem}[section]
\newtheorem{lemma}[theorem]{Lemma}

\DeclareMathOperator{\e}{e}
\DeclareMathOperator{\Real}{Re}
\DeclareMathOperator{\Imag}{Im}

\journal{The European Physical Journal Plus (EPJP)}
\input epsf.sty
\usepackage{lineno}
\begin{document}

\begin{frontmatter}
\title{Effect of delay and control on a predator-prey ecosystem with generalist predator and group defence in the prey species}

\author[label1]{Rajesh Ranjan Patra}
\author[label2]{Soumen Kundu\corref{cor1}}\ead{soumenkundu75@gmail.com}
\author[label1]{Sarit Maitra}

\address[label1]{Department of Mathematics, NIT Durgapur, Durgapur-713209, India}
\address[label2]{\corref{cor1} Department of Mathematics, Faculty of Science and Technology, ICFAI University Tripura, Kamalghat, Mohanpur, West Tripura-799210, India}
\cortext[cor1]{Corresponding author}

\begin{abstract}
Generalist predators consist an important component of an ecosystem which may act as a biocontrol agent and influence the dynamics significantly.
In this paper, we have studied the effect of delayed logistic growth of the prey species with group defence behaviour. The Lyapunov stability criteria for the interior equilibrium point is derived. Also, the condition of Hopf-bifurcation and the point of bifurcation are obtained. The length of the delay is also estimated for the system to preserve stability. Numerical simulations are performed and illustrated to support the obtained analytical results. Using a feedback control mechanism, the stability of the unstable equilibrium point is restored. Latin Hypercube Sampling/Partial Rank Correlation Coefficient (LHS/PRCC) sensitivity analysis, which is an efficient tool often employed in uncertainty analysis, is used to explore the entire parameter space of a model.


\end{abstract}
\end{frontmatter}

\section{Introduction}
Management of natural resources through preservation and restoration using biological control agents are drawing the attention of ecologists nowadays\cite{van2010}. Generalist predators have the potential to act as biological control agents. Biological control methods, which help in protecting the flora and fauna of an ecosystem, are used in many recovery plans\cite{causton2001}.

One way to successfully deploy biological control is by introducing a population species that preys upon the invasive species.  Using the generalist predators by exploring their ability to reduce the numbers of a pest significantly while ensuring the effective reduction of loss in crops have been discussed by the authors\cite{symondson2002}. The effect of predator-prey interactions, which may depend on the predator's per capita killing rate, is reflected in the form of functional response. Holling\cite{holling1959}, in 1959, proposed three functional responses based upon some characteristics of types of predation. Later various other types of functional responses have been introduced by Beddington-DeAngelis\cite{bedd,dea}, Arditi-Ginzburg\cite{arditi1989}, Hassell-Varley\cite{hassell}, etc. Several authors studied these functional responses in various ecosystems. Tian and Xu\cite{tian2011} studied the global dynamics of a predator-prey system using Holling type II functional response. Liu et al.\cite{liu2019} analyzed a host-parasitoid model using Holling type III functional response in the presence of the Allee effect. With Beddington-DeAngelis functional response Li et al.\cite{li2017} studied a stage-structured plant-pollinator model. Gakkhar\cite{gakkhar2002} studied chaos in a food chain model with the ratio-dependent functional response.

Forming group is one of the most fascinating behaviours seen in diverse
animal species. The characteristic is incorporated in the predator's functional responses while modelling predator-prey dynamics.
There are evidence of both predators and preys forming groups\cite{ioannou2017}. Predators form groups to maximize their predation rate, while preys form groups to reduce the predation rate. In 1969, Hassell-Varley\cite{hassell} introduced a functional response to incorporate the grouping behaviour in species. In 1999, Cosner\cite{cosner} proposed a theory on the structure of a functional response to model the grouping of individuals in a species. According to Cosner, a functional response will be both predator and prey density dependent if predators are assumed to form dense colonies. In contrast, a functional response will be only prey density dependent if predators are assumed to have homogeneous spatial distribution. Grouping in prey is often justified as group defence. The term group defence is used to characterise the phenomenon due to the ability of the prey species to defend or disguise themselves against the attacking predator species, as a result of which, the predation rate is decreased or sometimes prevented by a large number of prey individuals. In an example given by Tener\cite{tener1965}, a lone musk ox can be attacked successfully by wolves, whereas small herds of musk oxen, consisting of 2 to 6 animals, are predated but with rare success. Moreover, in large herds, no successful attacks have been observed. Another example was presented by Davidowicz, Gliwicz \& Gulati\cite{davidowicz1988}. Daphnia can feed on Filamentous algae at low concentrations, but the later can jam the filtering apparatus of the former when present in high concentrations. As a third example, large swarms of insects make individual identification difficult for their predators\cite{holmes1972}. More related examples can be found in \cite{andrews1968, edwards1970, boon1962}, where a very similar phenomenon, called as ``inhibitory effect", is presented, which limits the growth of the micro-organisms at higher concentrations of certain nutrients.

There are many literatures which use different functional responses to depict the herd behaviour in prey species. Ajraldi et al.\cite{ajraldi2011} used the square-root functional response, $F(x)=m\sqrt{x}$, to show group defence where the individuals forming the perimeter of the group are exposed to predation, which is directly proportional to the square root of the total population. Braza\cite{braza2012} used the same idea to formulate a new functional response, $F(x)=\alpha\sqrt{x}/(1+t_h\alpha\sqrt{x}),$ considering the portion of the population which is prone to attack and used the derivation process of the Holling type II functional response. Djilali\cite{djilali2019} further generalized the response function to include the rate of the pack shape($k$) and presented a functional response $F(x)=\alpha x^k/(1+t_h\alpha x^k),\, k>0.$ Geritz and Gyllenberg\cite{geritz2013} introduced a functional response where the prey population exhibits group defence by forming groups of some sizes which vary when  individuals join and leave the group. Kumar and Kumari\cite{kumarkumari2021} used the Ivlev-like response function for defence by the prey species in a tritrophic food chain model. Also, the Holling type IV response\cite{andrews1968}$F(x)=mx/(x^2+bx+c)$ and the simplified Holling type IV\cite{sokol1981} $F(x)=mx/x^2+c$ are used for group defence phenomenon.

Time delays are natural components of the dynamic processes of biology, physiology, ecology, epidemiology etc.\cite{gopalsamy1992}. Hence, delay differential equations are often used while modelling natural population dynamics\cite{ding2017}. 
The presence of delay can bring severe change in the stability of a system like the destabilization of the system, large oscillations\cite{fowler1986}, the occurrence of chaotic behaviour\cite{sun2007,may1976} etc. 
In some instances, like population growth, physiology of breeding, organic insusceptible reactions, the growth rate of the species does not respond immediately. So, we need to confront a time delay\cite{kundu2018}. 
Hutchinson\cite{hutch1948} was the first to introduce a time delay in a logistic differential equation. 
Also, time delays are introduced to population models to incorporate the time duration for various biological processes such as gestation, incubation and maturation of a species\cite{upadhyay2016}. 
The discrete delay parameter, in the delayed logistic equation, can represent maturation time delay in a population species as suggested by Murray\cite{murray2002}, Alfifi\cite{alfifi2021} \& Fowler\cite{fowler1986}.

As the growth of a species relies on the food source, hence food habit of a species also plays a prominent role in population density. While several species rely upon a particular food source, generalist predators can feed on a wide range of food varieties and survive a severe change in environmental conditions, which distinguishes them from others. Hence, these can switch to a different food source when their favourite food is not abundant. North American raccoons is a good example of generalist predators\cite{natgeo}. They are found in a wide variety of environments like forests, mountains and cities. These omnivore species can feed on almost everything from fruit and nuts to insects, frogs, eggs etc.

The paper is organized as follows. Section \ref{secmodel} shows the formulation of the mathematical model and section \ref{secprelim} contains some preliminary results, which include the positivity and the boundedness of the system when no delay is involved. Section \ref{seclyap} contains the Lyapunov stability analysis of the model. Hopf-bifurcation is shown in section \ref{BA} in the presence of the delay. In section \ref{secest}, the length of the delay parameter is obtained to preserve stability. Section \ref{secnum} illustrates the numerical simulations. In section \ref{secind}, the stability of the controlled model is studied by the method of direct control. Finally, section \ref{seccon} represents the conclusions and discussions.

\section{The Model} \label{secmodel}
We propose a two-dimensional mathematical model which includes the following enlisted assumptions for modelling our predator-prey system. The assumptions are as follows:
\begin{enumerate}
\item[(a)] the environment has a carrying capacity,
\item[(b)] growth in prey species involves maturation delay,
\item[(c)] predator's functional response represents group defence in the prey species,
\item[(d)] the predator species reproduces sexually,
\item[(e)] the functional response of the generalist predator species is modelled by the modified Leslie-Gower scheme.
\end{enumerate}
Though in a real ecological system, there may be many preys and predators: among predators some are specialists and generalists, yet to capture the effect of a generalist predator on a particular prey we have taken a two-dimensional system where there is loss to predator species due to severe scarcity of the prey species.

Although it is witnessed that both monotonic\cite{ajraldi2011,braza2012,djilali2019} and non-monotonic  functional responses\cite{kumarkumari2021,andrews1968,sokol1981} are used to model group defence in a prey population, as observed from examples of musk oxen-wolves population\cite{tener1965}, Daphnia-algae population\cite{davidowicz1988} and several others\cite{holmes1972,andrews1968, edwards1970}, non-monotonic functional responses seem to be best suited for modelling group defence. These functional responses monotonically increase up to some critical value and then monotonically decrease.

Let us consider the simplified Monod-Haldane function, also known as simplified Holling type IV,
\begin{equation}\label{monod_haldane}
F(x)=\frac{mx}{x^2+c},
\end{equation}
which is non-monotonic and used for modelling group defence by Xiao and Ruan\cite{xiao2001}.

As discussed in Xiao and Ruan\cite{xiao2001}, a functional response, representing group defence, must have the following characteristics:
\begin{eqnarray}\label{group_defence_conds}
F:[0,\infty)\mapsto\mathbb{R},\; F \text{ is continuously differentiable},\nonumber\\
F(0)=0,\; F(x)>0 \;\&\; F(x)\leq M \text{ for } \forall x>0,\\
\text{ and } \exists \gamma>0\text{ s.t. } F'(x)\begin{cases}
>0, & x<\gamma,\\
<0, & x>\gamma.
\end{cases}
\nonumber
\end{eqnarray}
The above simplified Monod-Haldane function satisfies these characteristics. Let us define a functional response
\begin{equation}\label{generalized_monod_haldane}
F(x)=\frac{mx}{x^p+c},\quad p>1,
\end{equation}
then the function satisfies all the criteria in (\ref{group_defence_conds}) and hence is suitable for modelling group defence. Here, $F(x)$ is non-monotonic. We consider this response function for our predator-prey model.

In a predator-prey model, Leslie\cite{leslie1948} first proposed that the carrying capacity of the predators increases as the prey population increases. Moreover, the carrying capacity is a multiple of the prey population. Starting with a logistic form for a predator population with carrying capacity proportional to prey abundance, Aziz-Alaoui\cite{aziz2003} has derived the equation for a generalist predator where an extra parameter is added to the response term that measures the residual loss of predator species when prey population is scarce. This is called the modified Leslie-Gower scheme. Hence, the functional response of generalist predators is modelled by the modified Leslie-Gower functional response\cite{batabyal2020}.

Incorporating the above considerations, the predator-prey model can be represented as follows:
\begin{eqnarray}\label{sysnd}
\frac{dX}{dT}&=&RX\left(1-\frac{X}{K}\right)-\frac{MXY}{X^p+C},\nonumber\\
\frac{dY}{dT}&=&\left(D-\frac{E}{X+A}\right)Y^2.
\end{eqnarray}
where $X(T)$ and $Y(T)$ are respectively the densities of the prey and predator species at time $T$. The parameters used in the system (\ref{sysnd}) bear the following meanings:
\begin{align*}
R~= &\text{ Intrinsic growth rate of the prey species }X\\
 K~= &\text{ Environmental carrying capacity for the preys} \\
 M~= &\text{ Maximum predation rate}  \\
 C~= &\text{ The protection provided to the prey population by the environment} \\
 D~= &\text{ Reproduction rate of the generalist predator by sexual reproduction}  \\
 E~= &\text{ Maximum rate of death of predator population} \\
 A~= &\text{ Residual loss of predator species } Y \text{ due to severe scarcity of prey specis } X
\end{align*}

Newly born individuals in a species  do not instantaneously contribute to the growth of the species as they are not mature enough to participate in breeding. Assuming all individuals in a species have the same maturation period, we consider a discrete delay parameter $\tau$ to represent the maturation delay\cite{murray2002,alfifi2021,fowler1986} for the prey species. As a piece of experimental evidence, maturation delay in populations can be seen in round rays population. In Heller's round rays, individuals attain sexual maturity nearly at the age of 3.8 years. In the case of spinytail round rays, males and females mature approximately at 2.3 years, whereas Roger's round rays mature at an early age of 1 year\cite{morales2021}.\\

After introducing the delay term to the system, the model is given by,
\begin{eqnarray}\label{sysOriginal}
\frac{dX}{dT}&=&RX\left(1-\frac{X(T-\tau)}{K}\right)-\frac{MXY}{X^p+C},\nonumber\\
\frac{dY}{dT}&=&\left(D-\frac{E}{X+A}\right)Y^2,
\end{eqnarray}
with initial conditions:
\begin{eqnarray}\label{initialsOriginal}
X(\Theta)&=&\psi_1(\Theta)>0,\nonumber\\
Y(\Theta)&=&\psi_2(\Theta)>0,\quad \Theta \in [-\tau, 0); \psi_i(0)>0, i=1, 2.
\end{eqnarray}

It is assumed that the growth of the predator species is due to sexual reproduction. Hence the rate of change of population density will be directly proportional to the number of male and female individuals in the species. Hence the predator equation will have a growth term which is a scalar multiple of $Y^2$.

\noindent Now, let us consider, the  non-dimensional variables
\begin{center}
$x=\dfrac{X}{K}$, $y=\dfrac{Y}{K}$ and $t=RT$.
\end{center}

We use the above transformations on the system (\ref{sysOriginal}) to reduce the system to non-dimensionalised form and the reduced predator-prey model is given by,
\begin{eqnarray}\label{eq:sys}
\frac{dx}{dt}&=&x\left(1-x(t-\varrho)\right)-\frac{mxy}{x^p+c},\nonumber\\
\frac{dy}{dt}&=&\left(d-\frac{e}{x+a}\right)y^2,
\end{eqnarray}
subject to the initial conditions:
\begin{eqnarray}\label{eq:initials}
x(\theta)&=&\phi_1(\theta)>0,\nonumber\\
y(\theta)&=&\phi_2(\theta)>0,\quad \theta \in [-\varrho, 0); \phi_i(0)>0, i=1, 2,
\end{eqnarray}
where $\varrho$ is the logistic delay, and\\
\begin{center}
$\varrho=R\tau$, $m=\dfrac{M}{RK^{p-1}}$, $c=\dfrac{C}{K^p}$, $d=\dfrac{DK}{R}$, $e=\dfrac{E}{R}$ and $a=\dfrac{A}{K}$.
\end{center}

\section{Preliminary Results}\label{secprelim}
In this section, we shall discuss the positivity and boundedness of the solutions of the system (\ref{eq:sys}) and and boundedness of the solutions of the system in absence of the delay.

\noindent The system (\ref{eq:sys}) has one interior equilibrium point which is $(x^*,y^*)$, where
$$x^*=\frac{e}{d}-a~\text{ and }~y^*=\frac{1}{m}(1-x^*)(x^{*^p}+c).$$
\noindent So, the equilibrium point exists when $0<e/d-a<1$. As we are interested in the study of stability of the equilibrium point, so we assume this condition holds, if not mentioned else.
\subsection{Positivity}
\begin{lemma} The positive quadrant $int(R_+^2)$ is invariant for system (\ref{eq:sys}).
\end{lemma}
\begin{proof} Here we want to show that $\forall t\in [0, A)$ and $A\in (0, +\infty)$, $x(t)> 0$ and $y(t)>0$. Now solving the system (\ref{eq:sys}) we have 
\begin{eqnarray}\label{eq:positive}
x(t)&=& x(0)~exp\left[\int_{0}^{t}\left\{1-x(s-\varrho) -\frac{my}{x^p+c}\right\} ds \right],\nonumber\\
y(t)&=& y(0)~exp\left[\int_{0}^{t}\left(d-\frac{e}{x+a}\right)y ds \right].
\end{eqnarray}
$\therefore$ $x(t)>0$ and $y(t)>0$ as $x(0)>0$ and $y(0)>0$.\\
Thus we see that the system (\ref{eq:sys}) has positive solution with the positive initial condition given in (\ref{eq:initials}). Thus the positive quadrant $int(R_+^2)$ is invariant \cite{sA3}.
\end{proof}
\subsection{Boundedness of system (\ref{eq:sys}) when $\varrho=0$}
\begin{theorem}\label{theorem1}
The solutions of the system (\ref{eq:sys}) for $\varrho=0$, originating in $\mathbb{R}^2_+$, are bounded, provided the conditions $ \mu<\dfrac{dx_*}{x_*^p+c} $ and $ d-\dfrac{e}{x_*+a}<\dfrac{d}{m}\left(\dfrac{x_*}{y_*}\right)^2 $ hold, where $(x_*,y_*)$ is the interior equilibrium point and $\mu=\min \{m,e\}$.
\end{theorem}

\begin{proof}
\noindent The system (\ref{eq:sys}) with $\varrho=0$ is given by
\begin{eqnarray}\label{nodelay_sys}
\frac{dx}{dt} &=& x(1-x)-\frac{mxy}{x^p+c},\nonumber\\
\frac{dy}{dt} &=& \left(d-\frac{e}{x+a}\right)y^2.
\end{eqnarray}
\noindent From the system (\ref{nodelay_sys}), we can write,
\begin{eqnarray}\label{4}
 \frac{dx}{dt}  \le x(1-x)
\implies & x(t)  \le 1=M_1.
\end{eqnarray}
So, $x(t)$ is bounded for all $t$.

Now, we need to show that $y(t)$ is bounded in $\mathbb{R}^2_+$. From system (\ref{nodelay_sys}), we observe, there are two nullclines in $\mathbb{R}_+^2$: $$1-x-\frac{my}{x^p+c}=0 \text{ and } x=\frac{e}{d}-a.$$
Considering the signs of $x'(t)$ and $y'(t)$ in the regions formed by the nullclines, $\mathbb{R}_+^2$ can be divided into three regions. Let region-I is the region in $\mathbb{R}_+^2$ where $y'<0$, region-II is the region where $x',y'>0$ and the portion of $\mathbb{R}_+^2$ where $x'<0$ \& $y'>0$ is region-III. 

\begin{figure}[H]
\includegraphics[width=0.6\textwidth, height=6cm]{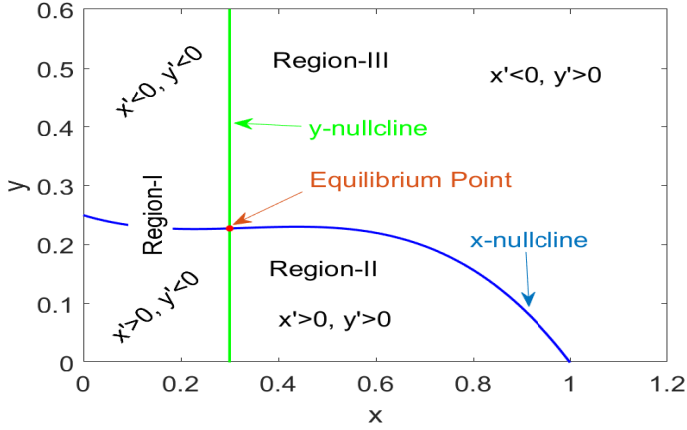}
\caption{The figure shows the nullclines of system (\ref{eq:sys}) in $\mathbb{R}_+^2$ region with $\varrho=0$ where the intersection of the nullclines is the positive equilibrium point $(x^*,y^*)$.}
\label{fig_nullclines}
\end{figure}

As $y'<0$ in region-I and $\mathbb{R}^2_+$ is an invariant space, so $y(t)$ will be bounded in region-I. $x(t)$ is bounded in $\mathbb{R}_+^2$, so $(1-x)(x^p+c)$ is bounded by some positive number, say $\delta$. As $x'>0$ in region-II, so we have,
\begin{eqnarray*}
1-x-\frac{my}{x^p+c}>0\implies y<\frac{\delta}{m}.
\end{eqnarray*}
Hence, $y(t)$ is bounded in region-II.

\noindent In region-III, $x'<0$ and $y'>0$. Let us define a quantity $\sigma(t)$ such that,
$$\sigma(t)=\frac{d}{m} x(t)+y(t).$$
Then from the expression of $\sigma(t)$ and its time derivative $\sigma'(t)$, we can obtain,
\begin{equation*}
\sigma'+\mu \sigma = \frac{d}{m}(1+\mu)\; x(t)+A_1y+A_2y^2,
\end{equation*}
where, $A_1=\mu-\dfrac{dx}{x^p+c},\text{ and } A_2=d-\dfrac{e}{x+a}-\frac{d}{m} \left(\frac{x}{y}\right)^2$ and $\mu=\min\{m,e\}$.

\noindent Now, if we choose, $ \mu<\dfrac{dx_*}{x_*^p+c} $ and $ d-\dfrac{e}{x_*+a}<\dfrac{d}{m}\left(\dfrac{x_*}{y_*}\right)^2 $, then $A_1<0$ and $A_2<0$, so
\begin{equation*}
\sigma'+\mu \sigma \le \frac{d}{m}(1+\mu)x(t)\le \frac{d}{m}(1+\mu)=M_2\mu>0.
\end{equation*}
Following \cite{kaviya2021}, by applying Gronwall's inequality, as $t\rightarrow\infty$,
\begin{equation}\label{5}
0<\sigma(t)\leq M_2 \quad\text{i.e,}\quad 0<y(t)\leq M_2,
\end{equation}
where $M_2 = \frac{d}{m\mu} (1+\mu)$. Hence, $y(t)$ has upper bound in region III.\\
Hence from (\ref{4}) $\&$ (\ref{5}), we have 
$$0< x(t)\leq M_1,$$
$$0< y(t)\leq M_2,$$
with positive initial condition, i.e. $x(0)>0$, $y(0)>0$.\\
As $x(t)$ is bounded in $\mathbb{R}^2_+$ and $y(t)$ is bounded in all three regions of $\mathbb{R}^2_+$, hence, the system (\ref{nodelay_sys}) is bounded in $\mathbb{R}^2_+$, when the given conditions hold.
\end{proof}

\noindent In the presence of delay, local stability of system (\ref{eq:sys}) is presented below.
\section{Lyapunov Stability Analysis}\label{seclyap}
Here,we have studied the local stability of (\ref{eq:sys}) by using a suitable Lyapunov functional as done in \cite{sA1}. The system (6) has positive equilibrium point $E=(x^*, y^*)$. Introducing the new set of variables $\bar{x}=x-x^*$ and $\bar{y}=y-y^*$ in system (\ref{eq:sys}),
we get the following linearised system as:
\begin{eqnarray}\label{eq:23}
\frac{du}{dt}&=&a_{11}\bar{x}+a_{12}\bar{y},\nonumber\\
\frac{d\bar{y}}{dt}&=&a_{22}\bar{x},
\end{eqnarray}
where, $u=\bar{x}-x^*\int_{t-\varrho}^{t}\bar{x}(s)ds$, \quad $a_{11}=\frac{mpx^{*^p}y^*}{(x^{*^p}+c)^2}$, \quad $a_{12}=-\frac{mx^*}{x^{*^p}+c}$ and $a_{22}=\frac{d{y^*}^2}{e}$.\\
Now following the steps as in \cite{sA1}, we shall check the stability of the system by assuming a suitable Lyapunov function $w(v)(t)$ as follows:
\begin{eqnarray}\label{eq:24}
w(v)(t)=k_1w_1(v)(t)+k_2w_2(v)(t)+k_3w_3(v)(t),
\end{eqnarray}
where, 
\begin{eqnarray}
w_1(v)(t)&=&u^2+x^*(a_{11}+a_{12})\int_{t-\varrho}^{t}\int_{s}^{t}\bar{x}^2(l) dl ds, \nonumber\\
w_2(v)(t)&=&\bar{y}^2,\nonumber\\
w_3(v)(t)&=&u\bar{y}+\frac{a_{22}x^*}{2}\int_{t-\varrho}^{t}\int_{s}^{t}\bar{x}^2(l) dl ds, \nonumber
\end{eqnarray}
and 
\begin{eqnarray}
k_1&=&2a_{12}\varrho-a_{12},\nonumber\\
k_2&=&a_{22}\left(1+\frac{\varrho}{2}\right) ,\nonumber\\
k_3&=&(a_{11}+a_{22})\varrho. \nonumber
\end{eqnarray}
As all the parameters are assumed positive so, $k_1>0$, $k_2>0$, $k_3>0$ and $w(v)(t)>0$. Taking the derivative of (\ref{eq:24}), we get
\begin{eqnarray}\label{eq:25}
\frac{d}{dt}w(v)(t)\leq \Lambda_1 \bar{x}^2+\Lambda_2 \bar{y}^2,
\end{eqnarray}
where, 
\begin{eqnarray}
\Lambda_1 &=&k_1\left\{2a_{11}-a_{11}x^*-x^*(a_{11}+a_{12})\varrho \right\}+k_3\frac{a_{22}x^*}{2}(\varrho-1) ,\nonumber\\
\Lambda_2 &=&a_{12}(k_3-k_1x^*).\nonumber
\end{eqnarray}
Further simplifying the conditions $\Lambda_1<0$, $\Lambda_2<0$, we have
\begin{eqnarray}\label{delay_lyap_conds}
\varrho>\pi_o\quad \text{and}\quad \pi_1\varrho^2+\pi_2\varrho+\pi_3>0,
\end{eqnarray}
where expressions for $\pi_i$'s, for $i=0,1,2$ \& $3$, are  given by,
$$\pi_0=\dfrac{a_{12}x^*}{2a_{12}x^*-a_{11}-a_{22}},\quad \pi_1=x^*\left(2a_{11}a_{12}+2a_{12}^2-\frac{1}{2}a_{11}a_{22}-\frac{1}{2}a_{22}^2\right),$$
$$\pi_2=a_{11}a_{12}x^*-a_{12}^2x^*-4a_{11}a_{12}, \quad \pi_3=a_{11}a_{12}(2-x^*)-\frac{1}{2}a_{22}x^*(a_{11}+a_{22}).$$
\begin{theorem}\label{th:2}
If the value of the delay $\varrho$ satisfies the conditions in (\ref{delay_lyap_conds}) then the interior equilibrium point $E(x^*, y^*)$ of (\ref{eq:sys}) is locally asymptotically stable.
\end{theorem}
\begin{proof}
Following the steps as done in \cite{sA1}, one can easily prove Theorem \ref{th:2}.
\end{proof}
\section{Bifurcation analysis}\label{BA} 
In this section we shall discuss the Hopf-bifurcation for the system (\ref{eq:sys}). Let the equilibrium point be $E(x^*, y^*)$. Letting $\bar{x}=x-x^*$ and $\bar{y}=y-y^*$ and substituting into Eq. (\ref{eq:sys}) we get the linearised form as:
\begin{eqnarray}\label{BA:1}
\frac{d\bar{x}}{dt}&=&a_{11}\bar{x}-x^* e^{-\lambda \varrho}\bar{x}+a_{12}\bar{y},\nonumber\\
\frac{d\bar{y}}{dt}&=&a_{22}\bar{x},
\end{eqnarray}
where the expressions for $a_{11}$, $a_{12}$ and $a_{22}$ as mentioned in section \ref{seclyap}. The characteristic equation for the system (\ref{BA:1}) as given below
\begin{eqnarray}\label{BA:2}
\left\{\lambda^2-a_{11}\lambda+a_{11}a_{22}\right\}+e^{-\lambda\varrho}\left\{x^*\lambda-x^*a_{12}\right\}=0.
\end{eqnarray}
Let $\lambda=i\omega(>0)$ then from (\ref{BA:2}) separating real and
imaginary part we get
\begin{eqnarray}\label{BA:3}
x^*\omega \sin(\omega\varrho)-x^* a_{12}\cos(\omega\varrho)&=&\omega^2-a_{11}a_{22},\nonumber\\
x^* a_{12}\sin(\omega\varrho)+x^*\omega \cos(\omega\varrho)&=&\omega(a_{11}+a_{22}),
\end{eqnarray}
which gives
\begin{eqnarray}\label{BA:4}
\omega^4+\omega^2\left(a_{11}^2+a_{22}^2-x^{*^2}\right)+a_{12}^2(a_{11}^2-x^{*^2})=0.
\end{eqnarray}
The equation (\ref{BA:4}) will have positive root if 
\begin{eqnarray}\label{BA:5}
a_{11}^2>x^{*^2}.
\end{eqnarray}
Now we eliminate $\sin(\omega\varrho)$ from (\ref{BA:3}) we have 
\begin{eqnarray}\label{BA:6}
\cos(\omega\varrho)=\frac{\omega^2(a_{11}+2a_{22})-a_{11}a_{22}^2}{x^*(\omega^2-a_{12}^2)}.
\end{eqnarray}
Let, $\omega=\omega_0$ be a positive root of (\ref{BA:4}), then
\begin{eqnarray}\label{BA:7}
\varrho_{n}^{*}=\frac{1}{\omega_0}\left[\arccos\frac{\omega_0^2(a_{11}+2a_{22})-a_{11}a_{22}^2}{x^*(\omega_0^2-a_{12}^2)}+2n\pi \right],\quad n=0, 1, 2,....
\end{eqnarray}
We define the function $\theta(\varrho)\in [0, 2\pi)$, such that $\cos \theta(\varrho)$ is given by the right hand side of (\ref{BA:7}). Then solving $$S_n(\varrho)=\varrho-\varrho_{n}^{*}$$ we get the $\varrho$, at which stability switching occurs. If $\lambda(\varrho)$ be the root of the characteristic equation (\ref{BA:2}) satisfying $\Real \lambda(\varrho_{n}^{*})=0$ and $\Imag \lambda(\varrho_{n}^{*})=\omega_0$, we get $$\left(\frac{d}{d\varrho}Re\lambda\right)_{\varrho=\varrho_0^*}\neq 0.$$ \\
Differentiating both sides of the equation (\ref{BA:2}) w.r.t. $\varrho$,
\begin{equation}
\left(\frac{d\lambda}{d\varrho}\right)^{-1}=-\frac{2\lambda-a_{11}+x^*\e^{-\lambda\varrho}}{\lambda^3-a_{11}\lambda^2+a_{11}a_{22}\lambda}-\frac{\varrho}{\lambda}.
\end{equation}
Putting $\lambda=i\omega_0$ and $\varrho=\varrho_0^*$,
\begin{equation}
\text{Sign}\left[\Real\left(\frac{d\lambda}{d\varrho}\right)^{-1}\right]_{\lambda=i\omega_0,\;\varrho=\varrho_0^*}=\frac{1}{\omega_0^2}\text{Sign}\left[\frac{2w_0^2}{w_0^2-a_{11}a{22}}\right].
\end{equation}
Now if $\omega_0^2-a_{11}a_{22}>0$, then $\text{Sign}\left[\Real\left(\frac{d\lambda}{d\varrho}\right)^{-1}\right]_{\lambda=i\omega_0}>0$, i.e,
\begin{equation}
\frac{d}{d\varrho}(\Real \lambda)>0.
\end{equation}
Thus, we see that the transversality condition holds and hence system (\ref{eq:sys}) undergoes a Hopf-bifurcation for $\varrho=\varrho_0^*$.
l\section{Estimation of the length of delay to preserve stability}\label{secest}
Here, we consider the system (\ref{eq:sys}) and the set of all continuous real functions defined on the domain $[-\varrho,\infty)$ with the positive initial conditions defined in (\ref{eq:initials}) on $[-\varrho,0]$. Linearizing the system (\ref{eq:sys}) around the origin after substituting $v_1=x-x^*$ and $v_2=y-y^*$, we get,
\begin{eqnarray}\label{dl:linsys}
\frac{dv_1}{dt} &=& -b_{11}v_1(t-\varrho)+b_{12}v_1-b_2v_2 \nonumber\\
\frac{dv_2}{dt} &=& b_3v_1
\end{eqnarray}
where, $$b_{11}=x^*,\quad b_{12}=\frac{mpx^{*^p}y^*}{(x^{*^p}+c)^2}, \quad b_2=\frac{mx^*}{{x^*}^p+c},\quad b_3=\frac{dy^{*^2}}{e}.$$
Applying Laplace transformation to the system (\ref{dl:linsys}), we get,
\begin{eqnarray}\label{dl:laplace}
(s+b_{11}e^{-\varrho s}+b_{12})\mathcal{V}_1 &=& v_1(0)-b_{11}\;K(s)\;e^{-\varrho s}-b_2\mathcal{V}_2\nonumber\\
s\mathcal{V}_2 &=& v_2(0)+b_3\mathcal{V}_1
\end{eqnarray}
where, $K(s)=\int_{-\varrho}^0 e^{-st}v_1(t)\; dt$, and $\mathcal{V}_1$ and $\mathcal{V}_2$ are Laplace transforms of $v_1$ and $v_2$ respectively. Following along the lines of \cite{sA3} and by using Nyquist criterion \cite{sA3}, we can have the conditions for local asymptotic stability for the interior equilibrium point of the system (\ref{eq:sys}) are given by
\begin{equation}\label{realh}
Re\;H(i\eta_0)=0
\end{equation}
\begin{equation}\label{imagh}
Im\;H(i\eta_0)>0
\end{equation}
where $H(s)=s^2+p_1s+p_2+e^{-s\varrho}(q_1s+q_2)$ is the characteristic equation and $\eta_0$ is the smallest positive root of $Re\;H(i\eta_0)=0$.
Hence, equations (\ref{realh}) and (\ref{imagh}) can be rewritten as
\begin{equation}\label{realeq}
q_2\cos{\eta_0\varrho}+q_1\eta_0\sin{\eta_0\varrho}=\eta_0^2-p_2
\end{equation}
\begin{equation}\label{imaggr}
q_1\eta_0\cos{\eta_0\varrho}-q_2\sin{\eta_0\varrho}>-p_1\eta_0.
\end{equation}

The stability of the system is guaranteed if the conditions (\ref{realeq}) and (\ref{imaggr}) hold simultaneously. We first need to find upper bound $\eta_+$ of $\eta_0$ independent of $\varrho$ such that the system will be stable for all such $\eta_0$ that lie in $[0,\eta_+]$. We will use this $\eta_+$ to estimate the delay length $\varrho$.\\
From equation (\ref{realeq}), as $|\sin{\eta_0\varrho}|\le 1$ \& $|\cos{\eta_0\varrho}|\le 1$,
\begin{equation}\label{etaeq}
\eta_0^2-p_2\le|q_2|+|q_1|\eta_0.
\end{equation}
Hence, from (\ref{etaeq}), if
\begin{equation}\label{eta+}
\eta_+\le\frac{1}{2}\left[|q_1|+\sqrt{q_1^2+4(p_2-|q_2|)}\right]
\end{equation}
then we have $\eta_0\le\eta_+$.\\
From equations (\ref{realeq}) \& (\ref{imaggr}),
\begin{equation}
\frac{q_1\eta_0}{q_2}\left[\eta_0^2-p_2-q_1\eta_0\sin{\eta_0\varrho}\right]-q_2\sin{\eta_0\varrho}>-p_1\eta_0
\end{equation}
simplifying which we get,
\begin{equation}\label{rho1st}
\left[\frac{q_1^2\eta_0^2}{q_2}+q_2\right]\sin{\eta_0\varrho}<\frac{q_1\eta_0}{q_2}(\eta_0^2-p_2)+p_1\eta_0.
\end{equation}
Now, we can have,
\begin{equation}\label{rhoeq1}
\left[\frac{q_1^2\eta_0^2}{q_2}+q_2\right]\sin{\eta_0\varrho}\le\frac{1}{|q_2|}(q_1^2\eta_+^2+q_2^2)\eta_+\varrho
\end{equation}
and
\begin{equation}\label{rhoeq2}
\frac{q_1\eta_0}{q_2}(\eta_0^2-p_2)+p_1\eta_0\le\frac{q_1(\eta_+^2-p_2)+|p_1q_2|}{|q_2|}\eta_+.
\end{equation}
Now, from  (\ref{rho1st}), (\ref{rhoeq1}) and (\ref{rhoeq2}),
\begin{equation}\label{delaylimit}
0\le\varrho<\varrho_+
\end{equation}
where
\begin{equation}\label{delayupper}
\varrho_+=\frac{q_1(\eta_+^2-p_2)+|p_1q_2|}{q_1^2\eta_+^2+q_2^2}.
\end{equation}
Thus when (\ref{delaylimit}) holds for $\varrho$, the system preserves local asymptotic stability.
\section{Numerical Simulation} \label{secnum}
In the previous section, the conditions for stability and for occurrence of Hopf-bifurcation in system (\ref{eq:sys}) have been derived analytically. Here, numerical computations are performed to understand our results obtained in previous sections by choosing suitable values of the parameters. For different values of delays, we have obtained different scenarios with $E(x^*, y^*)$ as the interior equilibrium point. The values of the parameters are taken as: $m=1.2, p=2, c=0.3, d=0.5, e = 0.2, a=0.2$. In Fig. \ref{fig:s1}, a stable solution for the species (\ref{eq:sys}) has been plotted by taking the delay as bifurcation parameter. The interior equilibrium point $E(x^*, y^*)$ is seen to be stable for less values of the delay i.e., $\varrho<\varrho_{0}^*=1.125$. At the critical value of delay ($\varrho=\varrho_{0}^*$) we get a stable periodic solution where the Hopf-bifurcation occurs and there we get a stable limit cycle around $E(x^*, y^*)$ (Fig. \ref{fig:s2}). For large values of delay ($\varrho>\varrho_{0}^*$) the system becomes unstable (Fig. \ref{fig:s3}).

\begin{figure}[H]
\includegraphics[width=5.2in, height=2.7in]{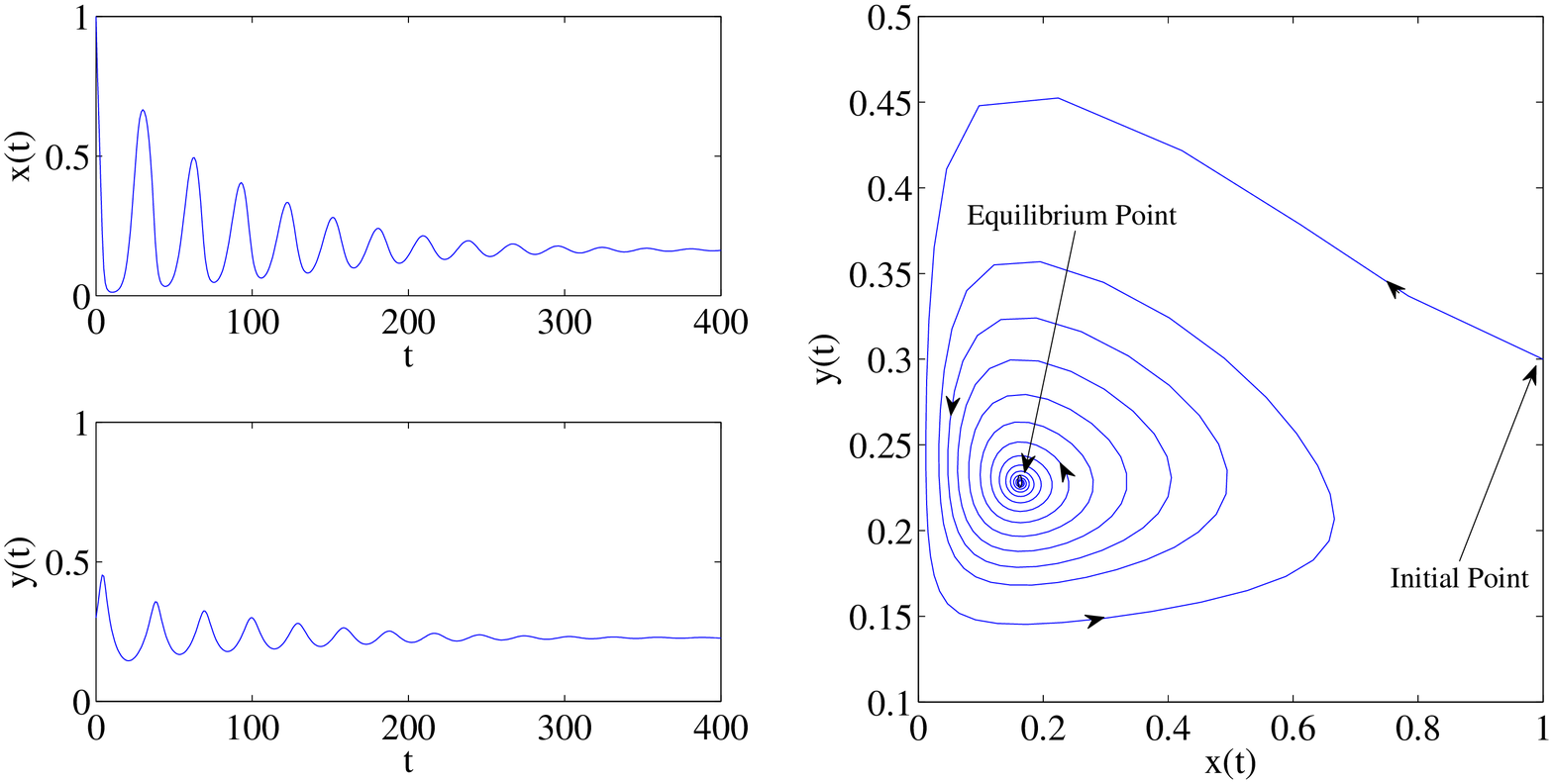}
\caption{Time series solution of the system (\ref{eq:sys}) has been plotted by taking $\varrho<\varrho_0^*$ (Left side). It shows that the solutions of the system (3) are stable. Right side of this figure shows the phase diagram of the system (3), which indicates that the equilibrium $E(x^*,y^*)$  is a stable equilibrium point. The values of the parameters are taken as: $m = 1.2$, $p = 2$, $c = 0.3$, $d=0.5$, $e = 0.2$, $a = 0.2$. The initial point is $(1,0.3)$.}
\label{fig:s1}       
\end{figure}
\begin{figure}[H]
\includegraphics[width=5.2in, height=2.7in]{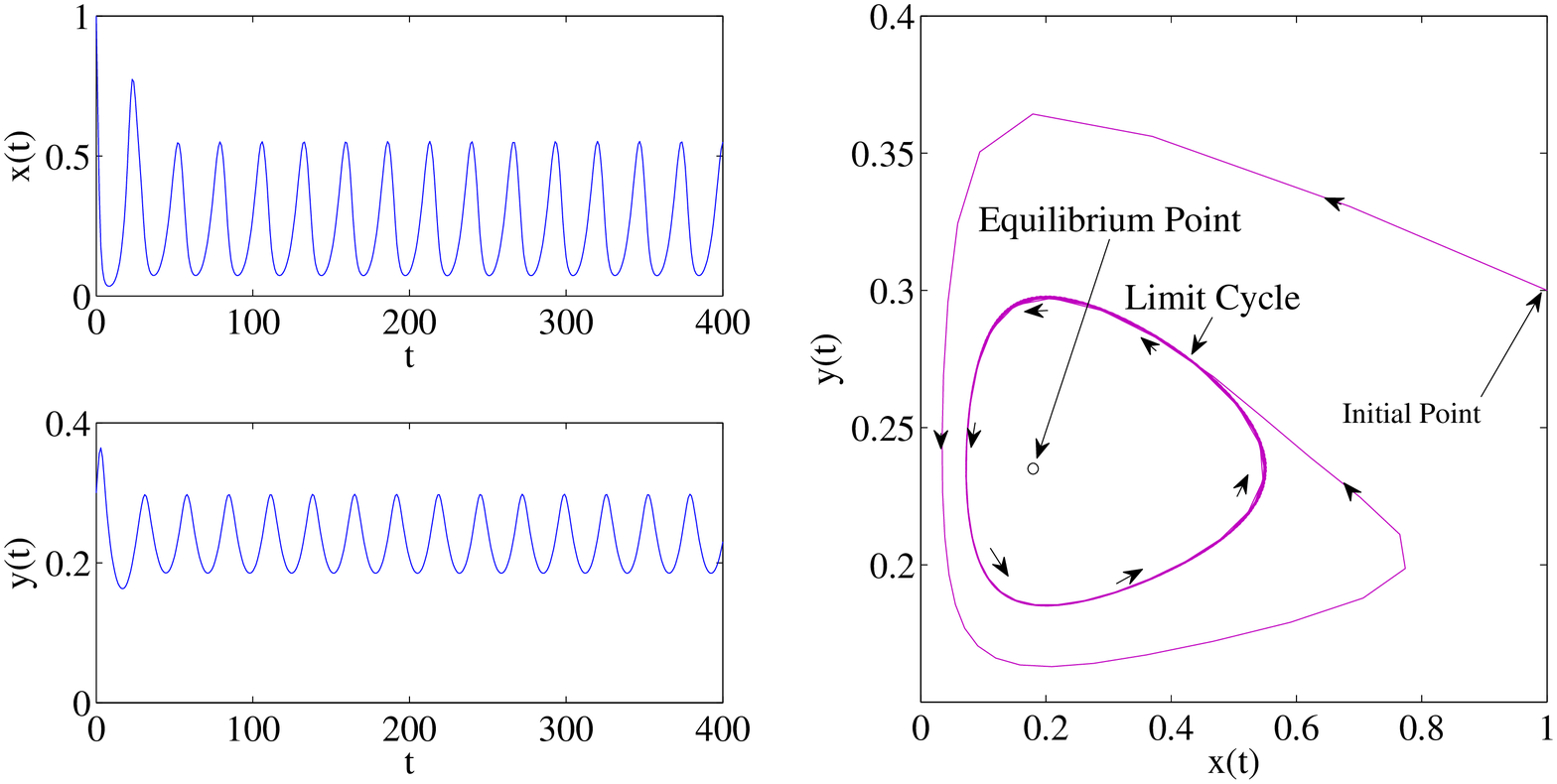}
\caption{Time series solution of the system (3) has been plotted by taking $\varrho=\varrho_0^*$ (Left side). It shows that the solutions are periodic. Right side of this figure shows the phase diagram of the system (\ref{eq:sys}), which indicates a stable limit cycle around the equilibrium $E(x^*,y^* )$, when other parameters are same as that of Fig. \ref{fig:s1}. }
\label{fig:s2}       
\end{figure}
\begin{figure}[H]
\includegraphics[width=5.2in, height=2.7in]{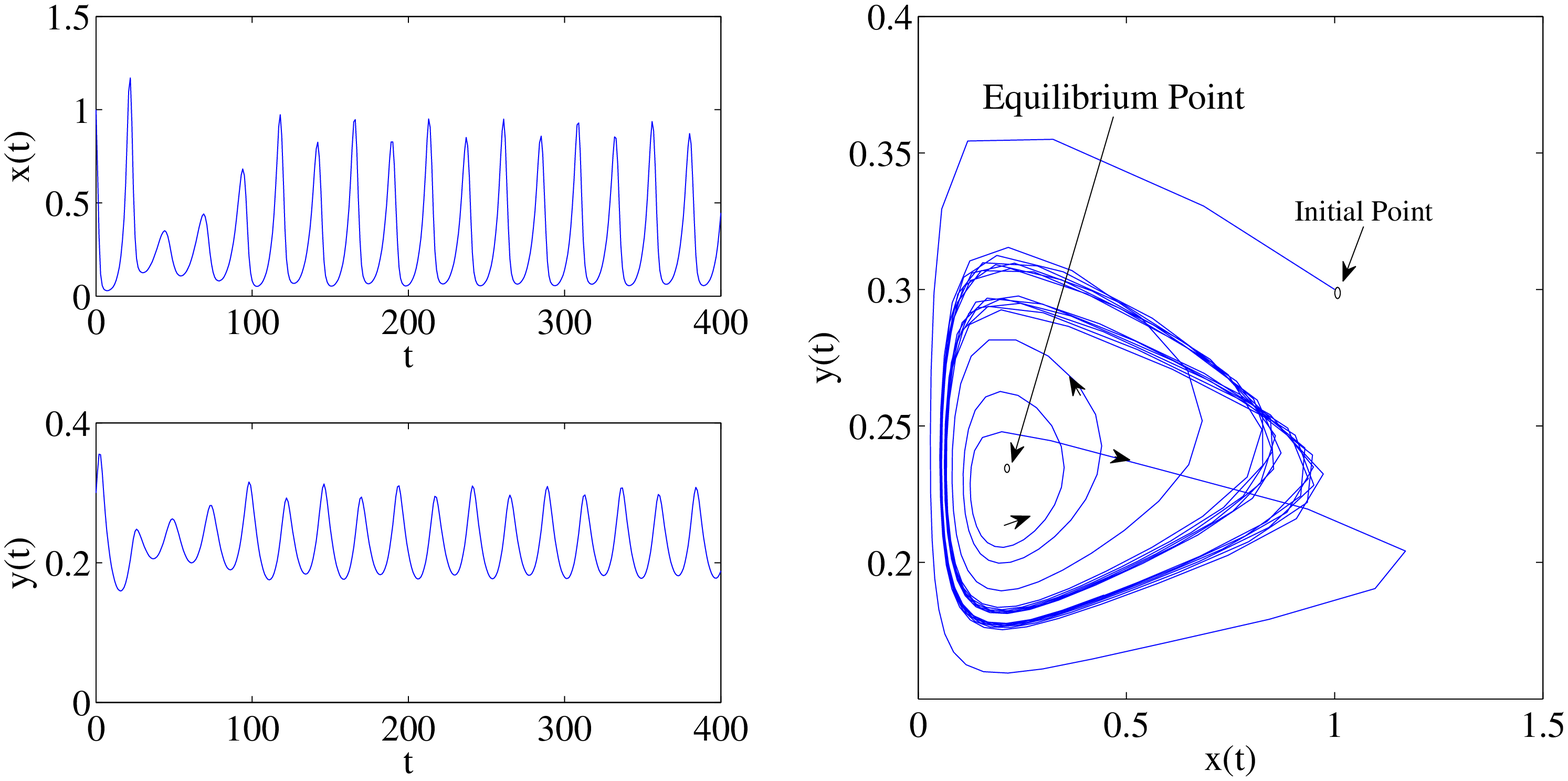}
\caption{This figure has been plotted by taking $\varrho>\varrho_0^*$. Left side of this figure shows the time series solution for species and the corresponding phase diagram has been plotted in the right side of this figure. It depicts unstable solution of the system (3) when $\varrho>\varrho_0^*$. Other parameters are same as that of Fig. \ref{fig:s1}. }
\label{fig:s3}       
\end{figure}
Occurrence of Hopf bifurcation with respect to parameters $c$, $d$, $e$ and $a$ are illustrated in the figures below for zero delay in system (\ref{eq:sys}). The  value of the parameters are taken as $m=1.2$, $p=2$, $c=0.3$, $d=0.4$, $e=0.2$ and $a=0.2$. Each time when performing bifurcation on a single parameter, other parameters are taken constant whose values are taken from the above parameter set.

Figure \ref{fig_c} illustrates the situation when the system (\ref{eq:sys}) undergoes a Hopf bifurcation as the parameter $c$ changes value in the absence of delay. The bifurcation is supercritical and occurs at $c=c_H$ for $m=1.2$, $p=2$, $d=0.4$, $e=0.2$ and $a=0.2$.
\begin{figure}[H]
 \subfloat[Bifurcation diagram\label{subfig-4}]{%
  \includegraphics[width=0.45\textwidth, height=6cm]{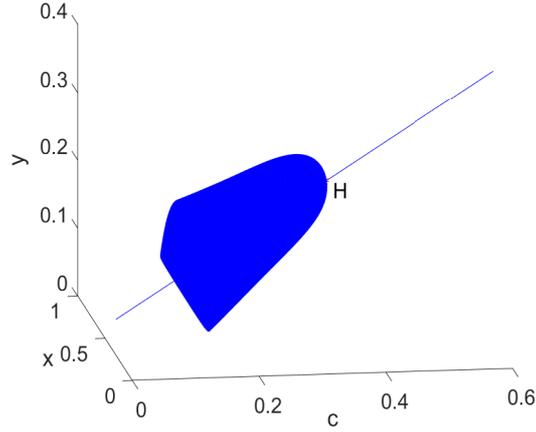}}\\
 \subfloat[Time series \& phase portrait when $c=0.4$\label{subfig-5}]{%
  \includegraphics[width=0.25\textwidth, height=3.5cm]{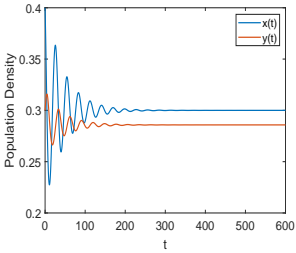}
  \includegraphics[width=0.25\textwidth, height=3.5cm]{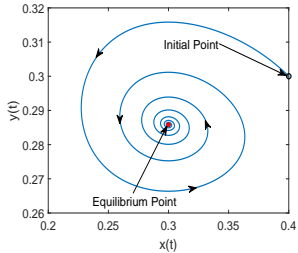}}
 \subfloat[Time series \& phase portrait when $c=0.25$\label{subfig-6}]{%
  \includegraphics[width=0.25\textwidth, height=3.5cm]{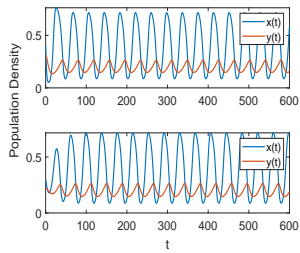}
  \includegraphics[width=0.25\textwidth, height=3.5cm]{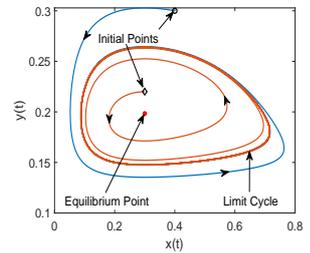}}
 \caption{{\bf(a)} The positive equilibrium of system (\ref{eq:sys}) undergoes a Hopf bifurcation for $c_H=0.3300005$ when $\varrho=0$. {\bf(b)} The equilibrium state $E$ is stable for $c=0.4>c_H$, where the initial point of the simulation is $\circ(0.4,0.3)$. {\bf(c)} For $c=0.25<c_H$, there exists a stable limit cycle around the unstable equilibrium point, where the initial points of the simulation are $\circ(0.4,0.3)$ \& $\diamond(0.3,0.22)$.}
 \label{fig_c}
\end{figure}
When $\varrho=0$, system (\ref{eq:sys}) undergoes Hopf bifurcation at E for $d=d_{H1}, d_{H2}, \& d_{H3}$ with other parameters as $m=1.2$, $p=2$, $c=0.3$, $e=0.2$ and $a=0.2$. The parameter $d$ can take positive values less than $1$ for $E$ to be biologically feasible. LPC-bifurcation is observed for $d=d_{LPC}$, where two limit cycles of different periods collide and vanish as $d$ crosses $d_{LPC}$. Here, LPC stands for "Limit Point of Cycles".
\begin{figure}[H]
 \subfloat[Bifurcation diagram\label{subfig-7}]{%
  \includegraphics[width=0.45\textwidth, height=6cm]{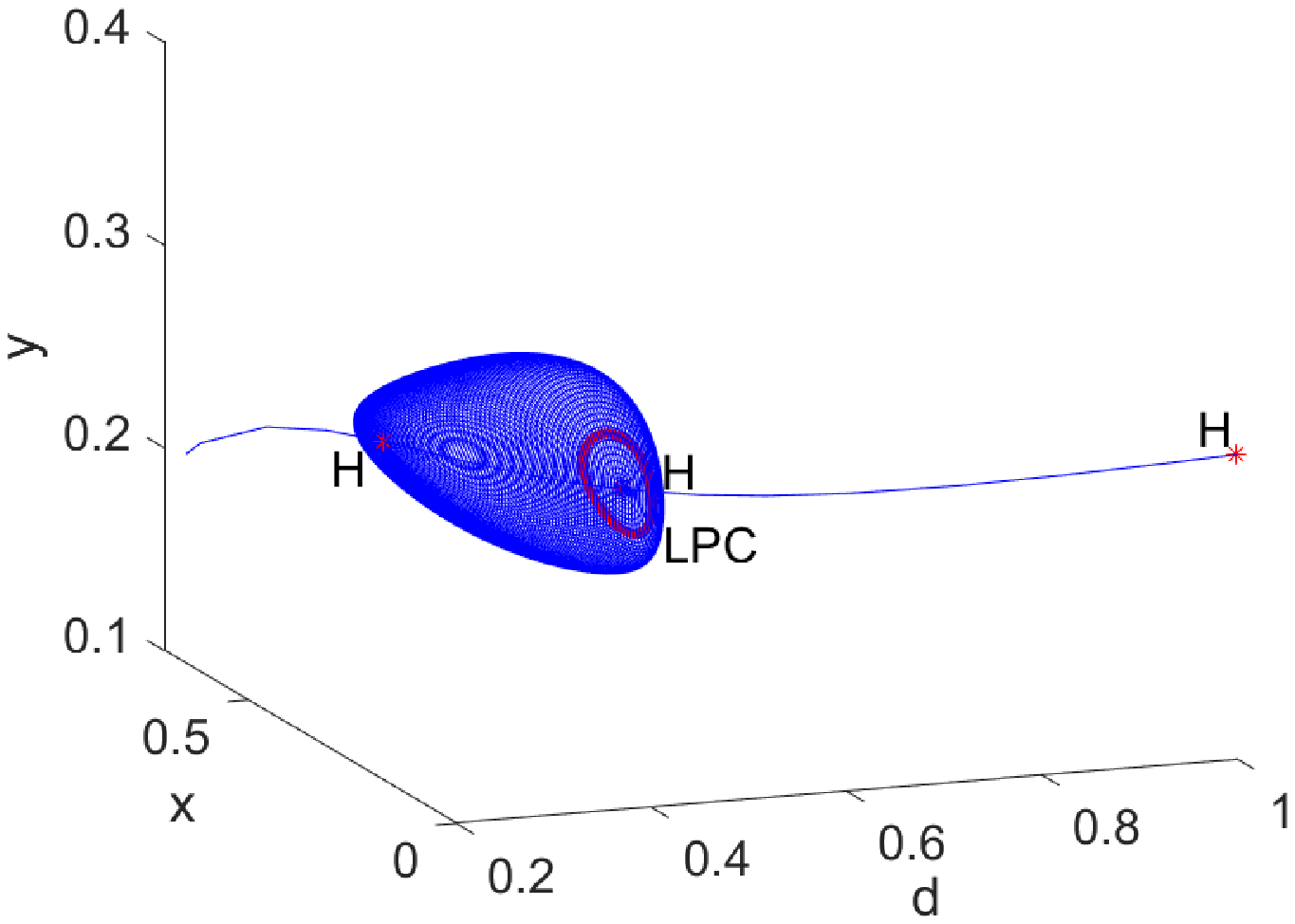}}
 \subfloat[Zoomed in version of (a) near the LPC\label{subfig-7}]{%
  \includegraphics[width=0.45\textwidth, height=6cm]{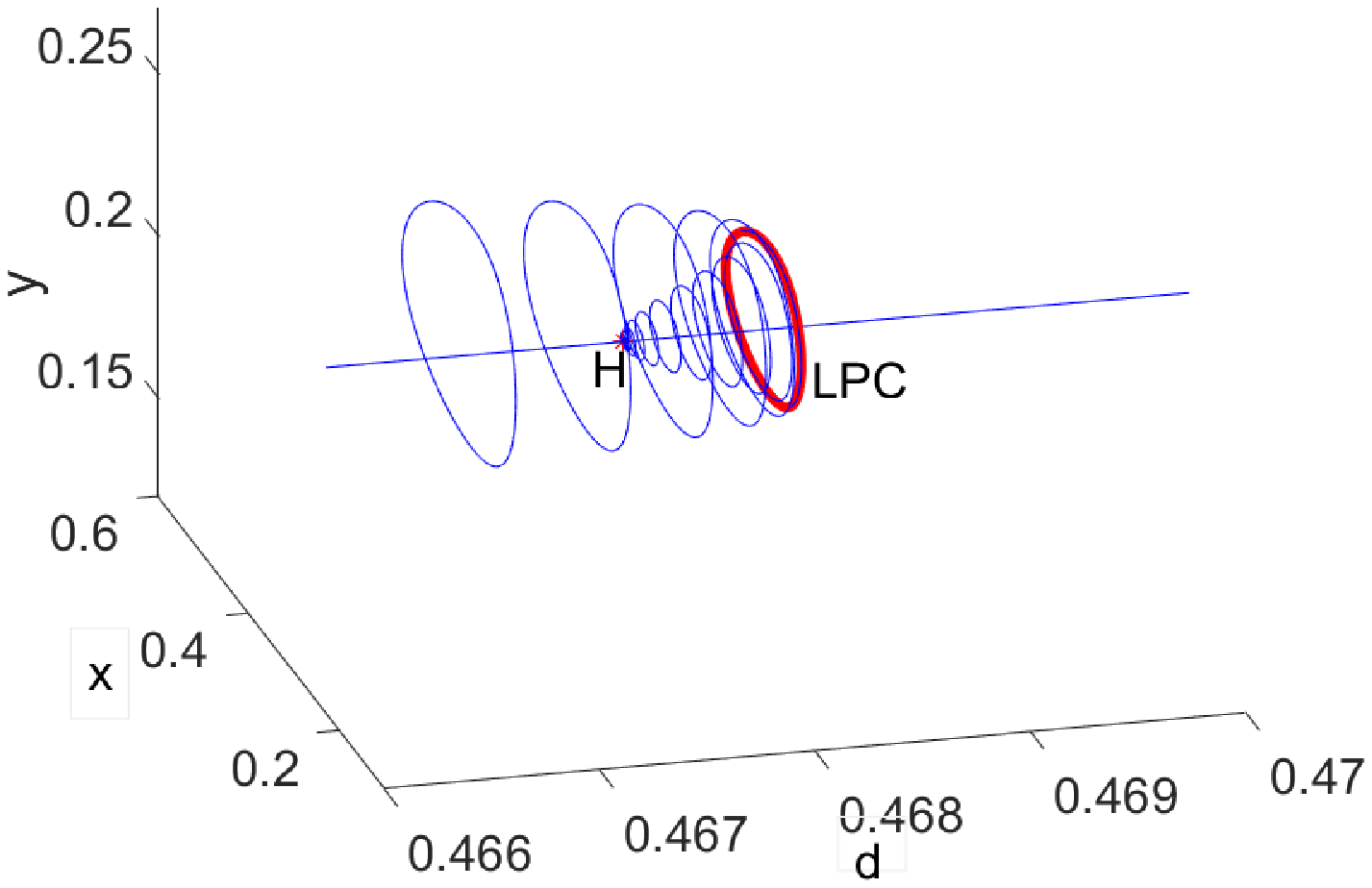}}\\
 \subfloat[Time series \& phase portrait when $d=0.28$\label{subfig-8}]{%
  \includegraphics[width=0.25\textwidth, height=3.5cm]{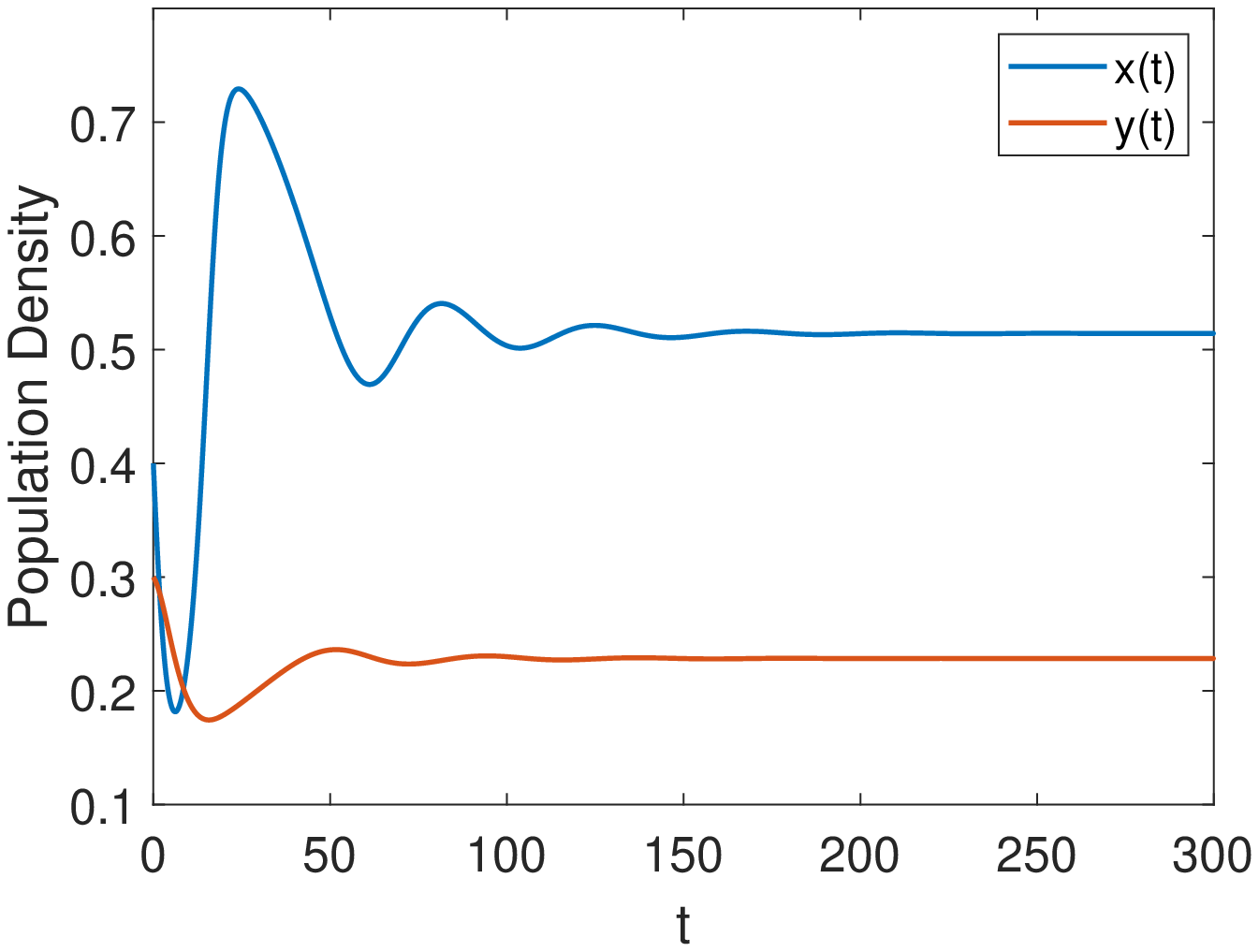}
  \includegraphics[width=0.25\textwidth, height=3.5cm]{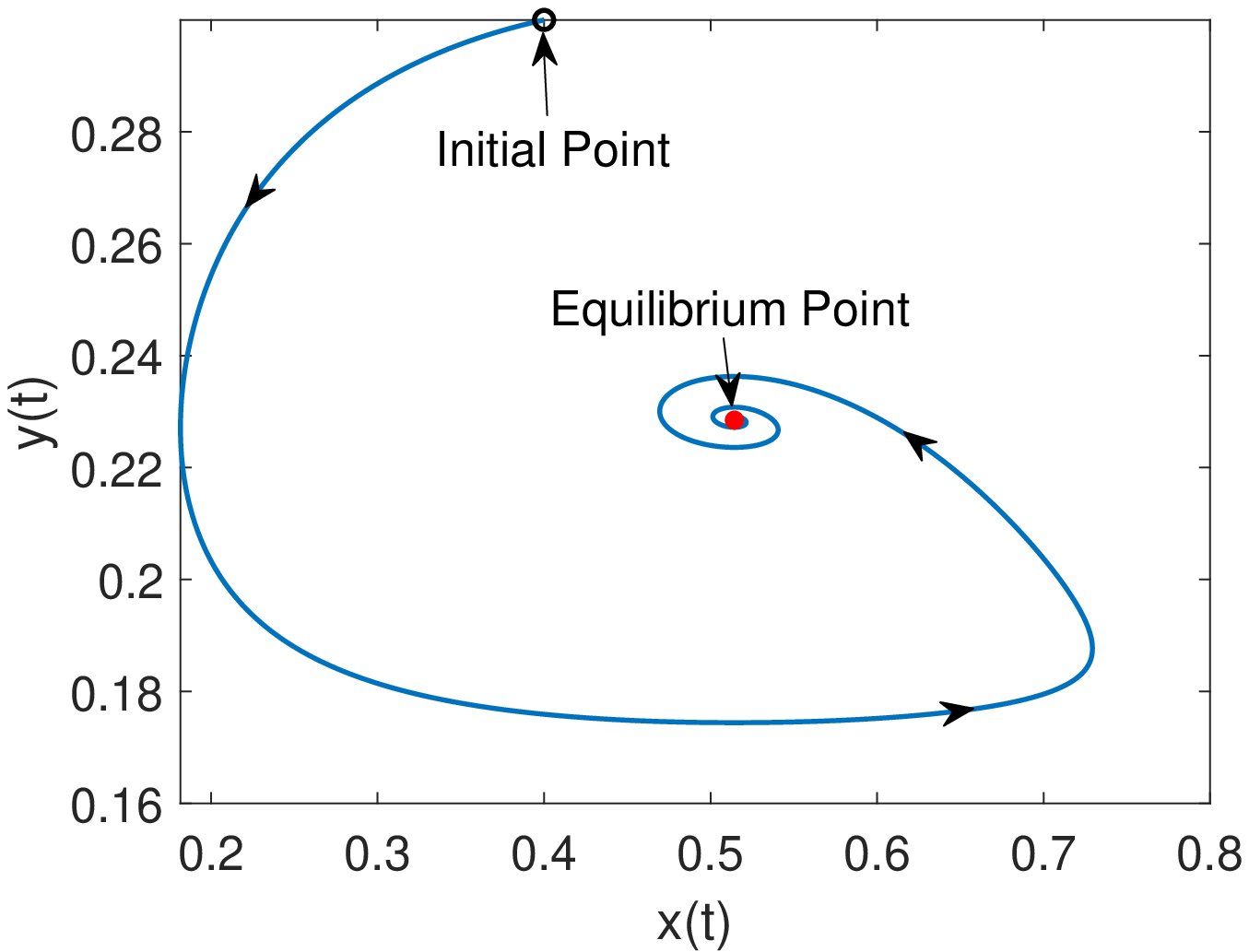}}
 \subfloat[Time series \& phase portrait when $d=0.35$\label{subfig-8}]{%
  \includegraphics[width=0.25\textwidth, height=3.5cm]{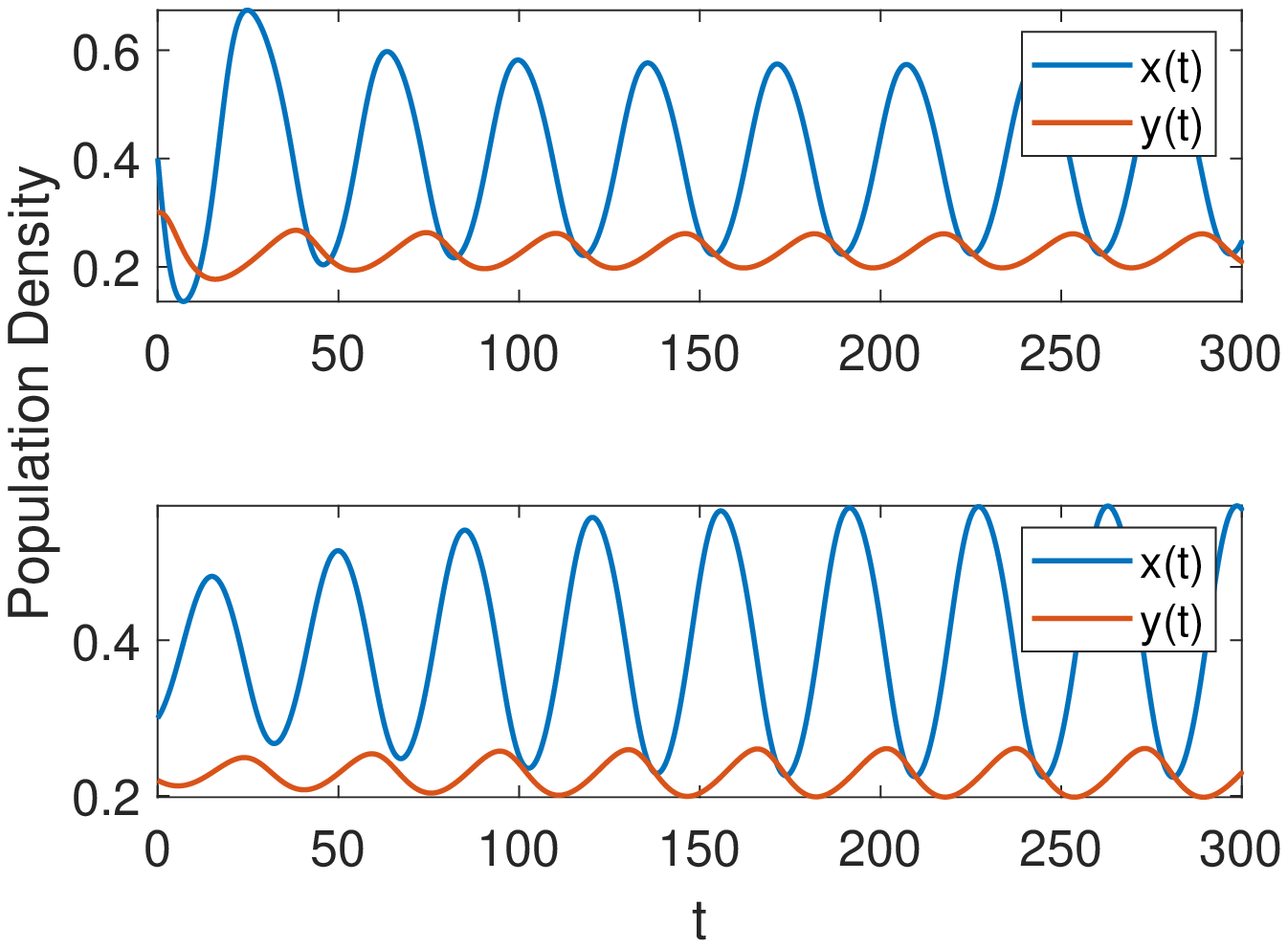}
  \includegraphics[width=0.25\textwidth, height=3.5cm]{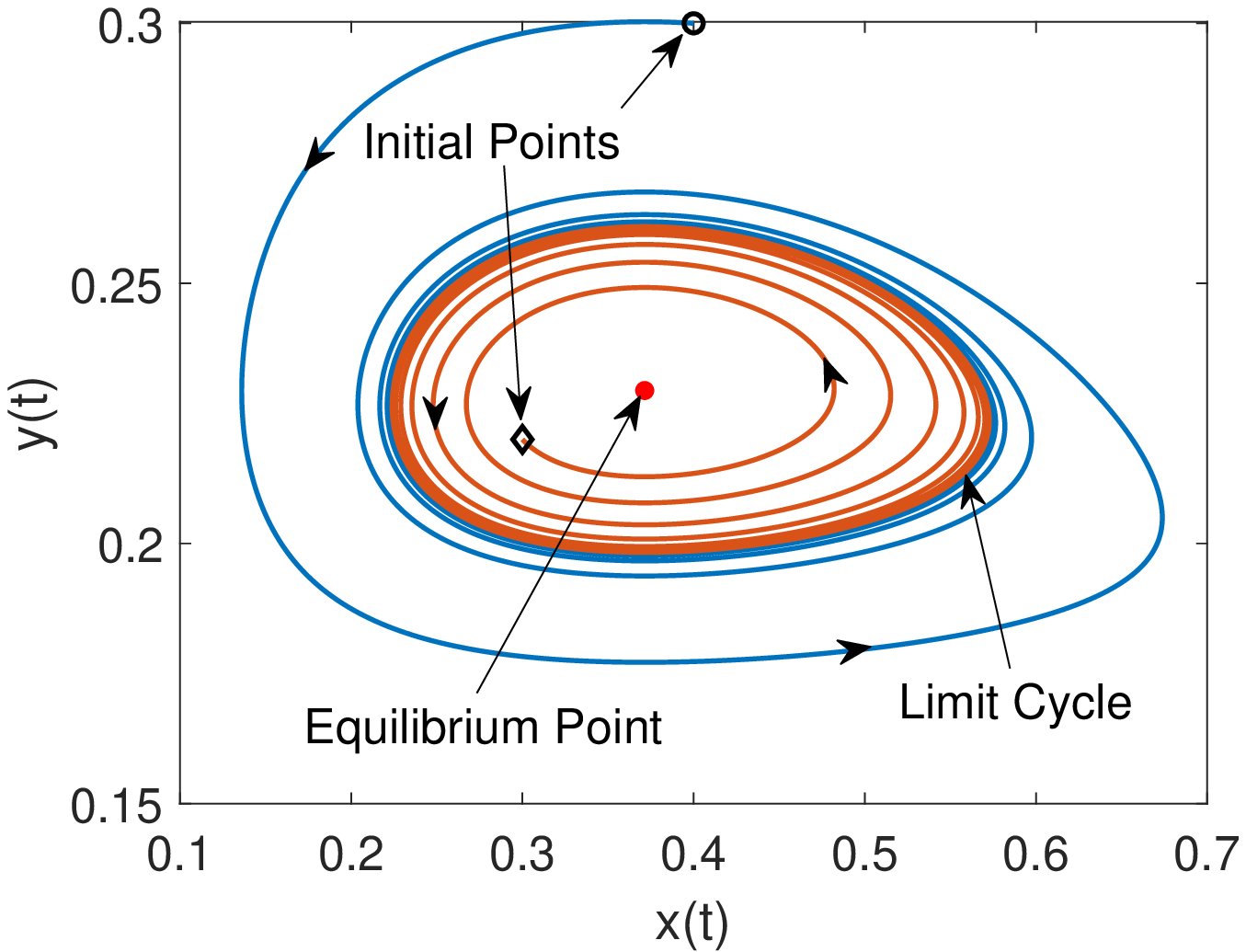}}\\
 \subfloat[Time series when $d=0.4676$\label{subfig-8}]{%
  \includegraphics[width=0.33\textwidth, height=4.5cm]{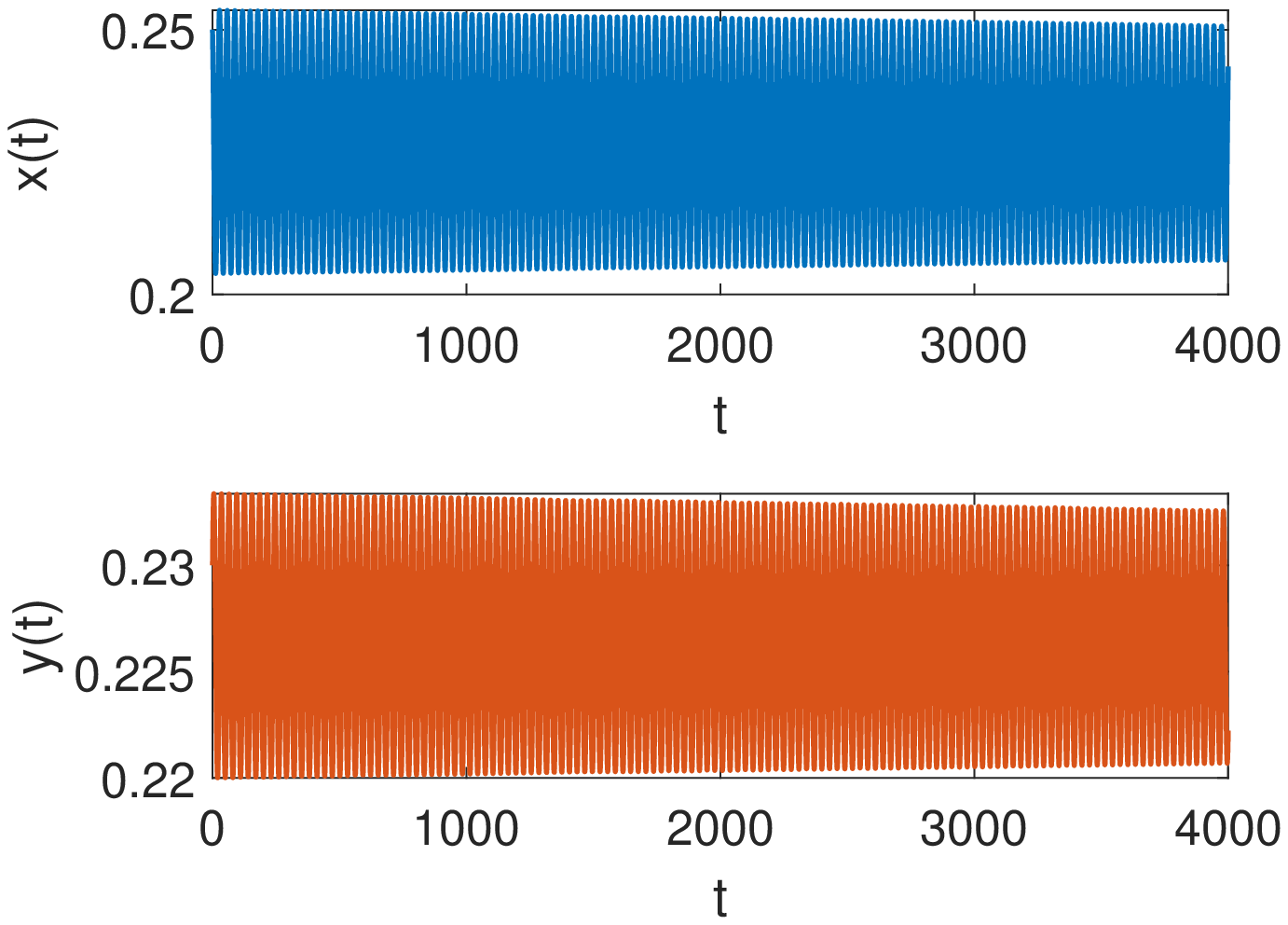}
  \includegraphics[width=0.33\textwidth, height=4.5cm]{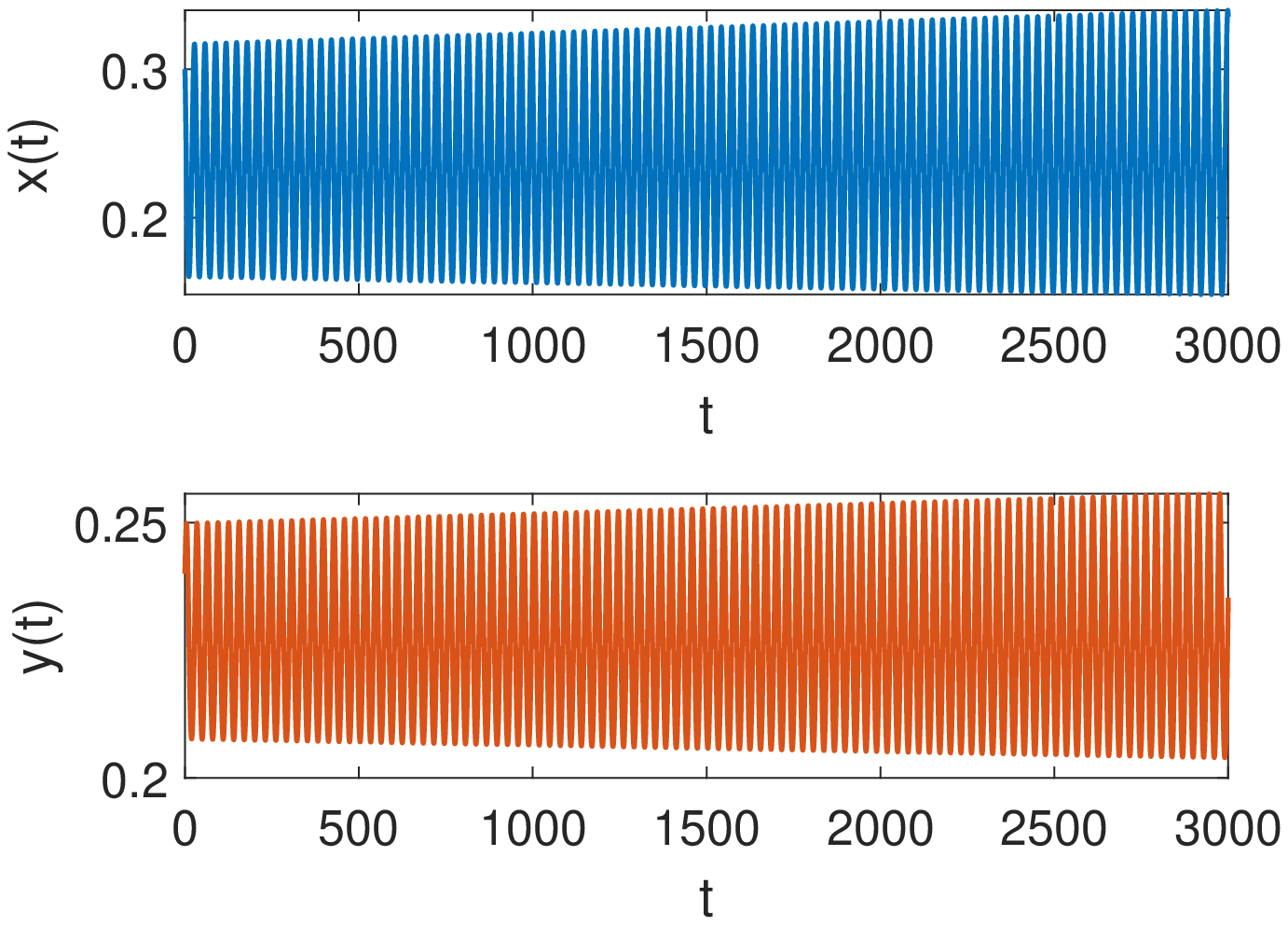}
  \includegraphics[width=0.33\textwidth, height=4.5cm]{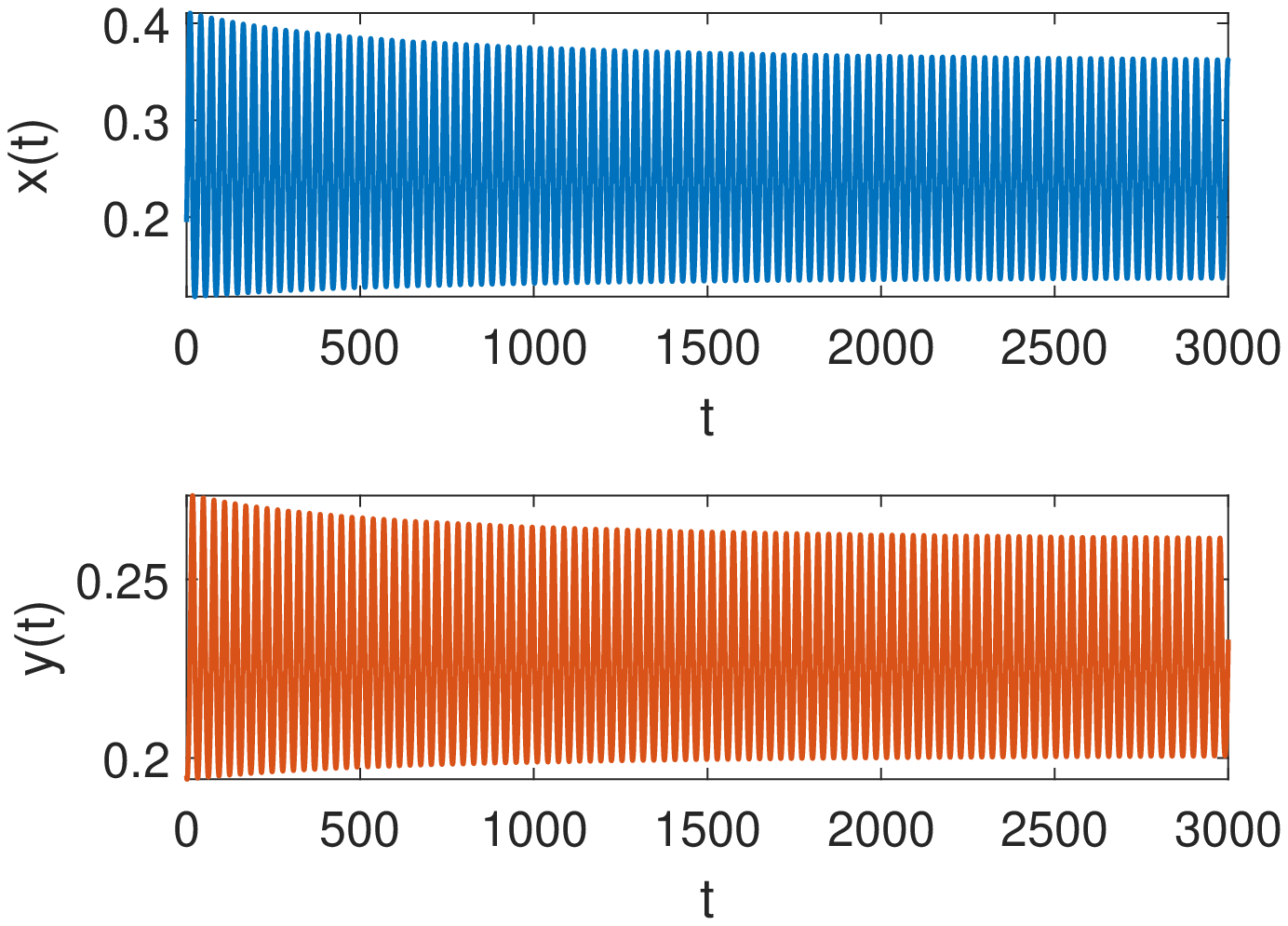}}\\
 \subfloat[Phase portrait when $d=0.4676$\label{subfig-8}]{%
  \includegraphics[width=0.33\textwidth, height=4.5cm]{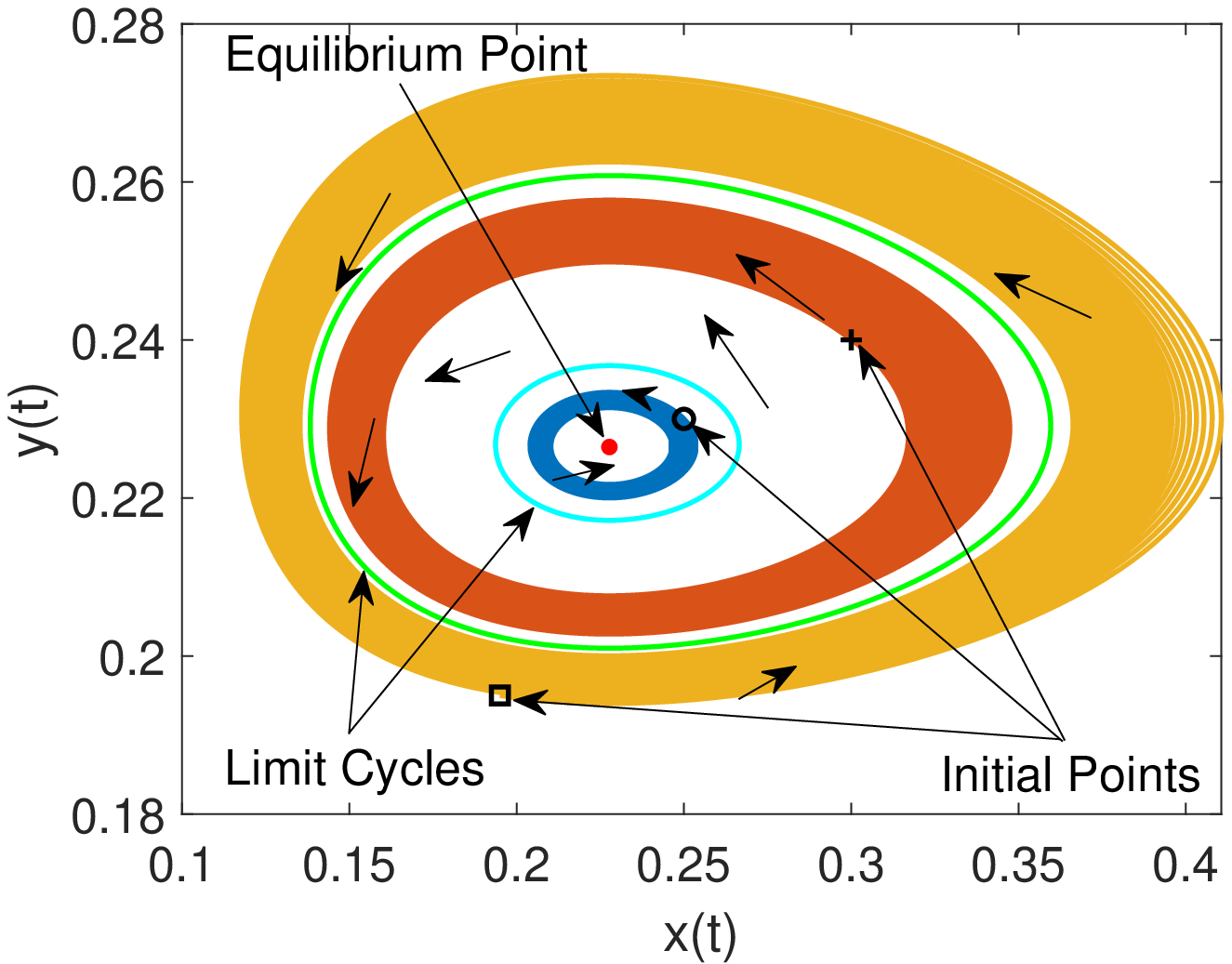}}
 \subfloat[Time series \& phase portrait when $d=0.5$\label{subfig-8}]{%
  \includegraphics[width=0.33\textwidth, height=4.5cm]{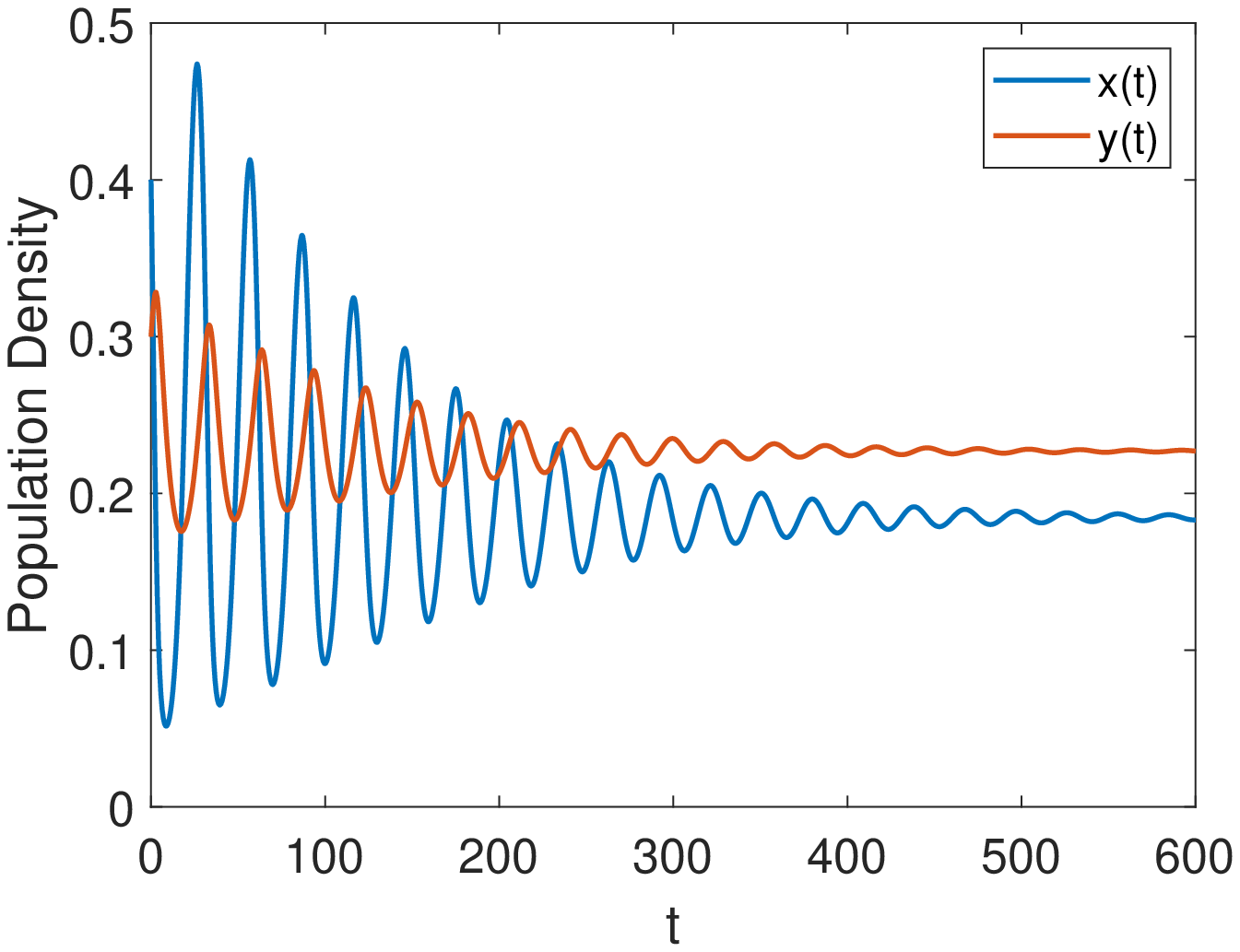}
  \includegraphics[width=0.33\textwidth, height=4.5cm]{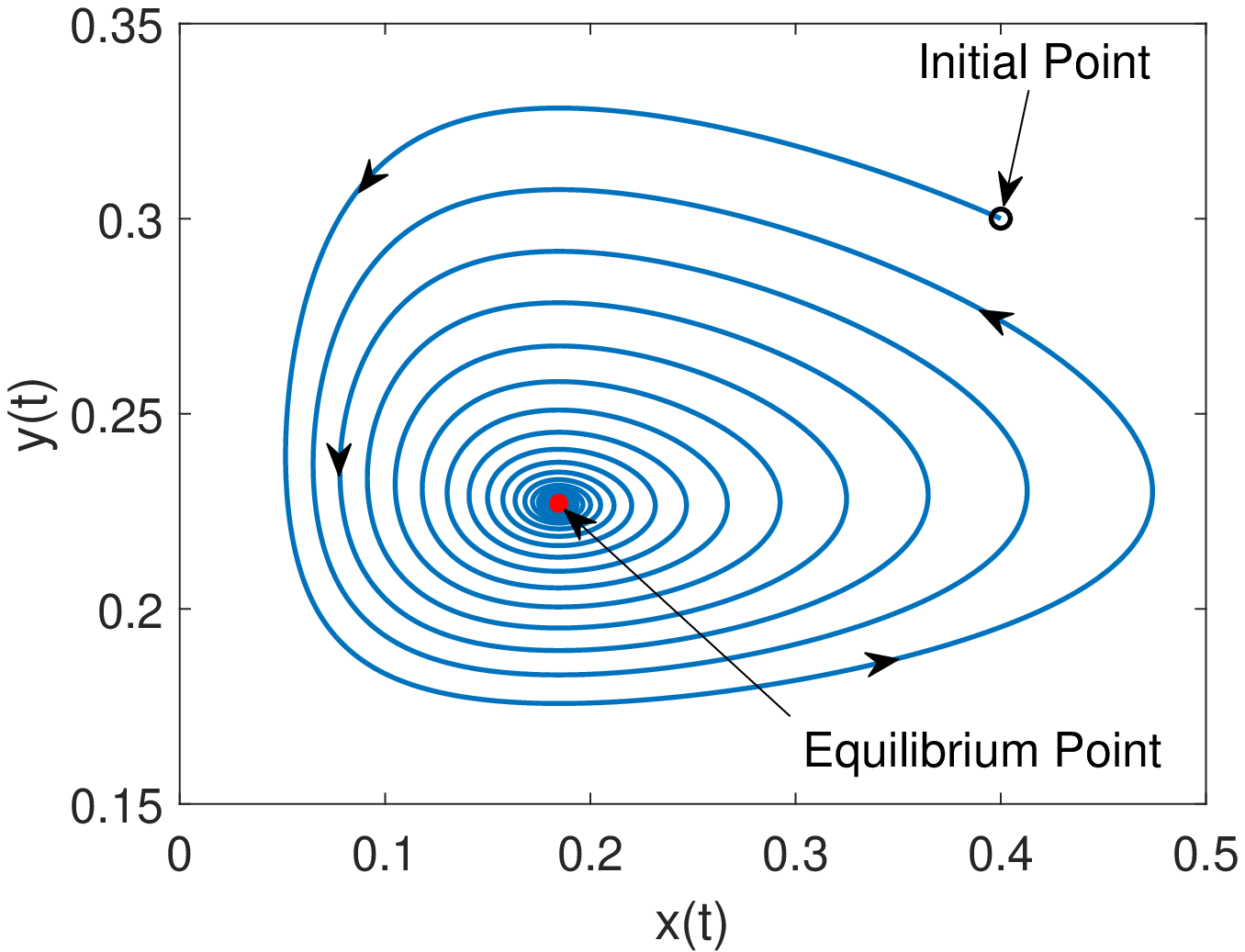}}
 \caption{{\bf(a)} The positive equilibrium of system (\ref{eq:sys}) undergoes Hopf bifurcation for $d_{H1}=0.31311396, d_{H2}=0.46737359$ \& $d_{H3}=0.99999883$ when $\varrho=0$. {\bf(b)} Limit point of cycles(LPC) occurs at $d_{LPC}=0.4680539$. {\bf(c)} $E$ is stable for $d=0.28<d_{H1}$, where the initial point of the simulation is $\circ(0.4,0.3)$. {\bf(d)} $E$ is unstable and there exists a stable limit cycle for $d_{H1}<d=0.35<d_{H2}$, where the initial points of the simulation are $\circ(0.4,0.3)$ \& $\diamond(0.3,0.22)$. {\bf(e)} Time series for $d_{H2}<d=0.4676<d_{LPC}$. {\bf(f)} Phase portrait corresponding to time series plotted in (e) with initial conditions $\circ(0.25,0.23)$, $+(0.3,0.24)$ \& $\square(0.195,0.195)$, where the equilibrium point is stable, the inner limit cycle is unstable and the outer limit cycle is stable. {\bf(g)} The equilibrium point $E$ is stable for $d=0.5>d_{LPC}$.}
 \label{fig_d}
\end{figure}
In figure \ref{fig_e}, the bifurcation diagram of the corresponding non-delayed system of system (\ref{eq:sys}) is plotted with respect to the parameter $e$ along with time series and phase portrait for different values of $e$.
The figure shows that the system undergoes Hopf bifurcation at the equilibrium point $E$ for $e=e_{H1}, e_{H2}$ \& $e_{H3}$ with $m=1.2$, $p=2$, $c=0.3$, $d=0.4$ and $a=0.2$. For the positive equilibrium point to be biologically feasible, the value of $e$ should be greater than $0.8$. Limit cycles originating from the Hopf  points collide and vanish at the LPC-point $e=e_{LPC}$.
\begin{figure}[H]
 \subfloat[Bifurcation diagram\label{subfig-10}]{%
  \includegraphics[width=0.35\textwidth, height=4.5cm]{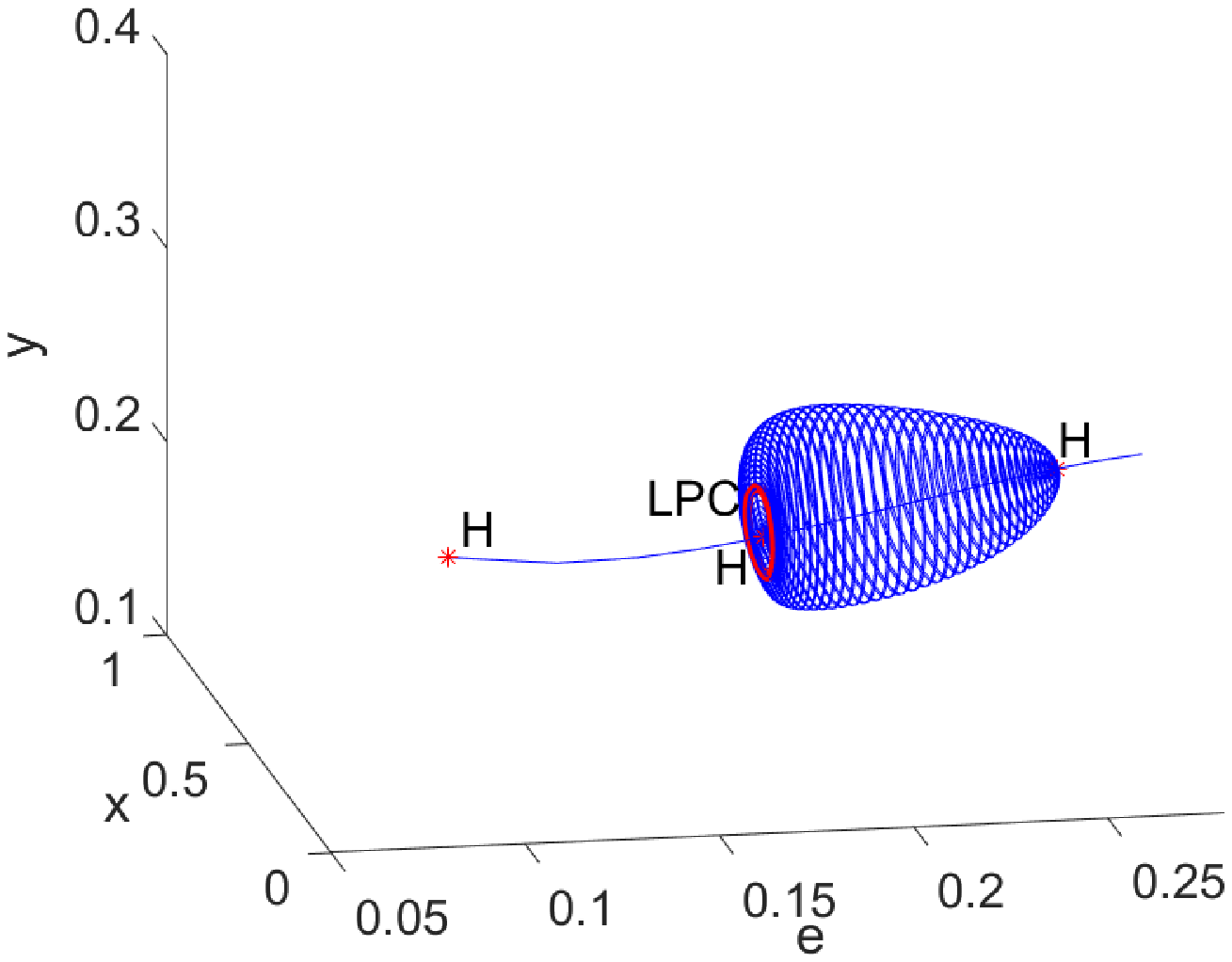}}\\
 \subfloat[Time series \& phase portrait for $e=0.12$\label{subfig-11}]{%
  \includegraphics[width=0.25\textwidth, height=3.5cm]{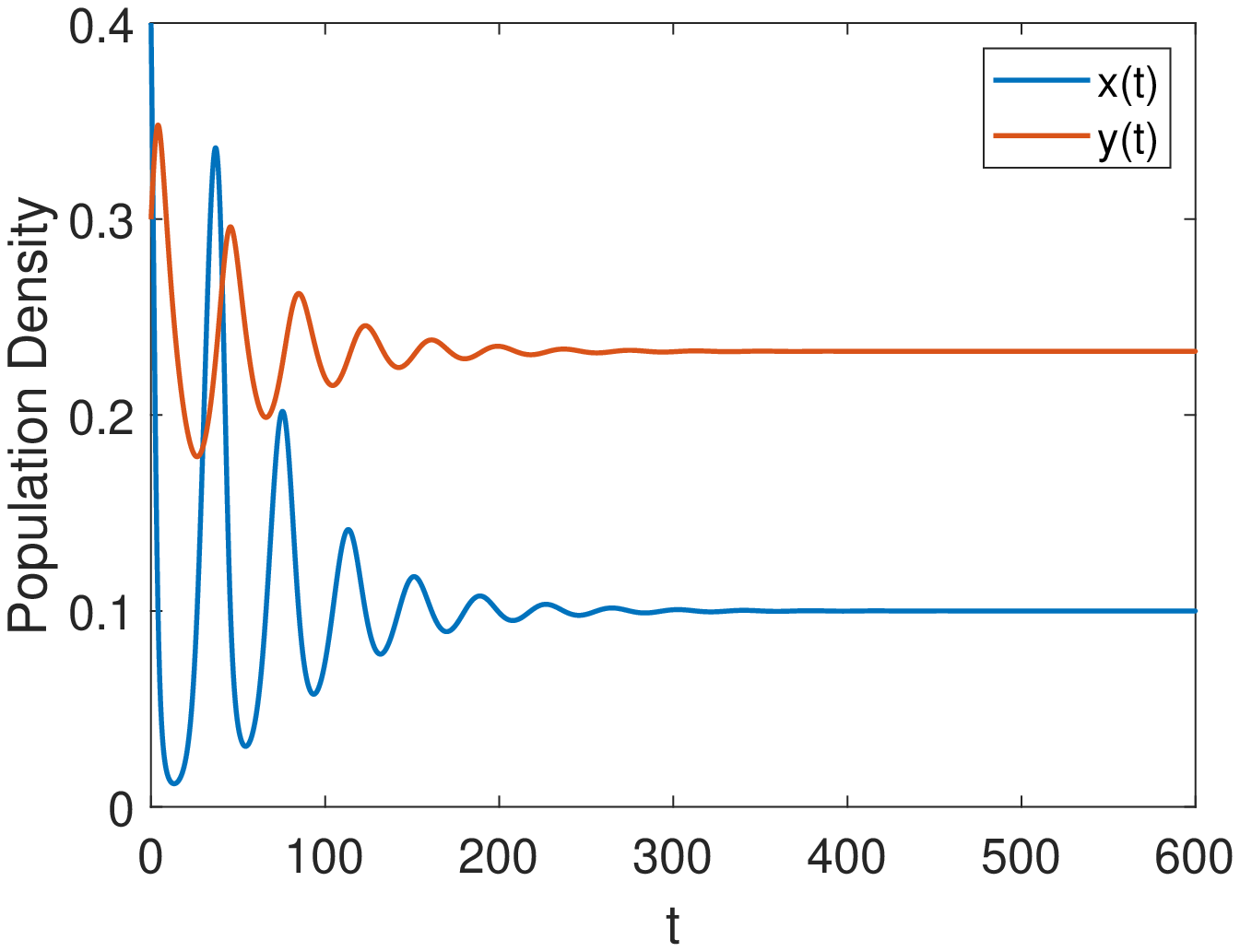}
  \includegraphics[width=0.25\textwidth, height=3.5cm]{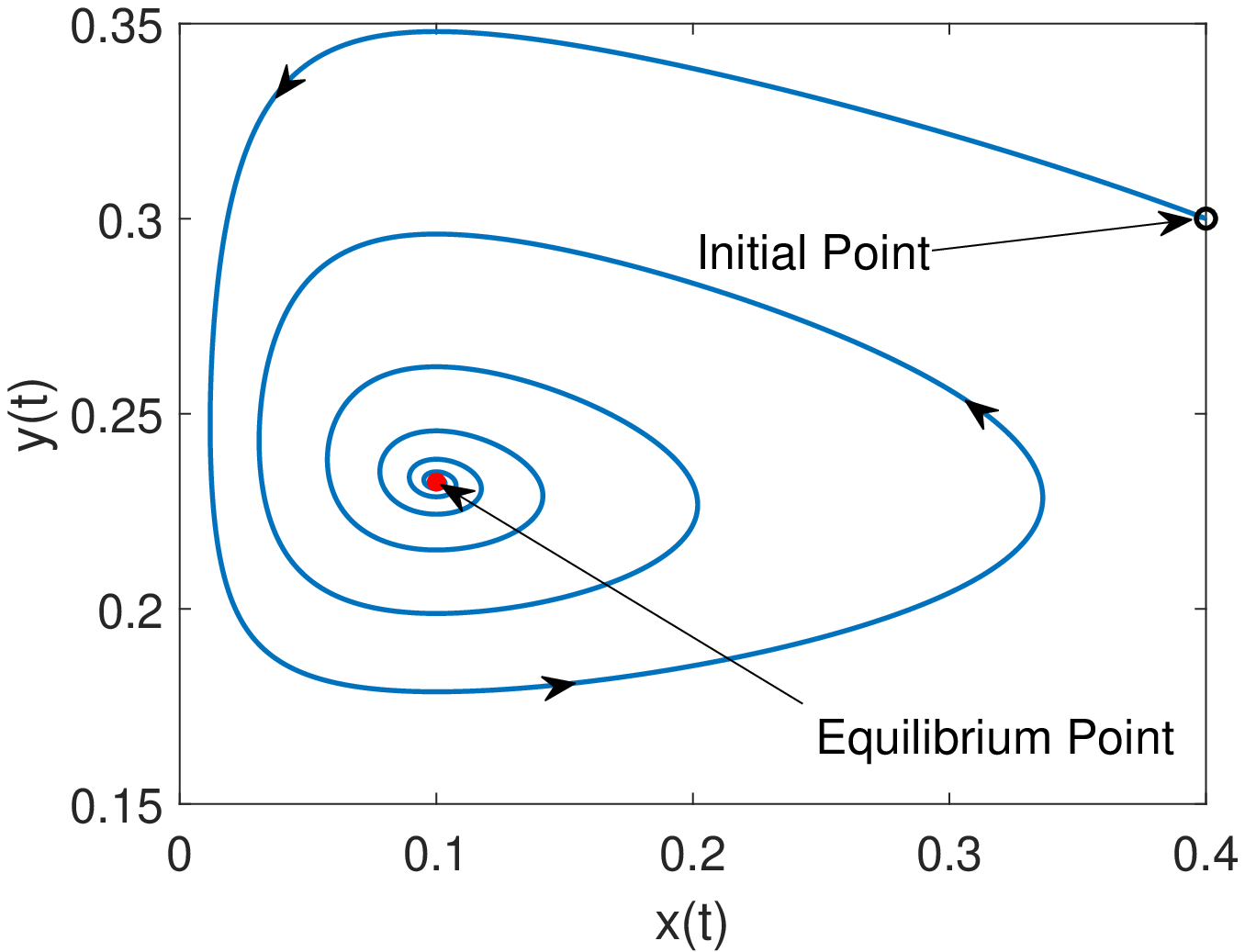}}
 \subfloat[Time series \& phase portrait for $e=0.2$\label{subfig-11}]{%
  \includegraphics[width=0.25\textwidth, height=3.5cm]{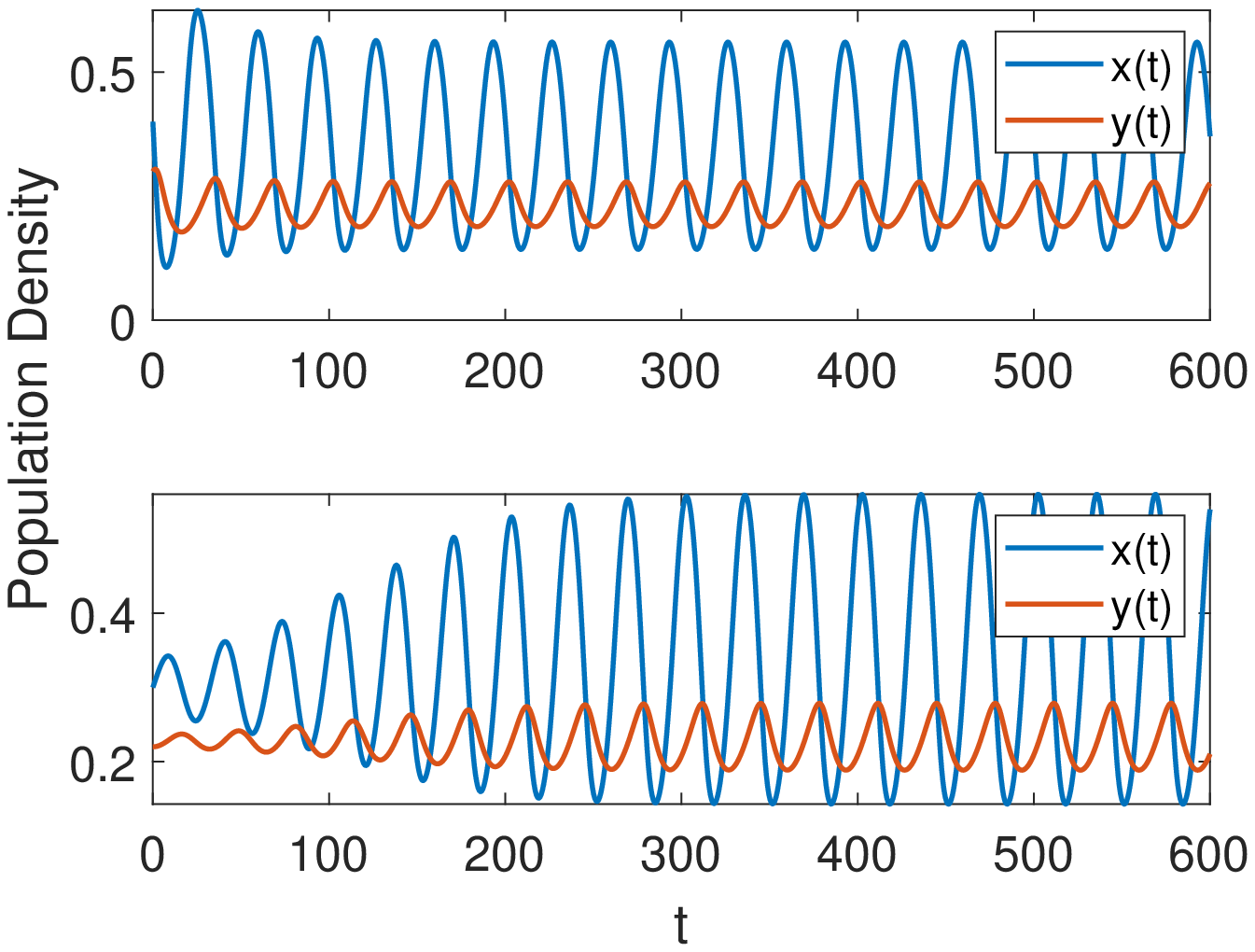}
  \includegraphics[width=0.25\textwidth, height=3.5cm]{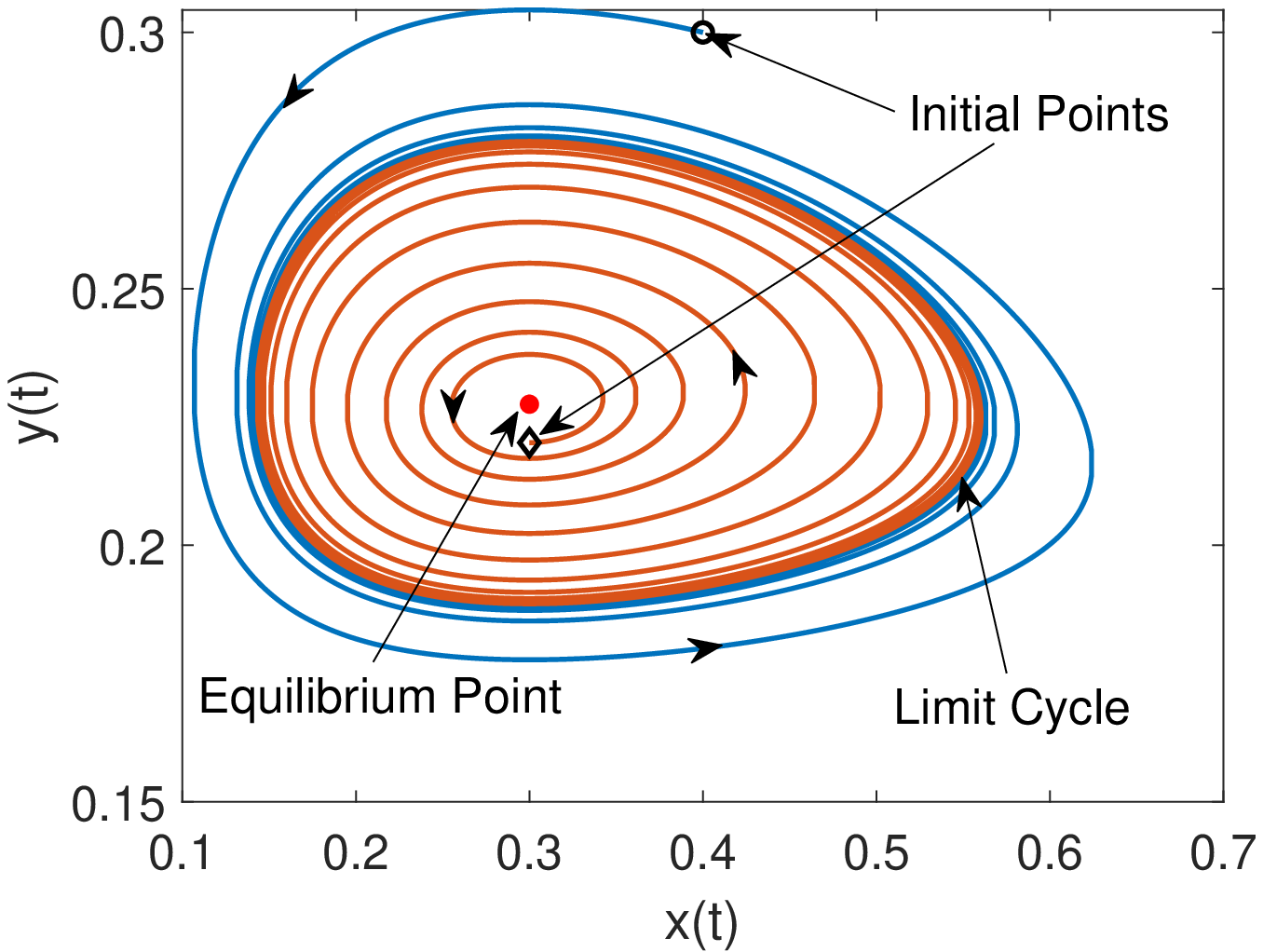}}\\
 \subfloat[Time series for $e=0.171$\label{subfig-11}]{%
  \includegraphics[width=0.33\textwidth, height=4.5cm]{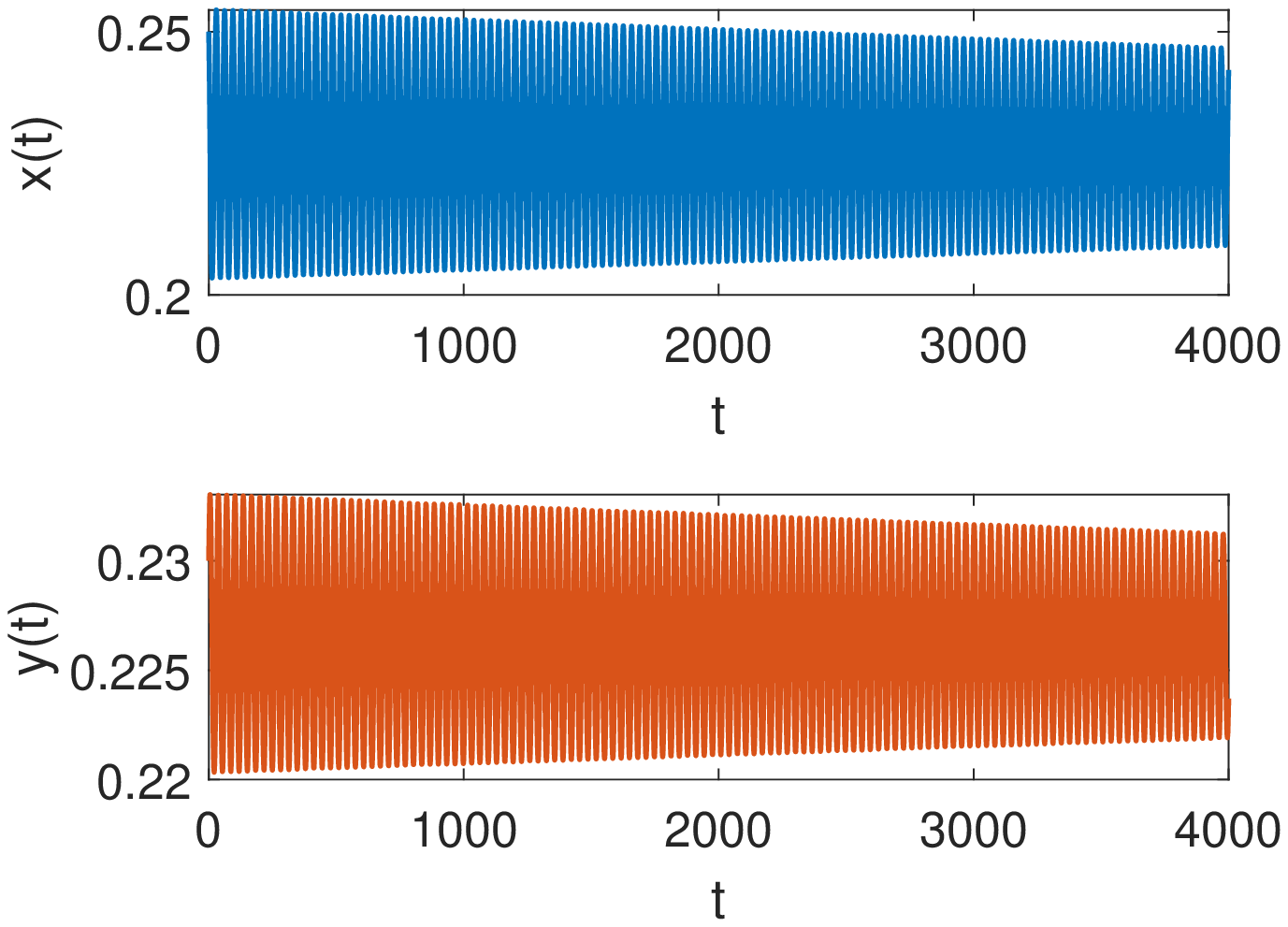}
  \includegraphics[width=0.33\textwidth, height=4.5cm]{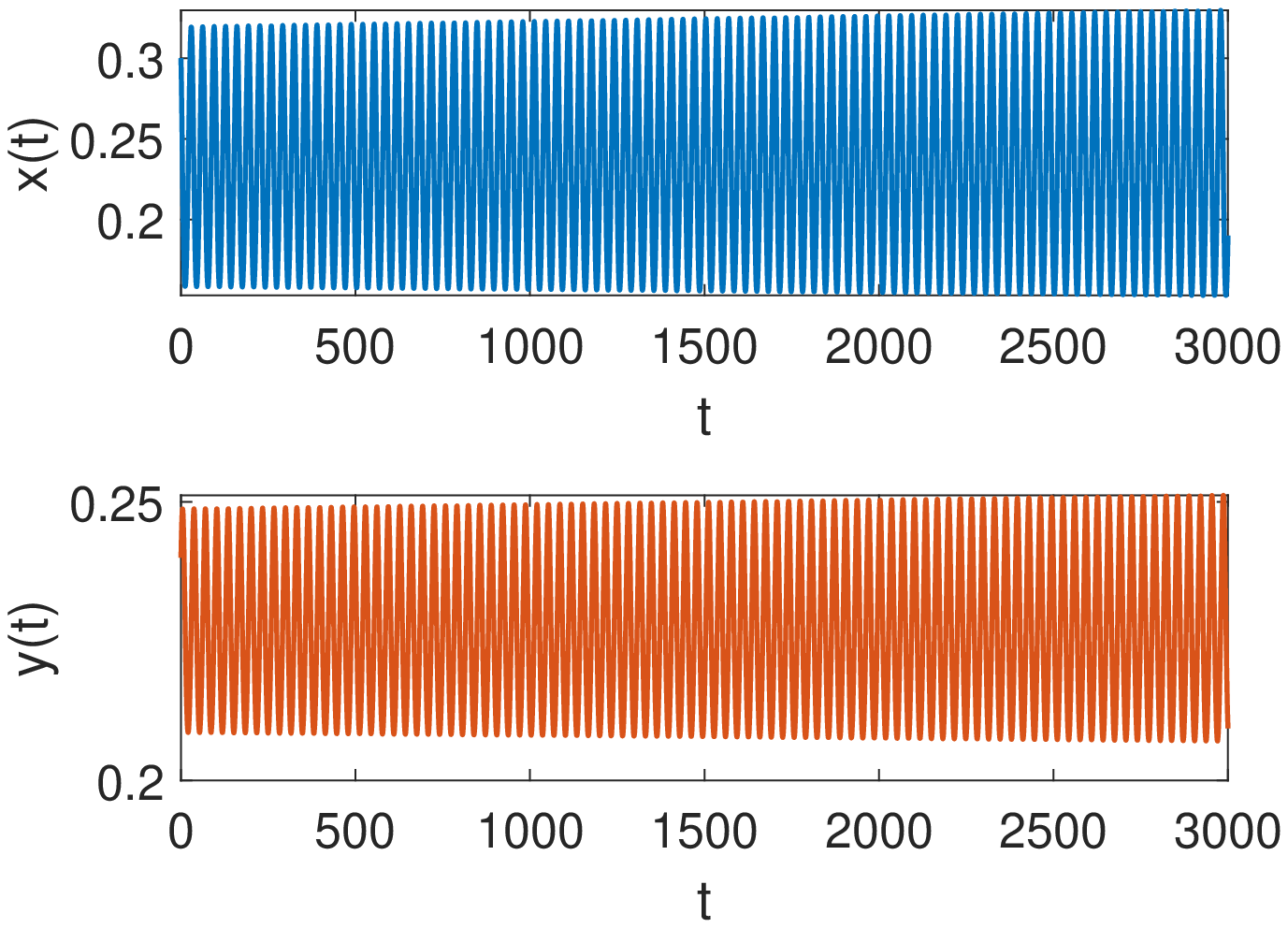}
  \includegraphics[width=0.33\textwidth, height=4.5cm]{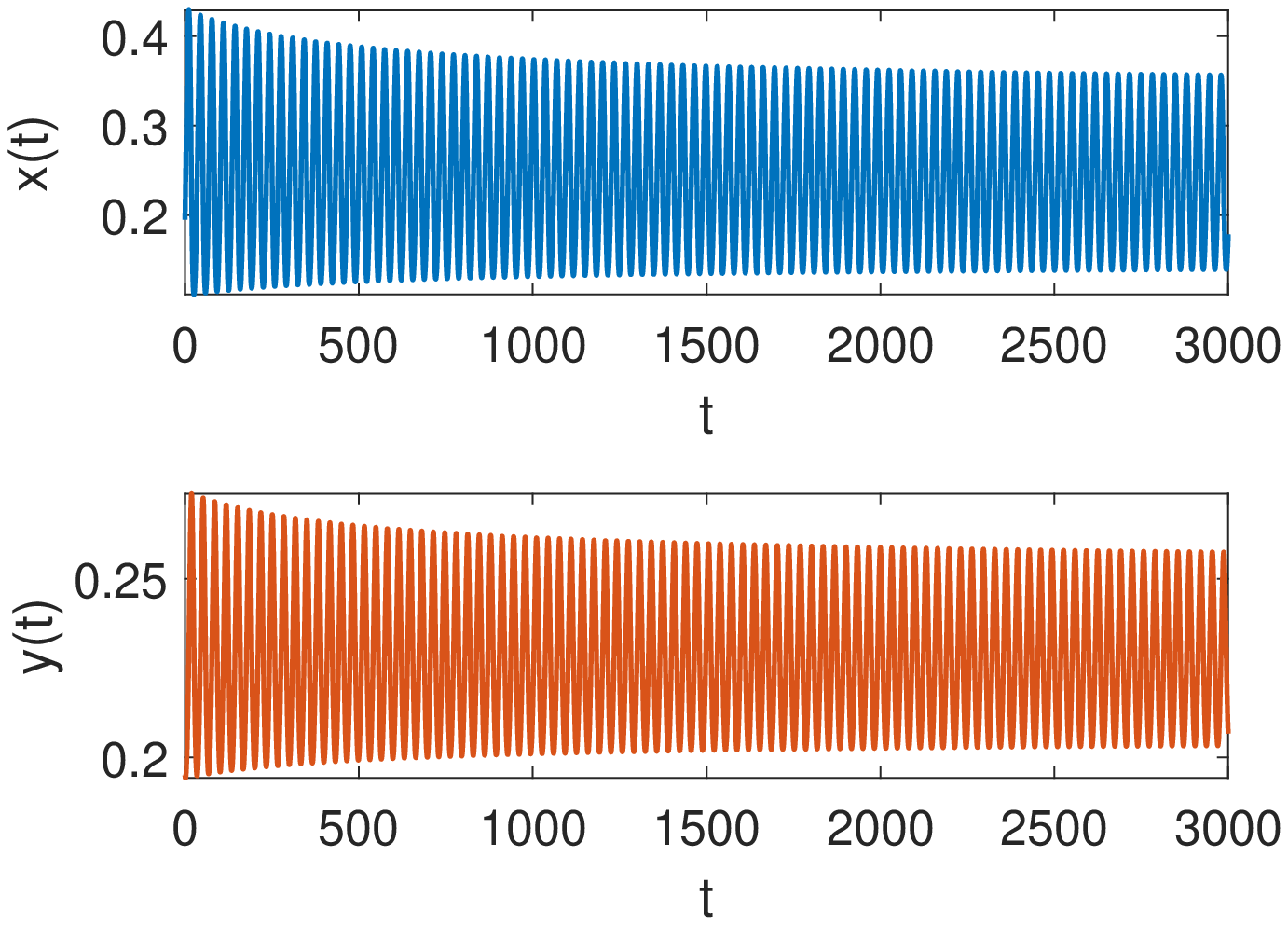}}\\
 \subfloat[Phase portrait for $e=0.171$\label{subfig-11}]{%
  \includegraphics[width=0.33\textwidth, height=4.5cm]{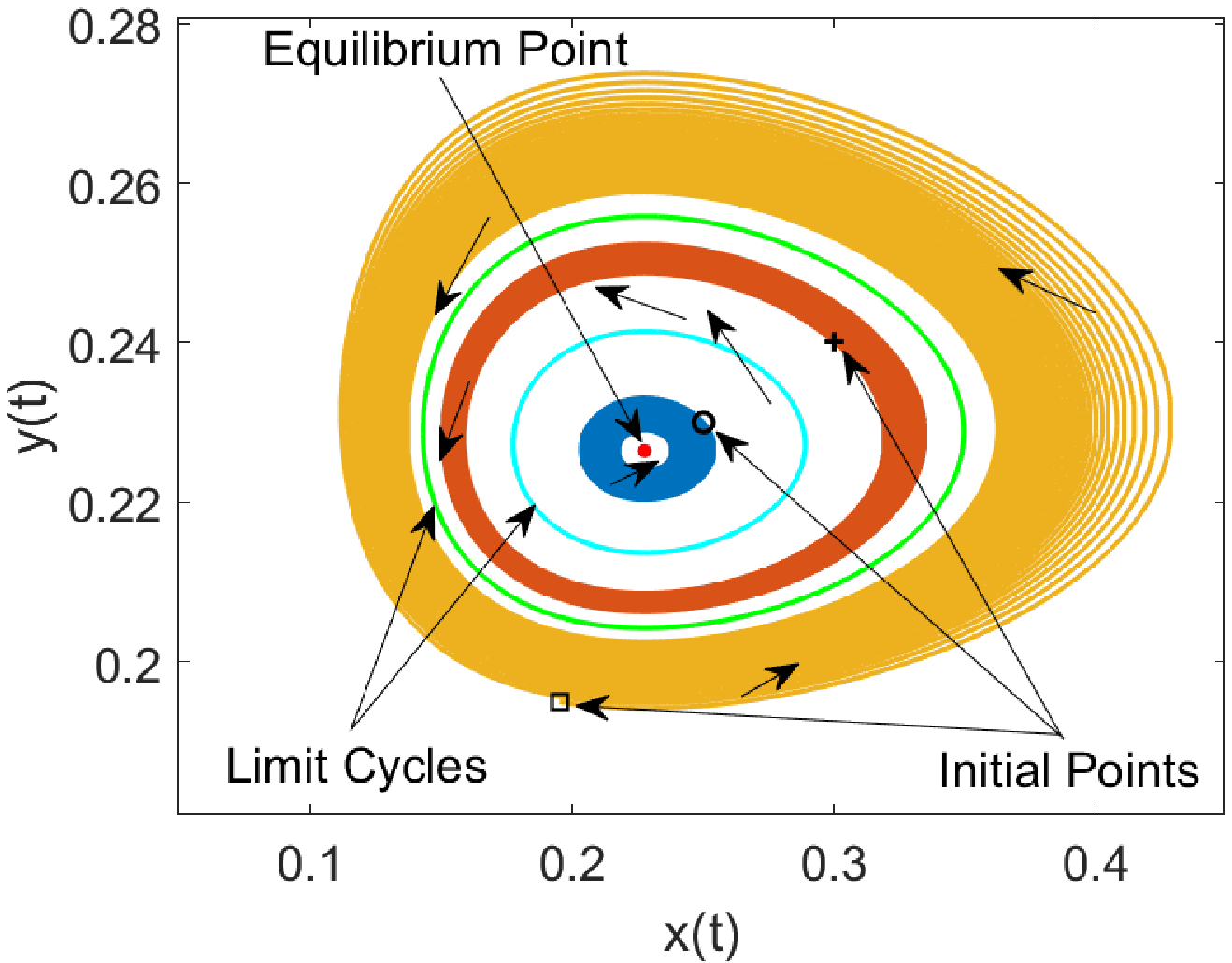}}
 \subfloat[Time series \& phase portrait for $e=0.28$\label{subfig-12}]{%
  \includegraphics[width=0.33\textwidth, height=4.5cm]{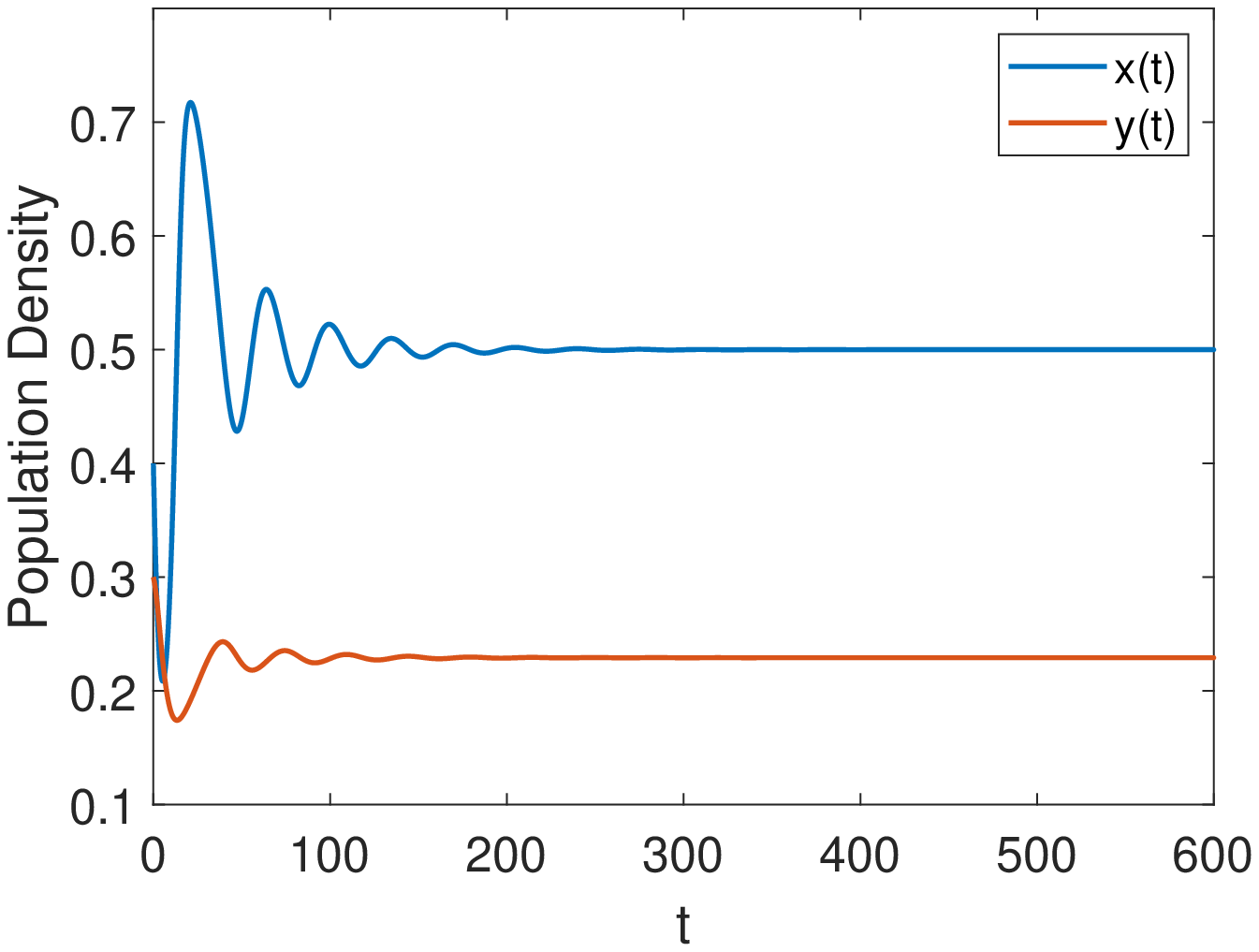}
  \includegraphics[width=0.33\textwidth, height=4.5cm]{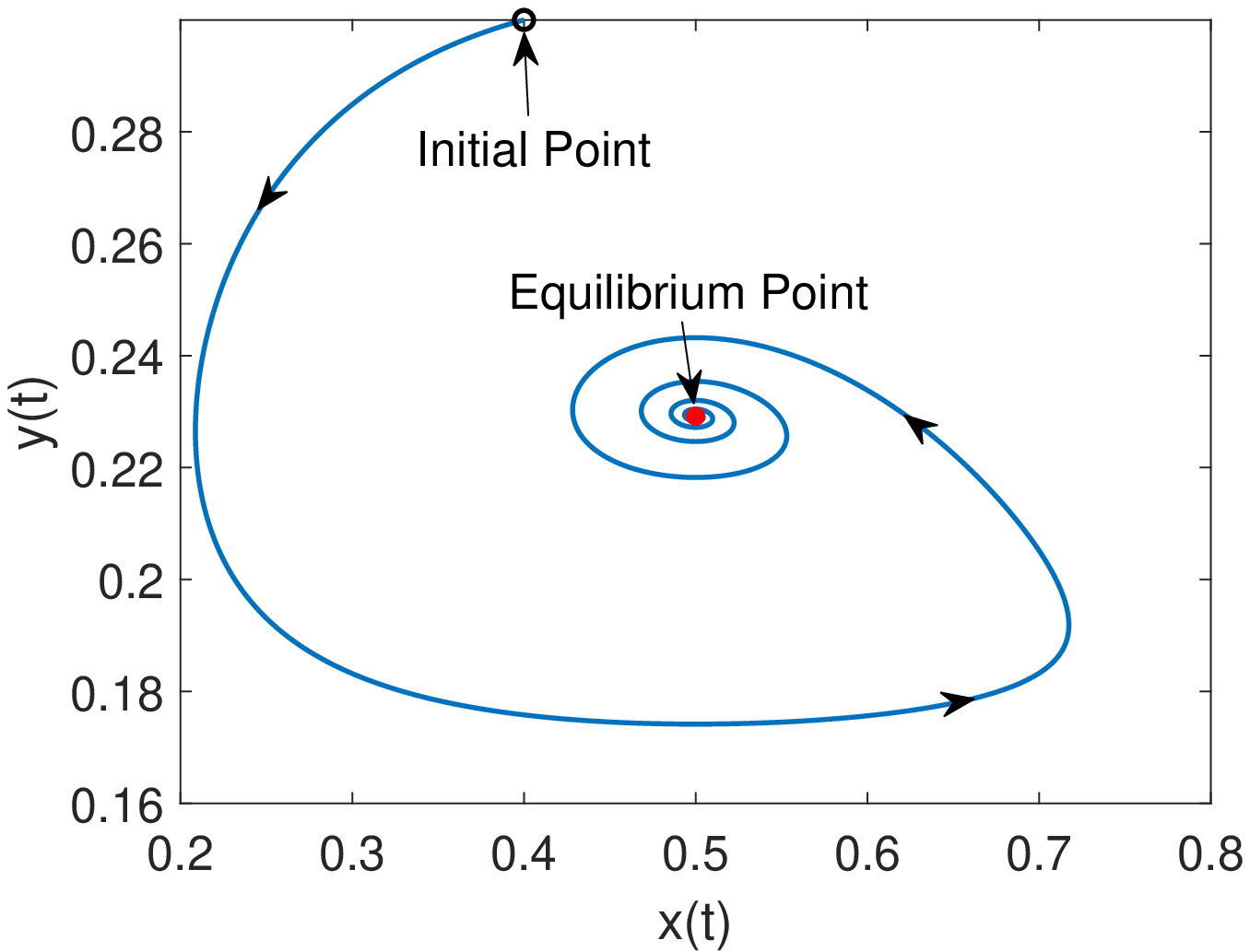}}
 \caption{{\bf(a)} The positive equilibrium of system (\ref{eq:sys}) undergoes Hopf bifurcation for $e_{H1}=0.08000347, e_{H2}=0.17116249$ \& $e_{H3}=0.2554956$ when $\varrho=0$. Limit point of cycles(LPC) occurs at $e_{LPC}=0.17092059$. {\bf(b)} $E$ is stable for $e=0.12<e_{LPC}$, where the initial point of the simulation is $\circ(0.4,0.3)$. {\bf(c)} $E$ is unstable and there exists a stable limit cycle for $e_{H2}<e=0.2<e_{H3}$, where the initial points of the simulation are $\circ(0.4,0.3)$ \& $\diamond(0.3,0.22)$. {\bf(d)} Time series for $e_{LPC}<e=0.171<e_{H2}$. {\bf(e)} Phase portrait corresponding to time series plotted in (d) with initial conditions $\circ(0.25,0.23)$, $+(0.3,0.24)$ \& $\square(0.195,0.195)$, where the equilibrium point is stable, the inner limit cycle is unstable and the outer limit cycle is stable. {\bf(f)} The equilibrium point $E$ is stable for $e=0.28>e_{H3}$, where the initial point is $\circ(0.4,0.3)$.}
 \label{fig_e}
\end{figure}
The bifurcation diagram for system (\ref{eq:sys}) along with time series and phase portrait for different values of $e$ are plotted in figure \ref{fig_a} which shows the occurrence of LPC-bifurcation at $a=a_{LPC}$ and three Hopf bifurcations at $E$ for $a=a_{H1}, a_{H2}$ \& $a_{H3}$ when $\varrho=0$ with $m=1.2$, $p=2$, $c=0.3$, $d=0.4$ and $e=0.2$. Clearly, $a>0$. The values of $a$ must be less than $0.5$ for the positive equilibrium point to exist.
\begin{figure}[H]
 \subfloat[Bifurcation diagram\label{subfig-13}]{%
  \includegraphics[width=0.35\textwidth, height=4.5cm]{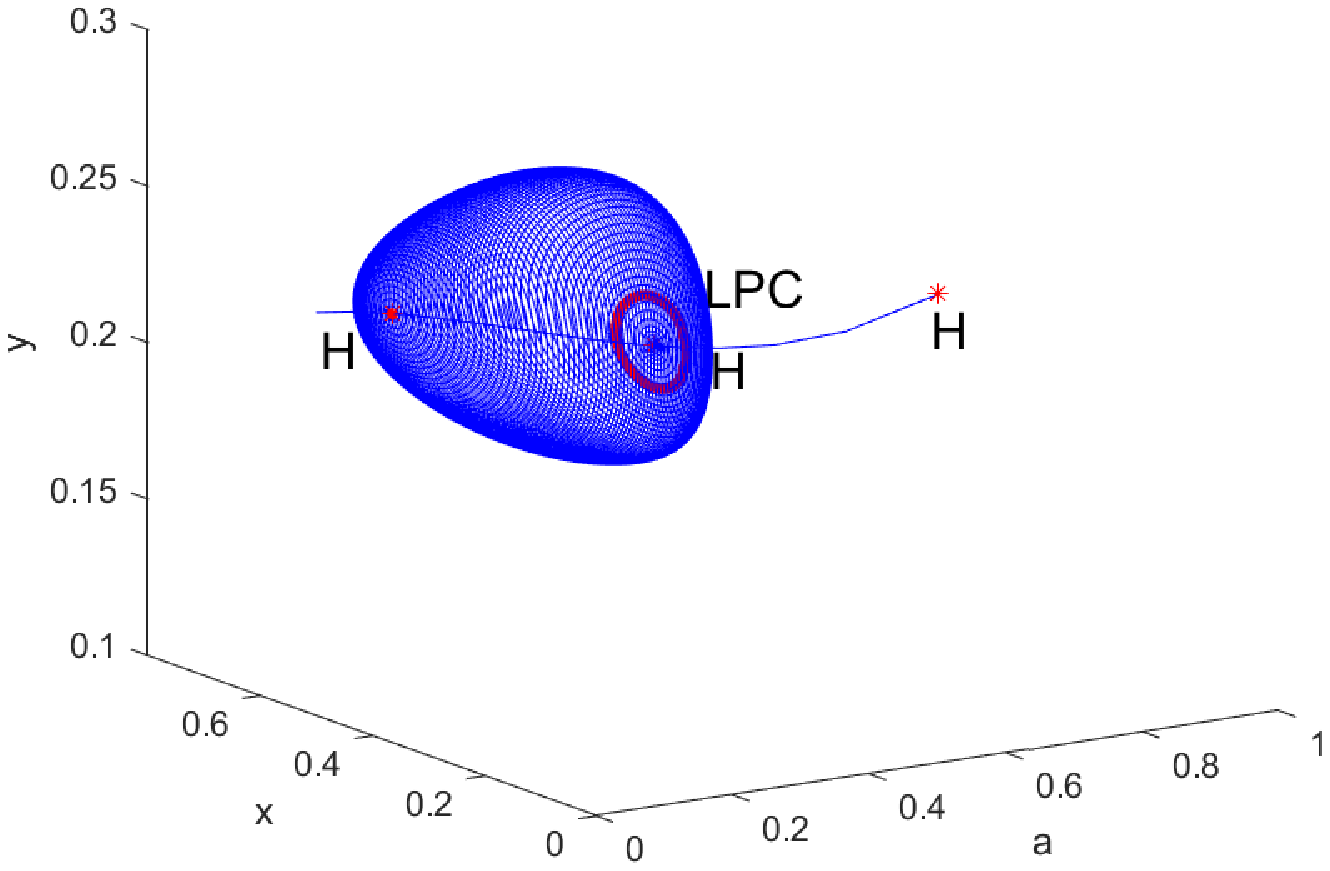}}\\
 \subfloat[Time series and phase portrait for $a=0.04$\label{subfig-14}]{%
  \includegraphics[width=0.25\textwidth, height=3.5cm]{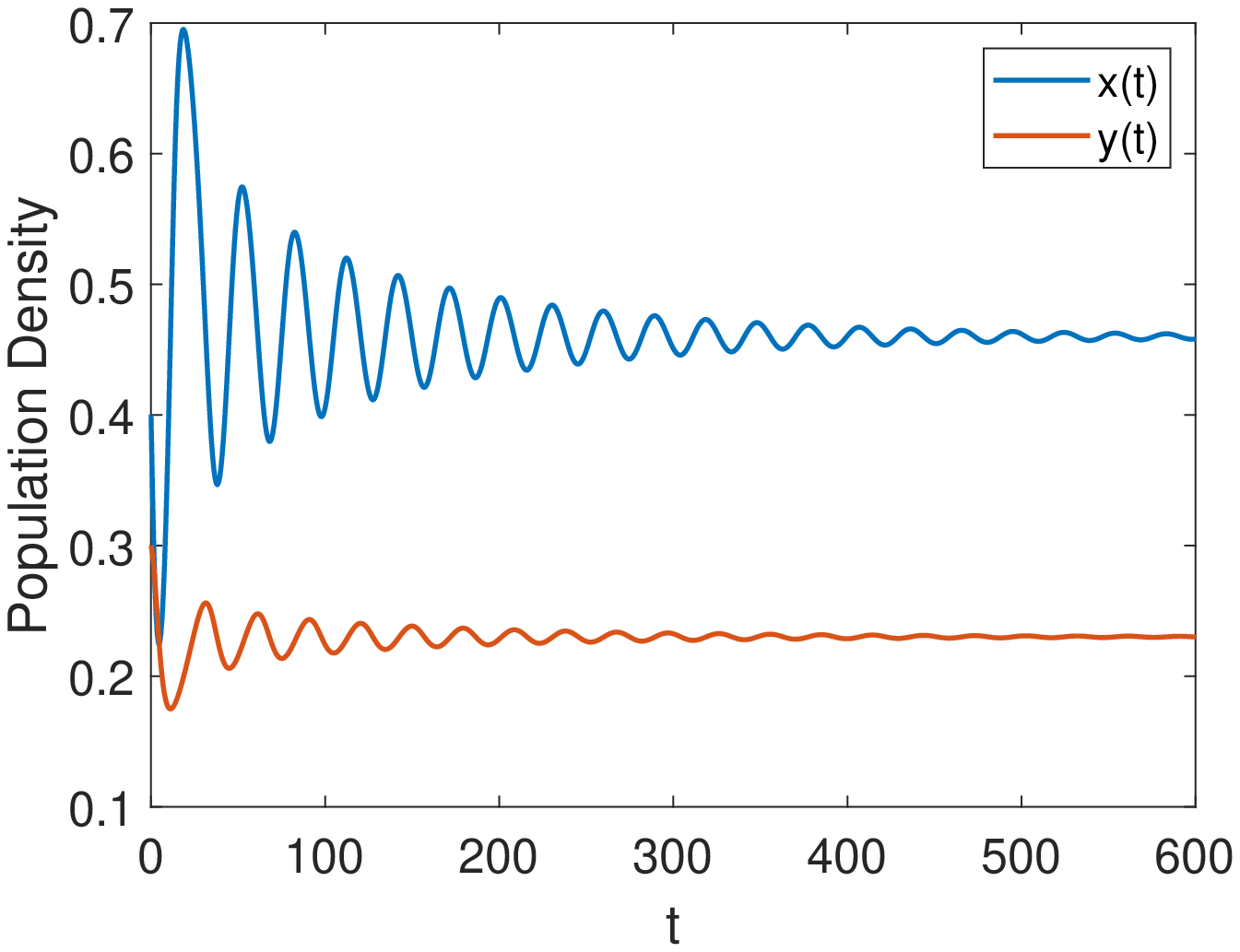}
  \includegraphics[width=0.25\textwidth, height=3.5cm]{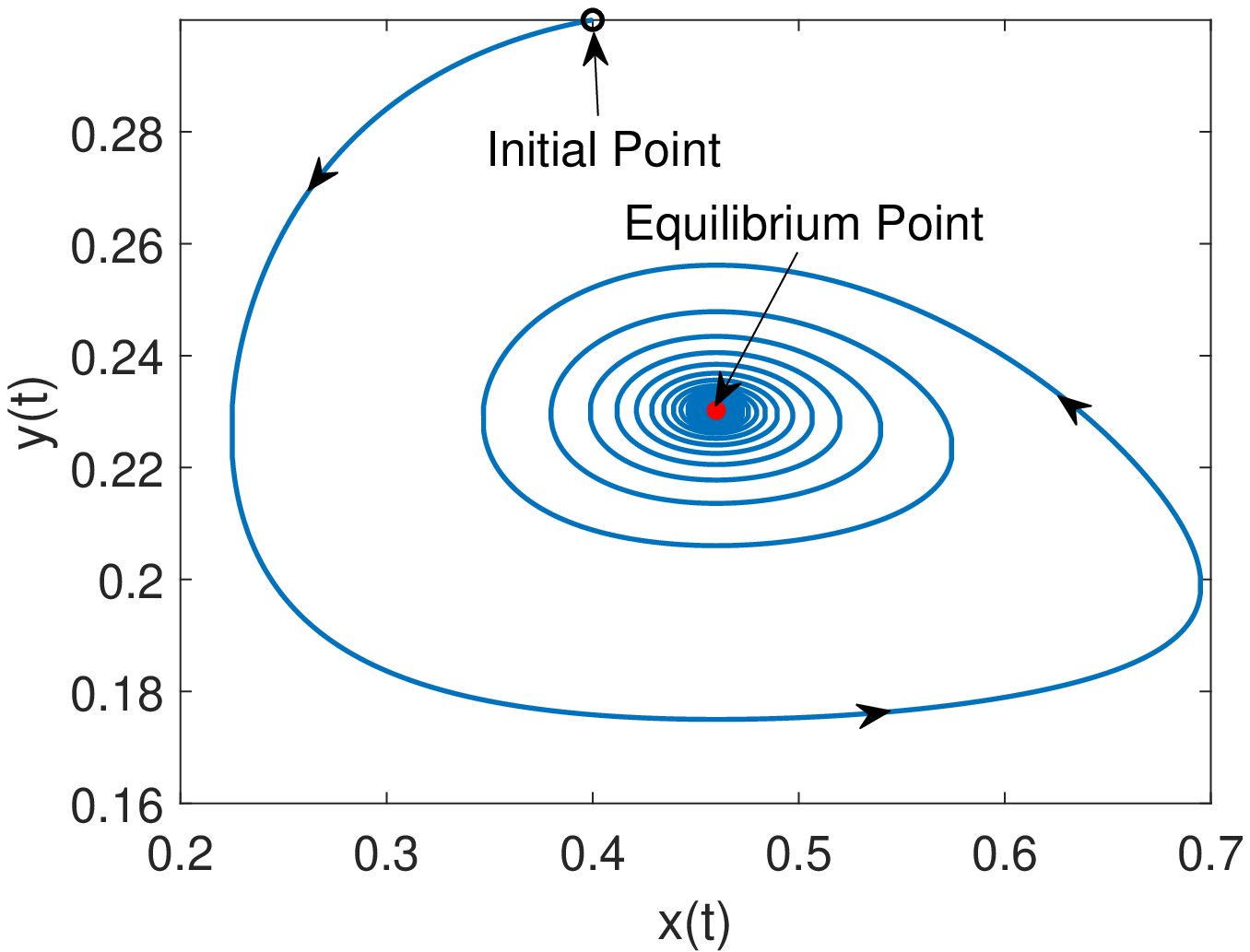}}
 \subfloat[Time series and phase portrait for $a=0.1$\label{subfig-14}]{%
  \includegraphics[width=0.25\textwidth, height=3.5cm]{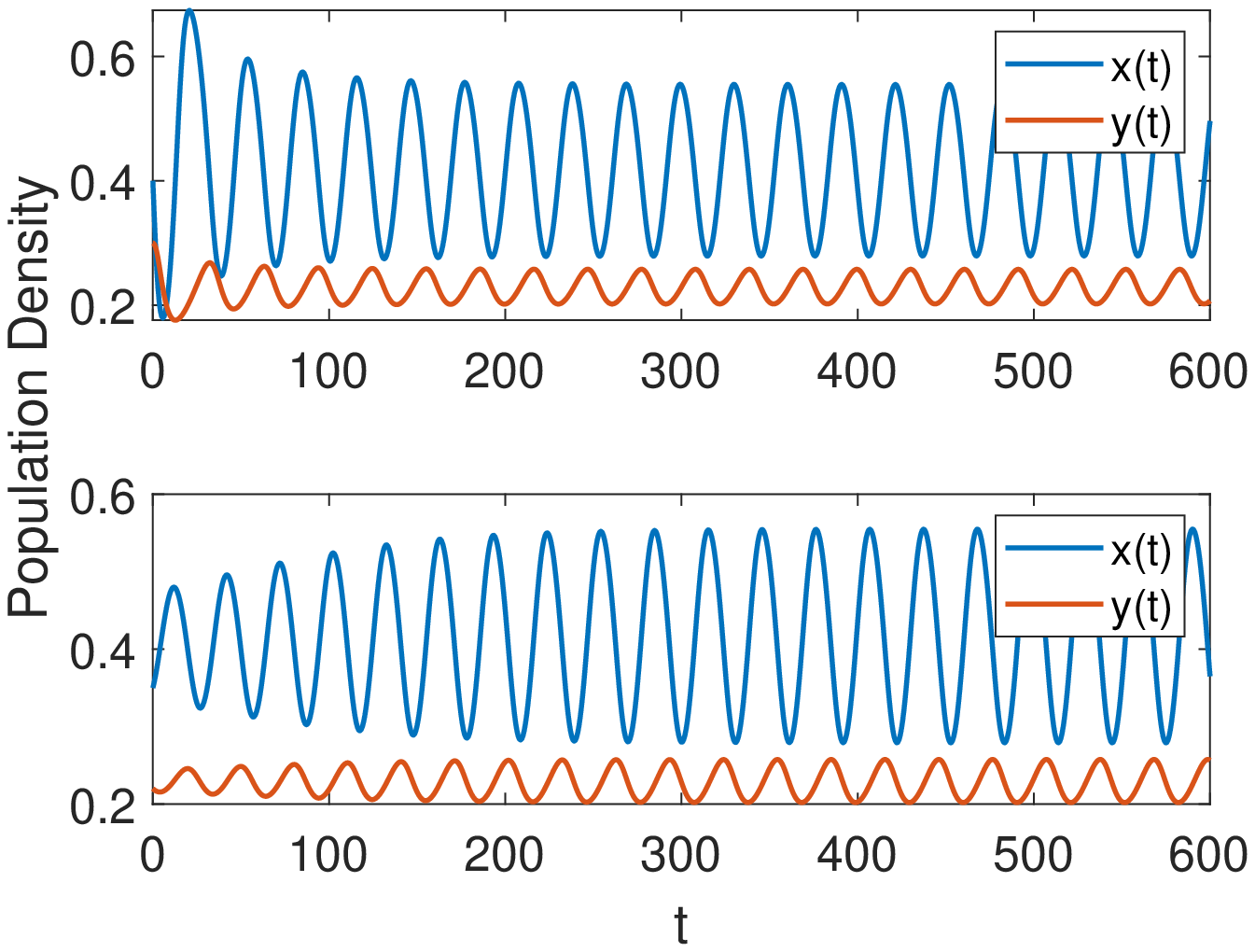}
  \includegraphics[width=0.25\textwidth, height=3.5cm]{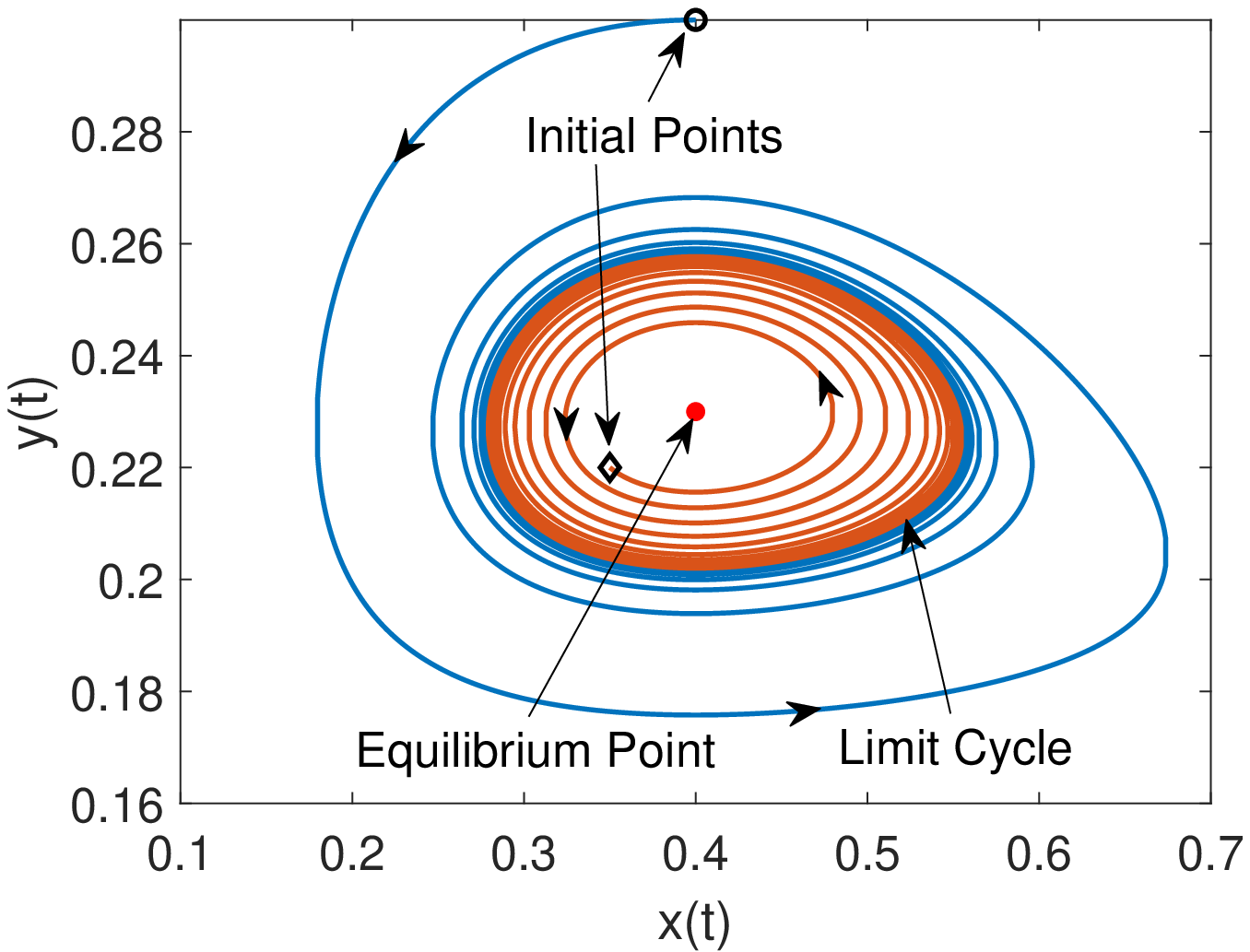}}\\
 \subfloat[Time series for $a=0.2722$\label{subfig-11}]{%
  \includegraphics[width=0.33\textwidth, height=4.5cm]{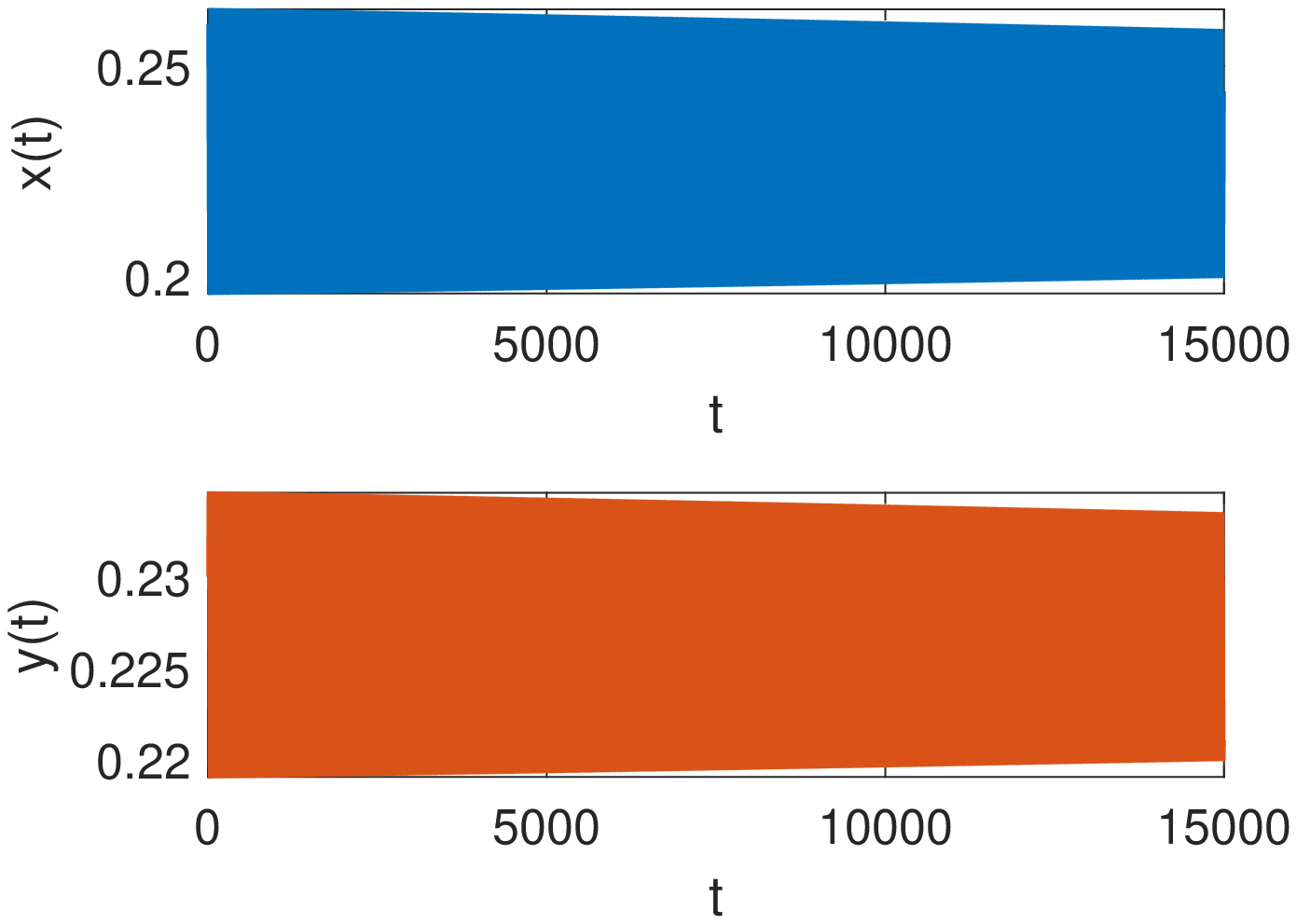}
  \includegraphics[width=0.33\textwidth, height=4.5cm]{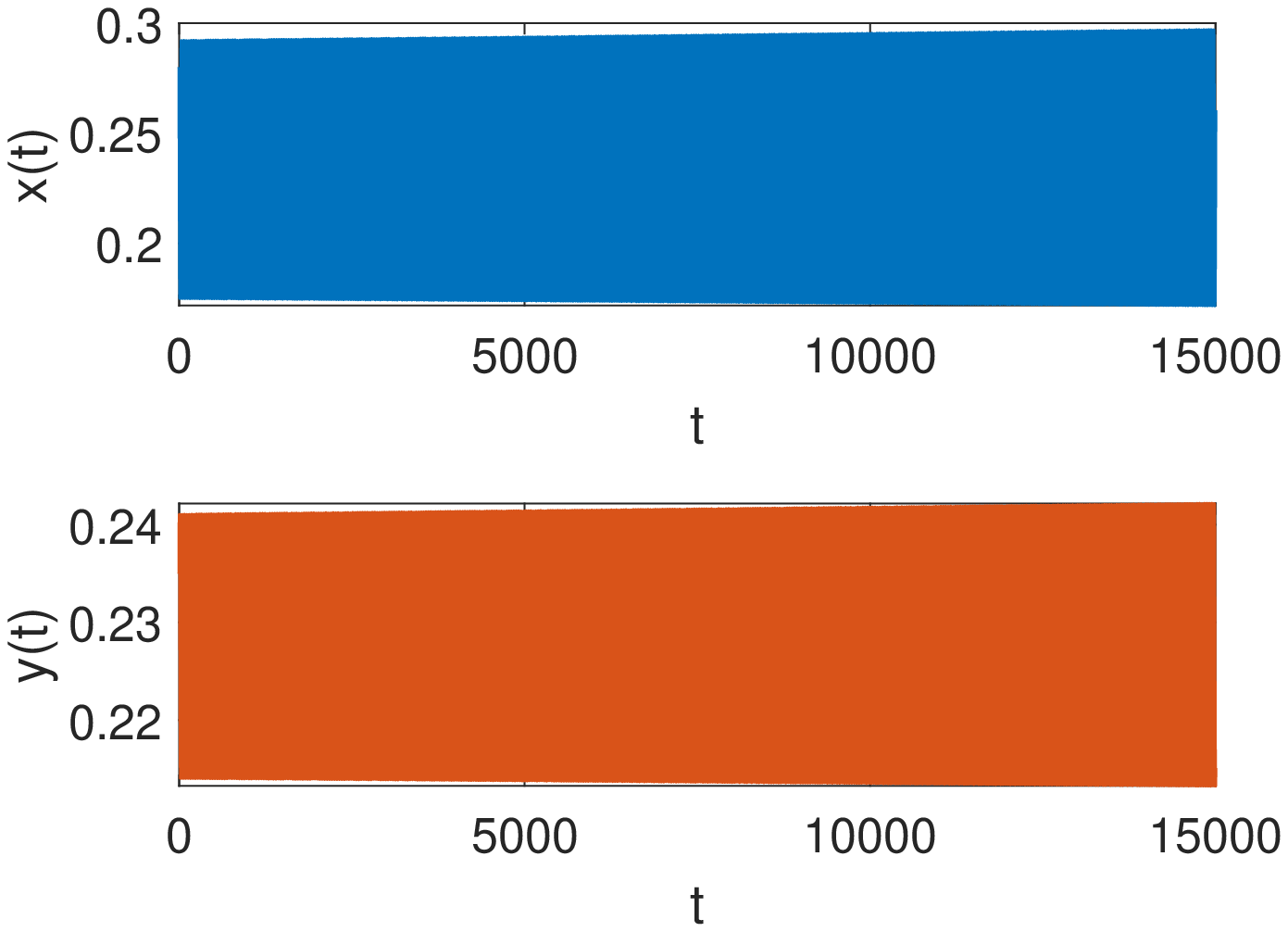}
  \includegraphics[width=0.33\textwidth, height=4.5cm]{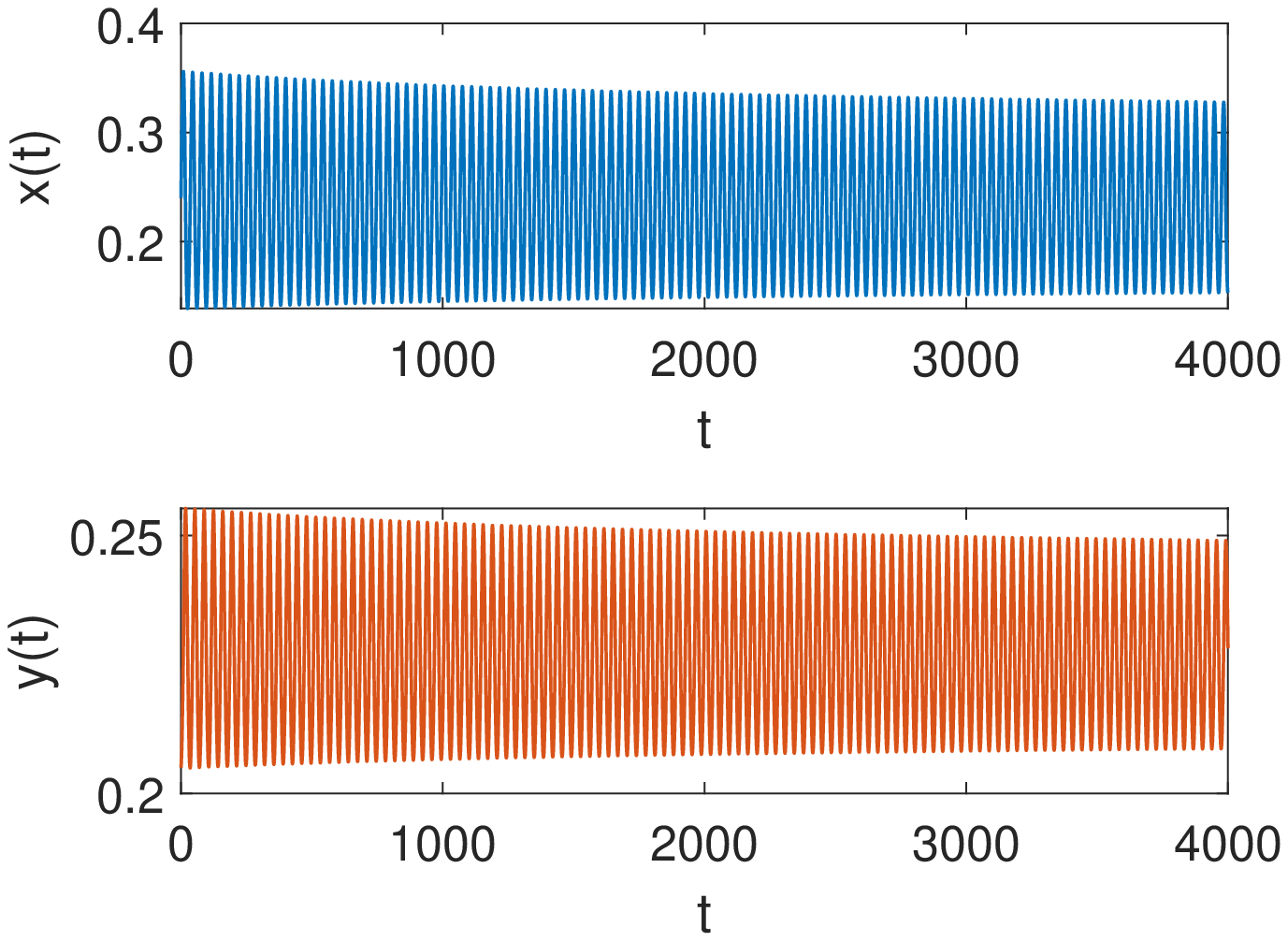}}\\
 \subfloat[Phase portrait for $a=0.2722$\label{subfig-11}]{%
  \includegraphics[width=0.33\textwidth, height=4.5cm]{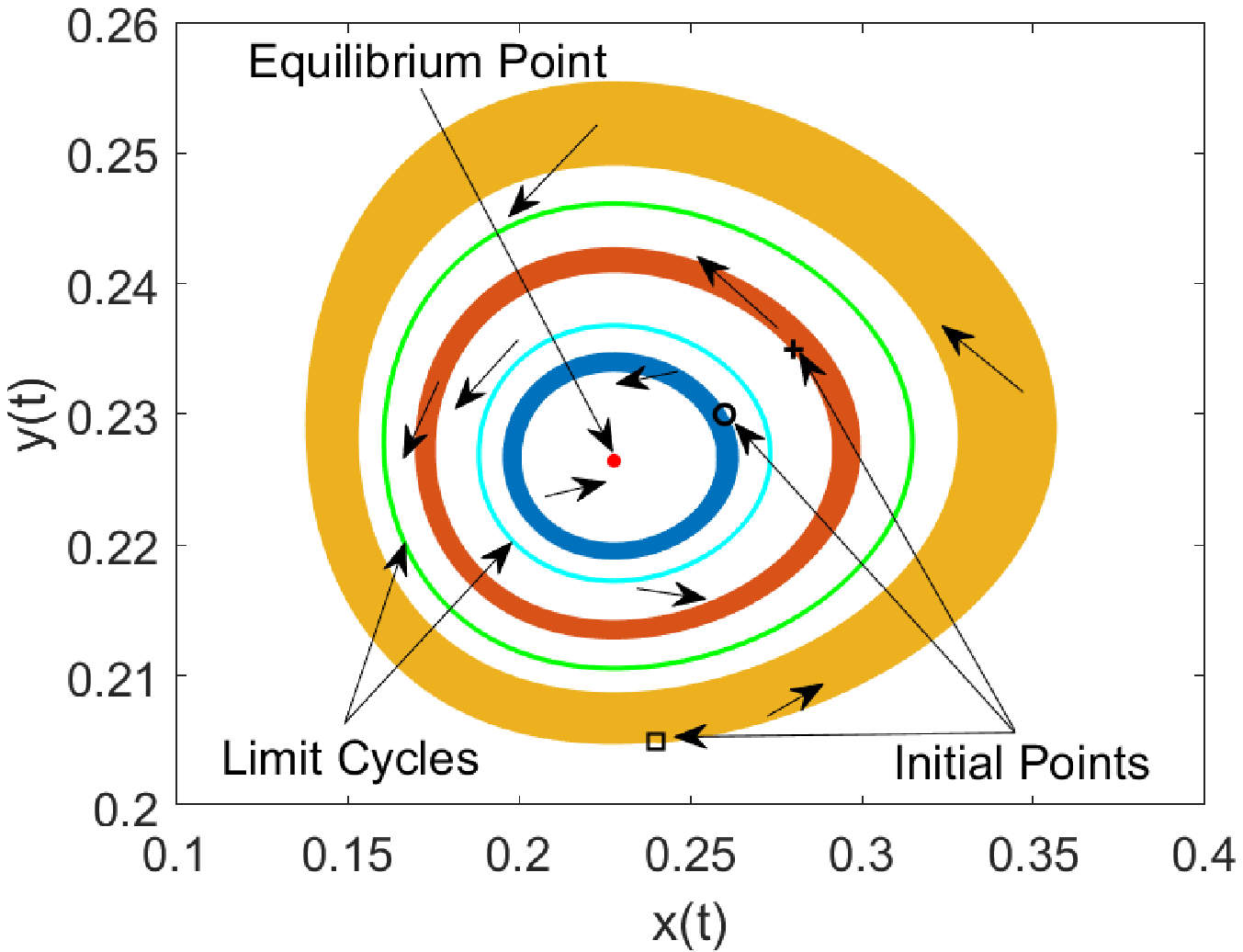}}
 \subfloat[Time series \& phase portrait for $a=0.35$\label{subfig-12}]{%
  \includegraphics[width=0.33\textwidth, height=4.5cm]{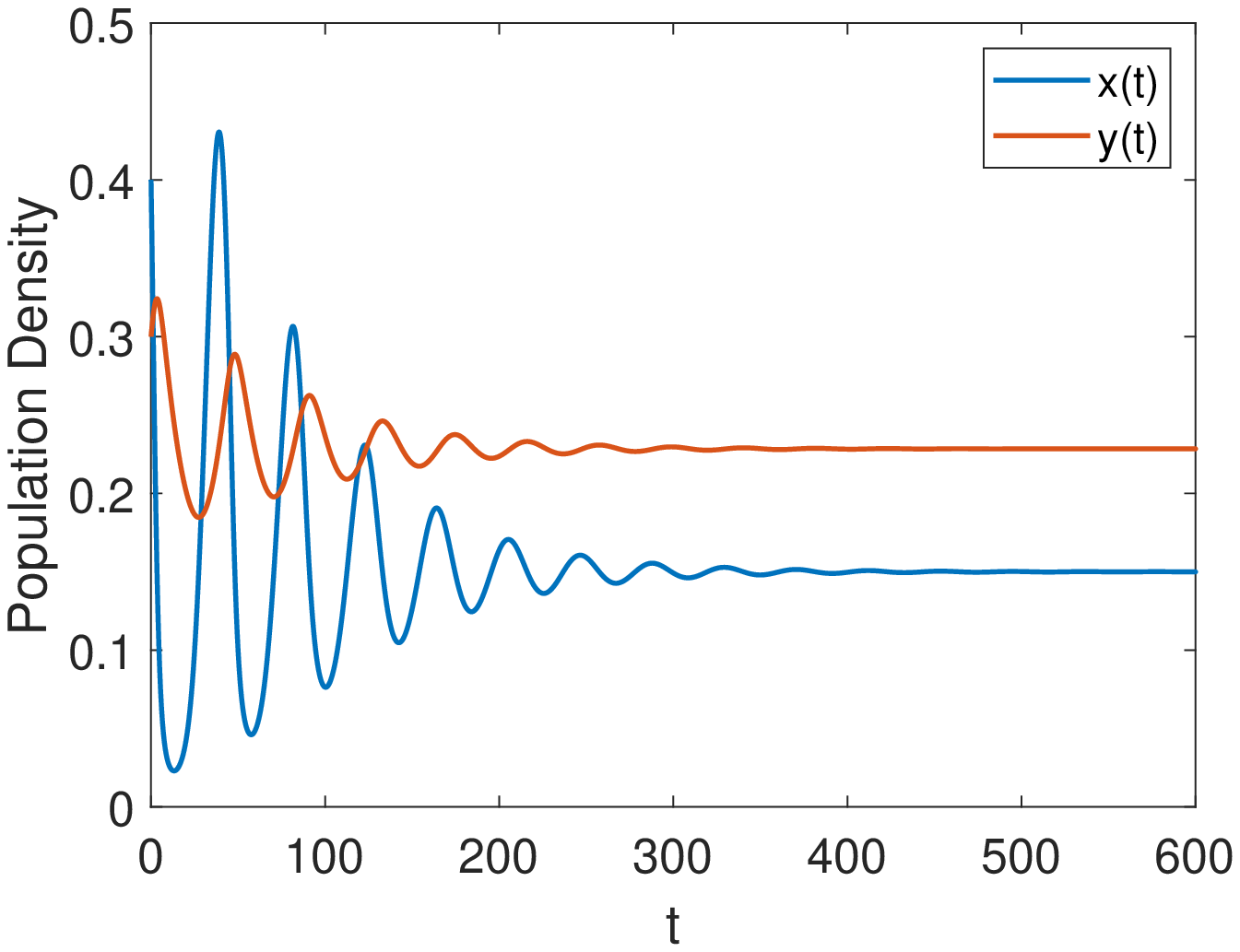}
  \includegraphics[width=0.33\textwidth, height=4.5cm]{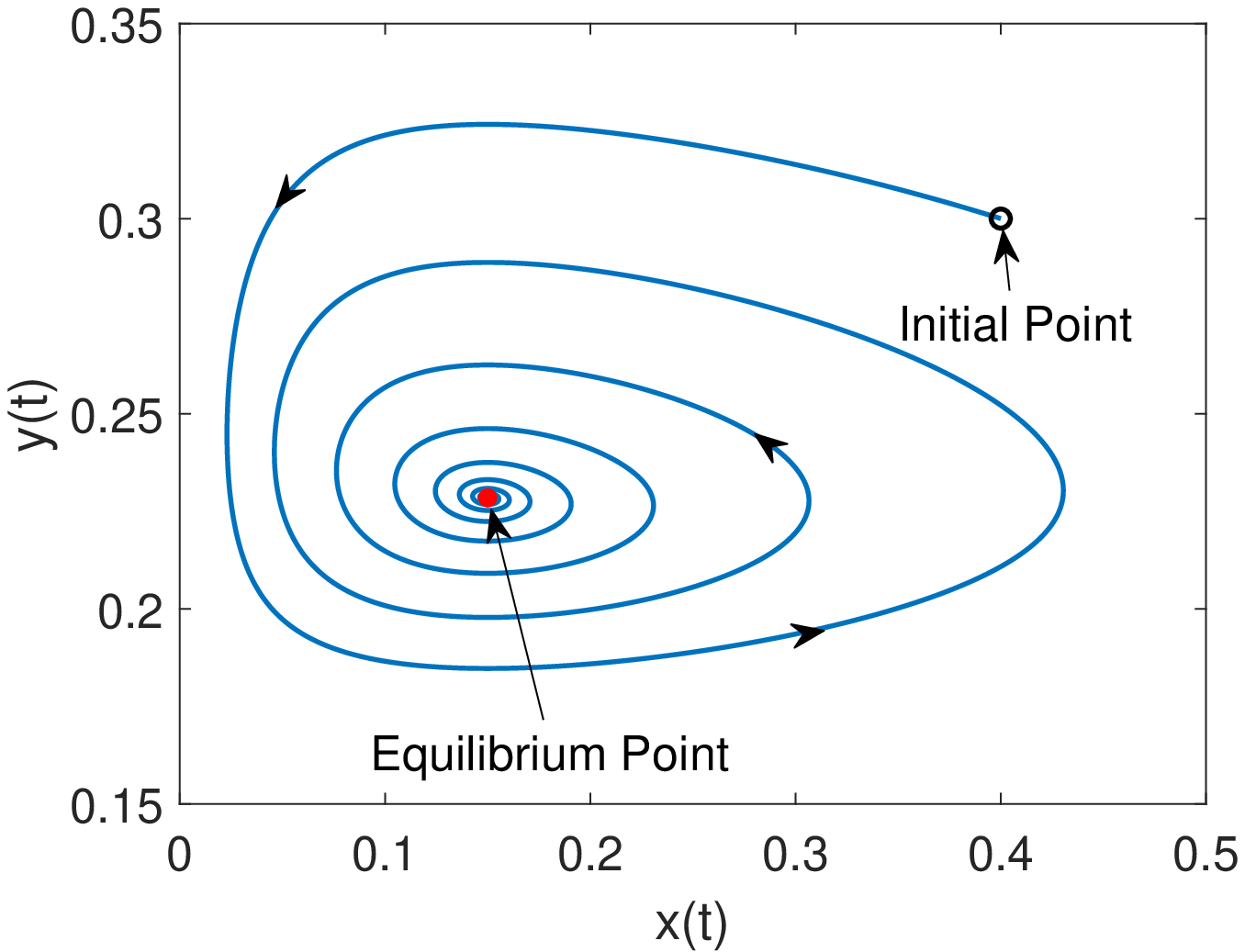}}
 \caption{{\bf(a)} The positive equilibrium of system (\ref{eq:sys}) undergoes Hopf bifurcation for $a_{H1}=0.061449, a_{H2}=0.27209046$ \& $a_{H3}=0.49999777$ when $\varrho=0$. Limit point of cycles(LPC) occurs at $a_{LPC}=0.27224979$. {\bf(b)} $E$ is stable for $a=0.04<e_{H1}$, where the initial point of the simulation is $\circ(0.4,0.3)$. {\bf(c)} $E$ is unstable and there exists a stable limit cycle for $a_{H1}<a=0.1<a_{H2}$, where the initial points of the simulation are $\circ(0.4,0.3)$ \& $\diamond(0.35,0.22)$. {\bf(d)} Time series for $a_{H2}<a=0.2722<a_{LPC}$. {\bf(e)} Phase portrait corresponding to time series plotted in (d) with initial conditions $\circ(0.26,0.23)$, $+(0.28,0.235)$ \& $\square(0.24,0.205)$, where the equilibrium point is stable, the inner limit cycle is unstable and the outer limit cycle is stable. {\bf(f)} The equilibrium point $E$ is stable for $a_{H2}<a=0.35<a_{H3}$, where the initial point is $\circ(0.4,0.3)$.}
 \label{fig_a}
\end{figure}
The point ``H" in the above bifurcation diagrams shows the parameter value for which the non-delayed system enters into Hopf bifurcation and the red dot in the above phase portrait diagrams shows the positive equilibrium point $E$.

The corresponding non-delayed system of the system (\ref{eq:sys}) does not show Hopf-bifurcation with respect to the parameter $m$.

Using sensitivity analysis, one can determine which parameters influence the model output the most or the least. Consequently, influential parameters on the model output need to be assigned accurate values while less influential parameters suffice to have a rough estimate\cite{prcc1}. In this study, partial rank correlation coefficient (PRCC), a global sensitivity analysis technique proven to be the most reliable and efficient among sampling-based methods, is utilized. The PRCC addresses the effect of changes in a specific parameter (linearly discounting the influences over the other parameters) on the reference model output\cite{prcc2}.

The PRCC deals with the impact of changes in a particular parameter on the model output \cite{prcc1}.
So to get the PRCC estimations, Latin Hypercube Sampling (LHS) is picked for the input
parameters where stratified sampling without substitution is performed. In the present study,
uniform dissemination is assigned to each model parameter and sampling is done autonomously.
The range for every parameter is at first set to $\pm 25\%$ of the nominal values (Fig. \ref{fig_prcc}). It is seen that PRCC values lie between $-0.8$ to $0.7$ for the non-delayed model. A positive value of PRCC indicates a positive correlation of the parameter with the model output, where negative values indicate the negative correlation. A positive correlation implies that a positive change in the parameter will increase the model output. Similarly, the negative correlation implies that a negative change in the parameter will decrease the model output. The PRCC values for the non-delayed model is depicted as bar graphs in Fig. \ref{fig_prcc1} and its time evolution are illustrated in Fig. \ref{fig_prcc2} respectively.

\begin{figure}[H]
 \subfloat[\label{fig_prcc1}]{%
  \includegraphics[width=0.5\textwidth, height=5cm]{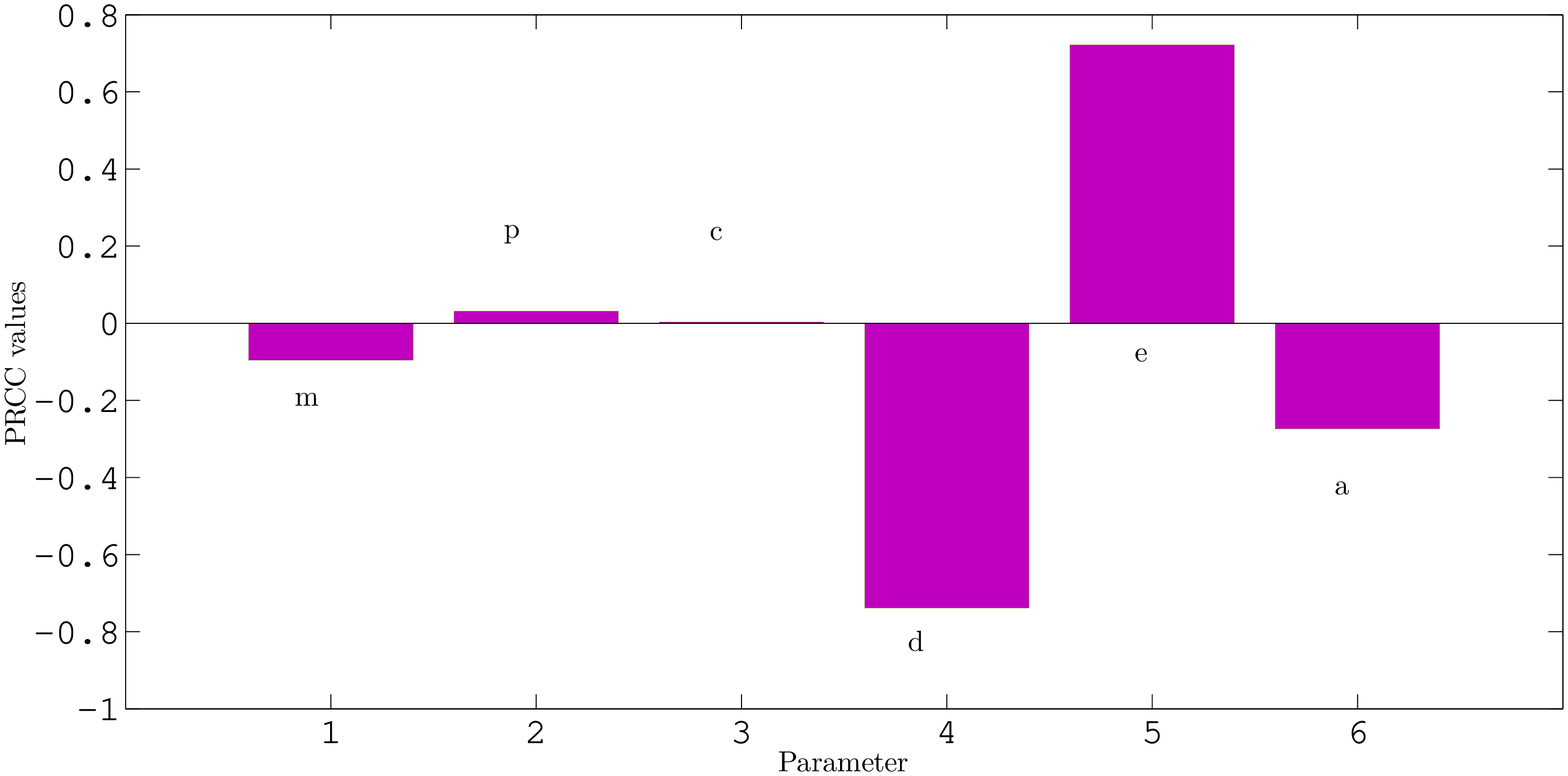}}
 \subfloat[\label{fig_prcc2}]{%
  \includegraphics[width=0.5\textwidth, height=5cm]{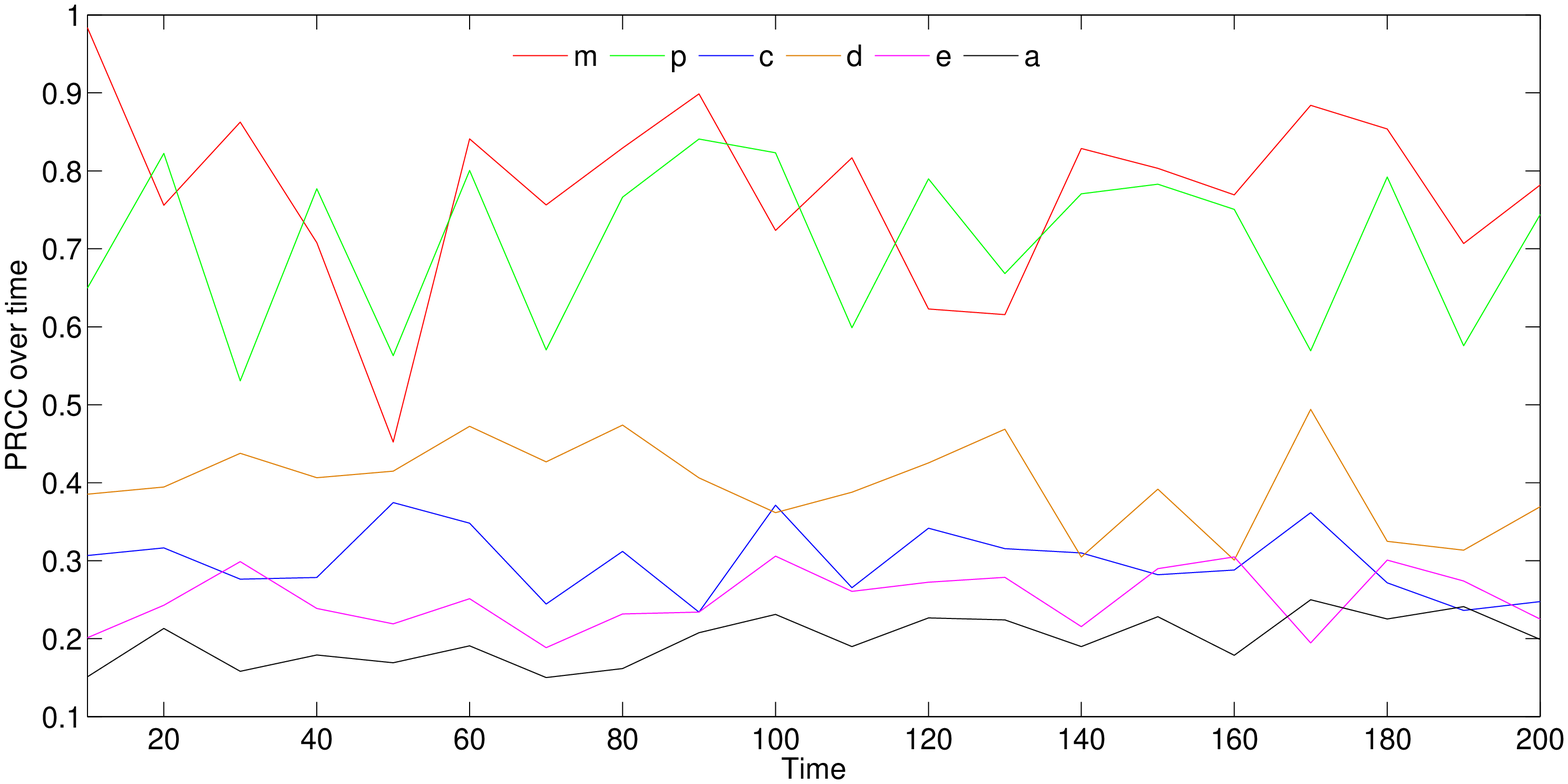}}
 \caption{In this figure the partial rank correlation coefficients (PRCCs) of the model parameters at different time points has been plotted in absence of delay. In Figure (a), PRCC of the model parameters has been shown and in Figure (b) PRCC has been plotted over time.}
 \label{fig_prcc}
\end{figure}
\section{Feedback Control}\label{secind}
In ecology, control mechanisms are adopted for population and resource management. A desired population dynamic(density) can be achieved by controlling the population abundance. Sometimes control is applied to both predator and prey abundances while sometimes a single species is controlled(for example, prey) which indirectly controls the other species(predator)\cite{nadim2018}. In a direct way, state control can be implemented to control the interacting populations by addition or removal(harvest) of species individuals\cite{loehle2006}.

The primary goal of using control is to maintain balance in the habitat by controlling population extinction or overpopulation. Drastic oscillations can pose an extinction threat to both species and may not lead to a co-existence state\cite{nadim2018,zhang2016}. So controlling such oscillations and driving the dynamics towards an equilibrium point, we can remove the risk.

In this section, we will introduce feedback control to the system to control instability or change the dynamics of the system. Let $u$ be the control parameter. With the use of linear feedback control, system (\ref{eq:sys}) is modified as
\begin{align}\label{controlsys}
&\frac{dx}{dt}\;\,=x(1-x(t-\varrho))-\frac{mxy}{x^p+c}-u(x-x^*), \nonumber\\
&\frac{dy}{dt}\;\,=\left(d-\frac{e}{x+a}\right)y^2-u(y-y^*).
\end{align}
where $(x^*,y^*)$ is interior equilibrium point of the system (\ref{eq:sys}) which is equivalent to system (\ref{controlsys}) when $u=0$. Now, $(x^*,y^*)$ is also the interior equilibrium point of system (\ref{controlsys}).

Now, translating the interior equilibrium point of the system to origin and linearizing the system around it, let $A$ be the obtained Jacobian matrix, then
\begin{equation}\label{controljac}
A=\begin{bmatrix}
A_{11}-x^*\e^{-\lambda\varrho}-u & -A_{12}\\
A_{21} & -u
\end{bmatrix}
\end{equation}
where $$A_{11}=\frac{mpy^*x^{*^p}}{(x^{*^p}+c)^2},\quad A_{12}=\frac{mx^*}{x^{*^p}+c},\quad \text{and}\quad A_{21}=\frac{d^2y^{*^2}}{e}.$$
Hence, its characteristic equation is given by,
\begin{equation}\label{controlchar}
H(\lambda)\equiv\lambda^2+a_1\lambda+a_2+\e^{-\lambda\varrho}(b_1\lambda+b_2)=0
\end{equation}
where
$$a_1=2u-A_{11},\quad a_2=A_{12}A_{21}+u^2-A_{11}u,$$
$$b_1=x^*,\quad b_2=ux^*.$$
Studying the nature of the roots of (\ref{controlchar}), we can have the stability property of the controlled system.\\
\underline{\bf Case - 1}:
Let $\varrho=0$. Then the characteristic equation reduces to
\begin{equation}\label{controlchar0}
\lambda^2+(a_1+b_1)\lambda+(a_2+b_2)=0.
\end{equation}
From (\ref{controlchar0}), the Hurwitz matrix is given by
\begin{equation*}
A=\begin{bmatrix}
a_1+b_1 & 1\\
0 & a_2+b_2
\end{bmatrix}
\end{equation*}
whose minors are positive when $a_1+b_1>0$ and $a_2+b_2>0$. These give rise to the conditions
\begin{eqnarray}\label{delay0cond}
u>\frac{1}{2}(A_{11}-x^*),\nonumber\\
u^2-(A_{11}-x^*)u+A_{12}A_{21}>0.
\end{eqnarray}
Therefore, $(x^*,y^*)$ is stable when inequalities in (\ref{delay0cond}) are satisfied.

\begin{figure}[H] 
 \subfloat[$u=0$\label{u_nodelay}]{%
  \includegraphics[width=0.5\textwidth, height=5cm]{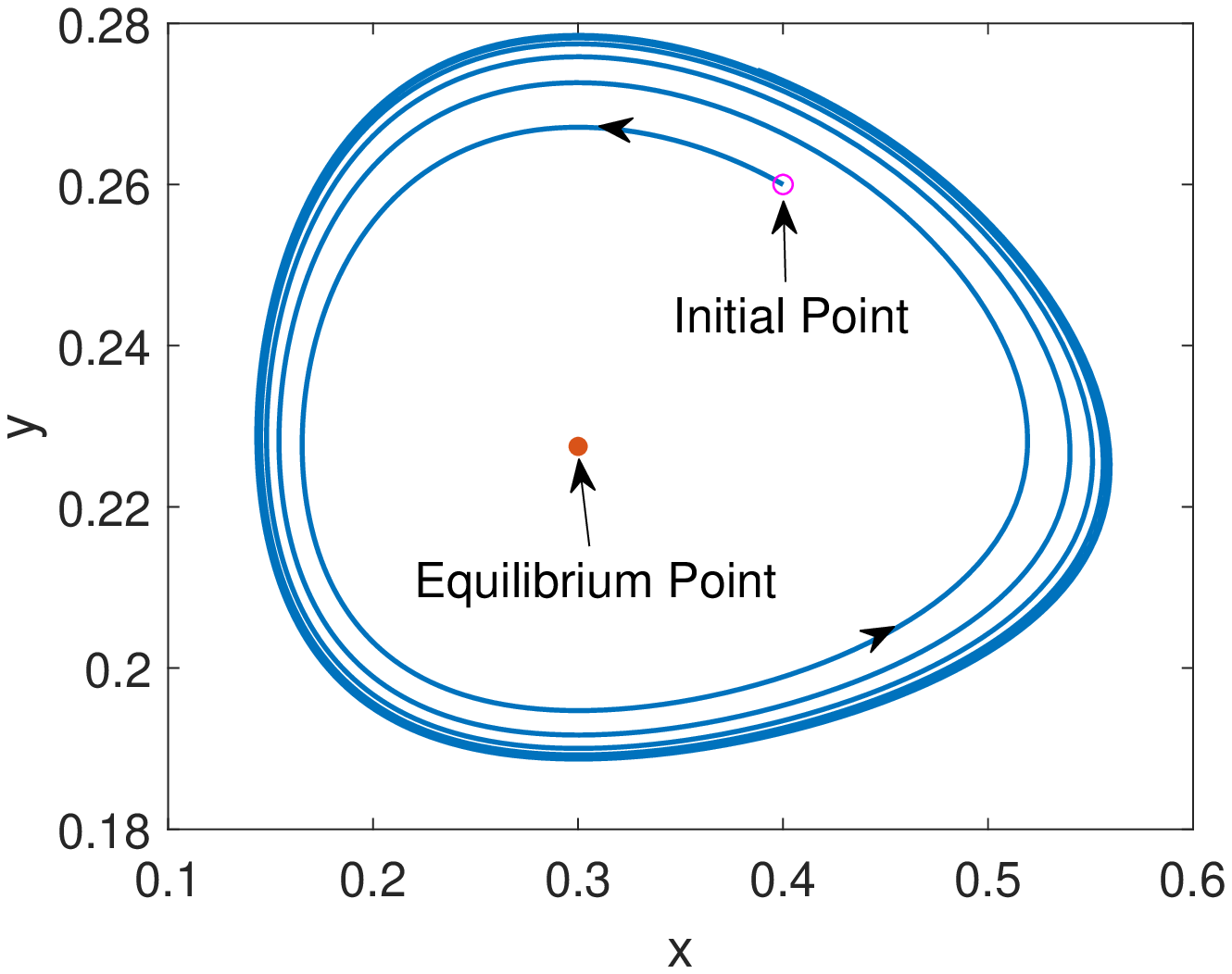}}
 \subfloat[$u=0.5$]{%
  \includegraphics[width=0.5\textwidth, height=5cm]{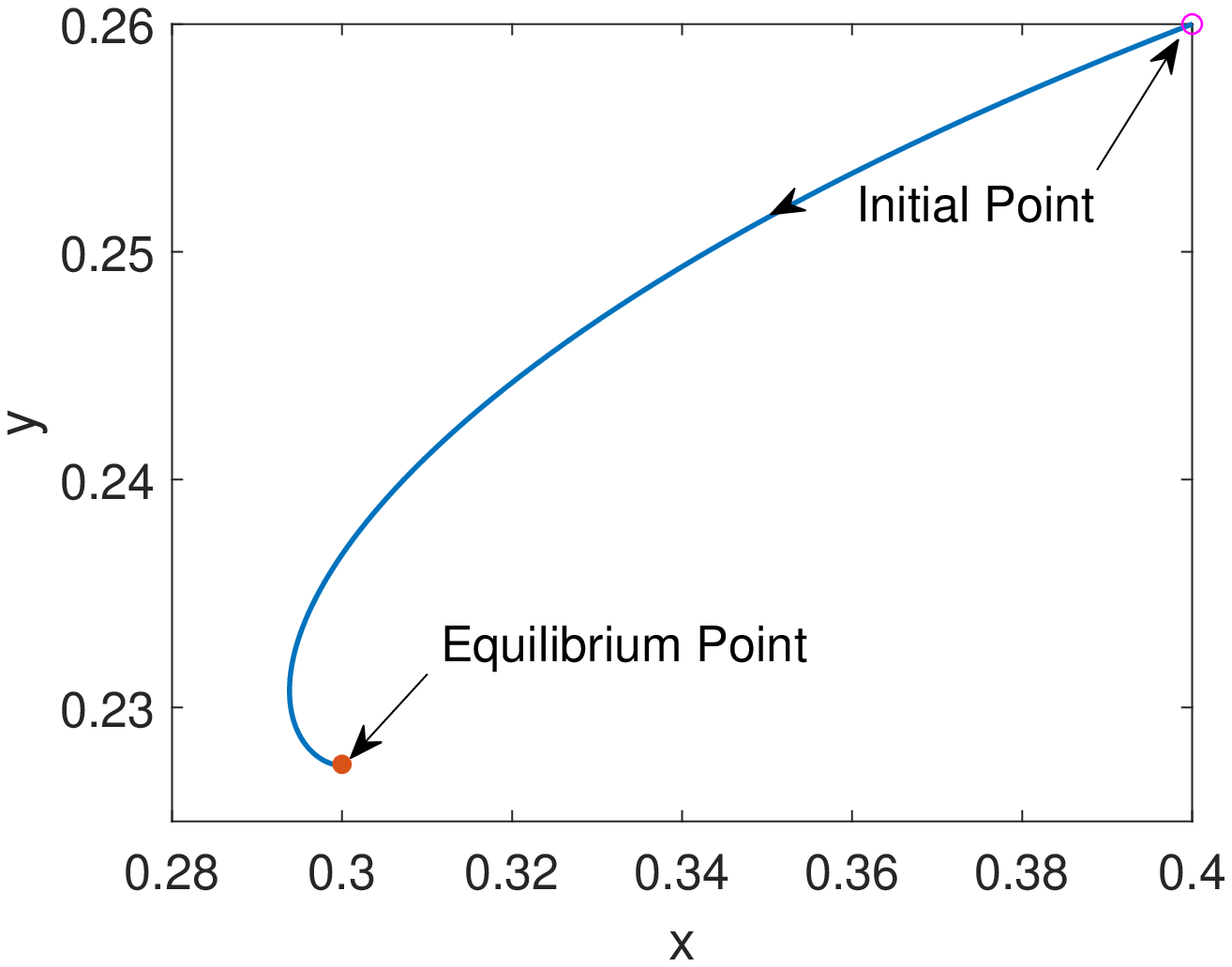}}
 \caption{The initial point of the simulation is $(0.4,0.26)$ and the equilibrium point is $(0.3,0.1108)$. (a) The equilibrium point is unstable for $u=0$. (b) The equilibrium point becomes asymptotically stable for $u=0.5$.}
 \label{nodelay_u_val}
\end{figure}

With the help of MATLAB, considering the value of $u$ in the interval $[-10,10]$, we observe, conditions in (\ref{delay0cond}) are satisfied for $u\ge 0.0115$ when the system parameters are chosen as $m=1.2$, $p=2$, $c=0.3$, $d=0.4$, $e=0.2$ and $a=0.2$ in the absence of delay. Hence, the positive equilibrium point becomes stable for $u\ge 0.0115$.\\
\underline{\bf Case - 2}:
Now, let $\varrho\neq 0 $. Let $\lambda=i\omega$, $i=\sqrt{-1}$, then from (\ref{controlchar}),
\begin{eqnarray*}
\Real H(i\omega) &=& -\omega^2+x^*\omega\sin{(\omega\varrho)}+u^2-A_{11}u+A_{12}A_{21}+x^*u\cos{(\omega\varrho)},\\
\Imag H(i\omega) &=& (2u-A_{11})\omega+x^*\omega\cos{(\omega\varrho)}-x^*u\sin{(\omega\varrho)}.
\end{eqnarray*}
From \cite{erbe1986} 
and Nyquist criterion\cite{bb2006}, the conditions of local asymptotic stability are given by
\begin{eqnarray}\label{omega0}
\Real H(i\omega_0)=0\\
\Imag H(i\omega_0)>0
\end{eqnarray}
where, $H(\lambda)=0$ is the characteristic equation given in (\ref{controlchar}) and $\omega_0$ is the smallest positive root of equation (\ref{omega0}).
Therefore, in the presence of delay, the asymptotic stability of $(x^*,y^*)$ is guaranteed when
\begin{flalign}\label{controlstabcond}
u^2-A_{11}u+A_{12}A_{21}+x^*u\cos{(\omega_0\varrho)}-\omega_0^2+x^*\omega_0\sin{(\omega_0\varrho)}=0,\nonumber\\
(2u-A_{11})\omega_0+x^*\omega_0\cos{(\omega_0\varrho)}-x^*u\sin{(\omega_0\varrho)}>0
\end{flalign}
hold simultaneously.
\begin{figure}[H] 
 \subfloat[\label{u_delay_a}]{%
  \includegraphics[width=0.5\textwidth, height=5cm]{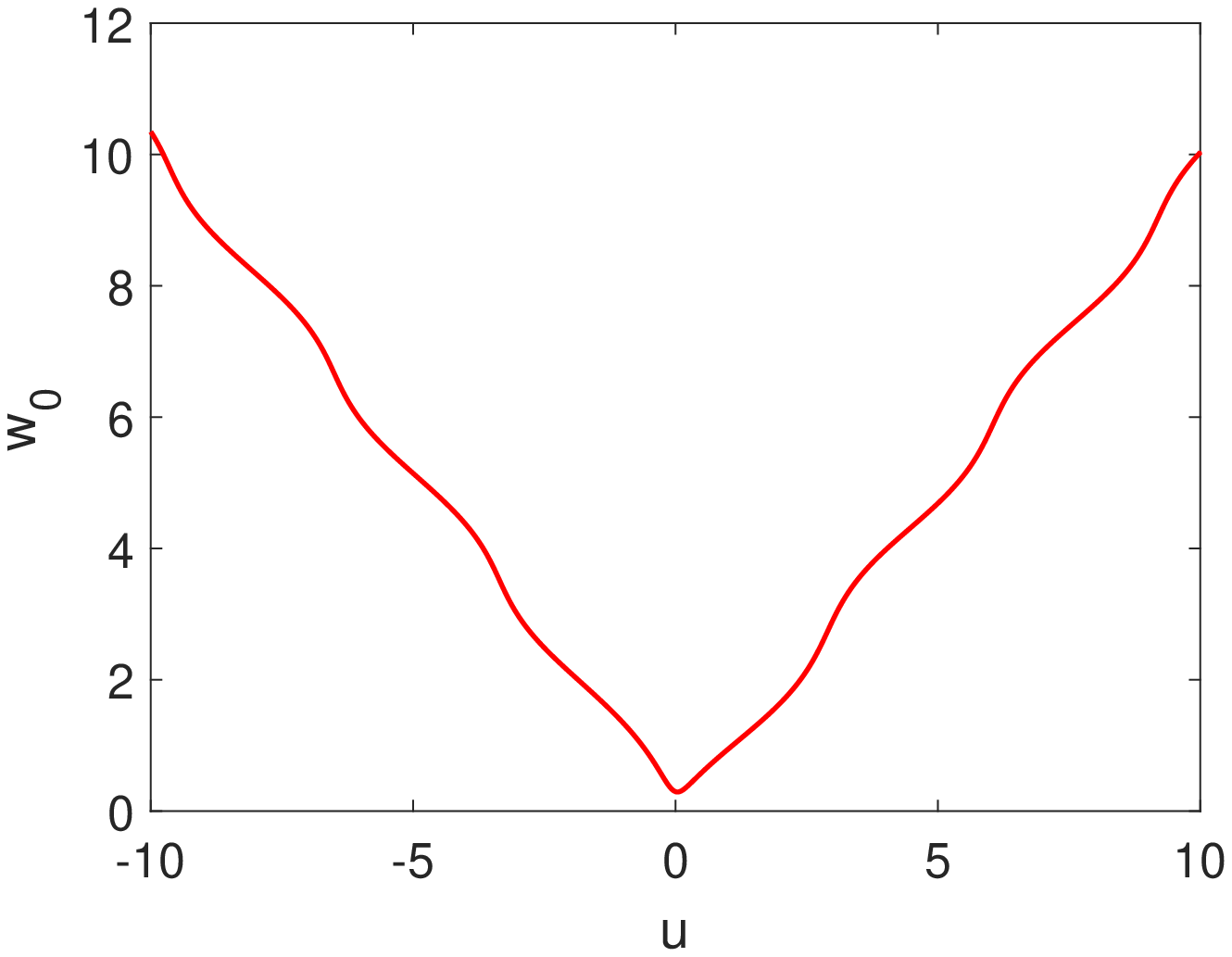}}
 \subfloat[]{%
  \includegraphics[width=0.5\textwidth, height=5cm]{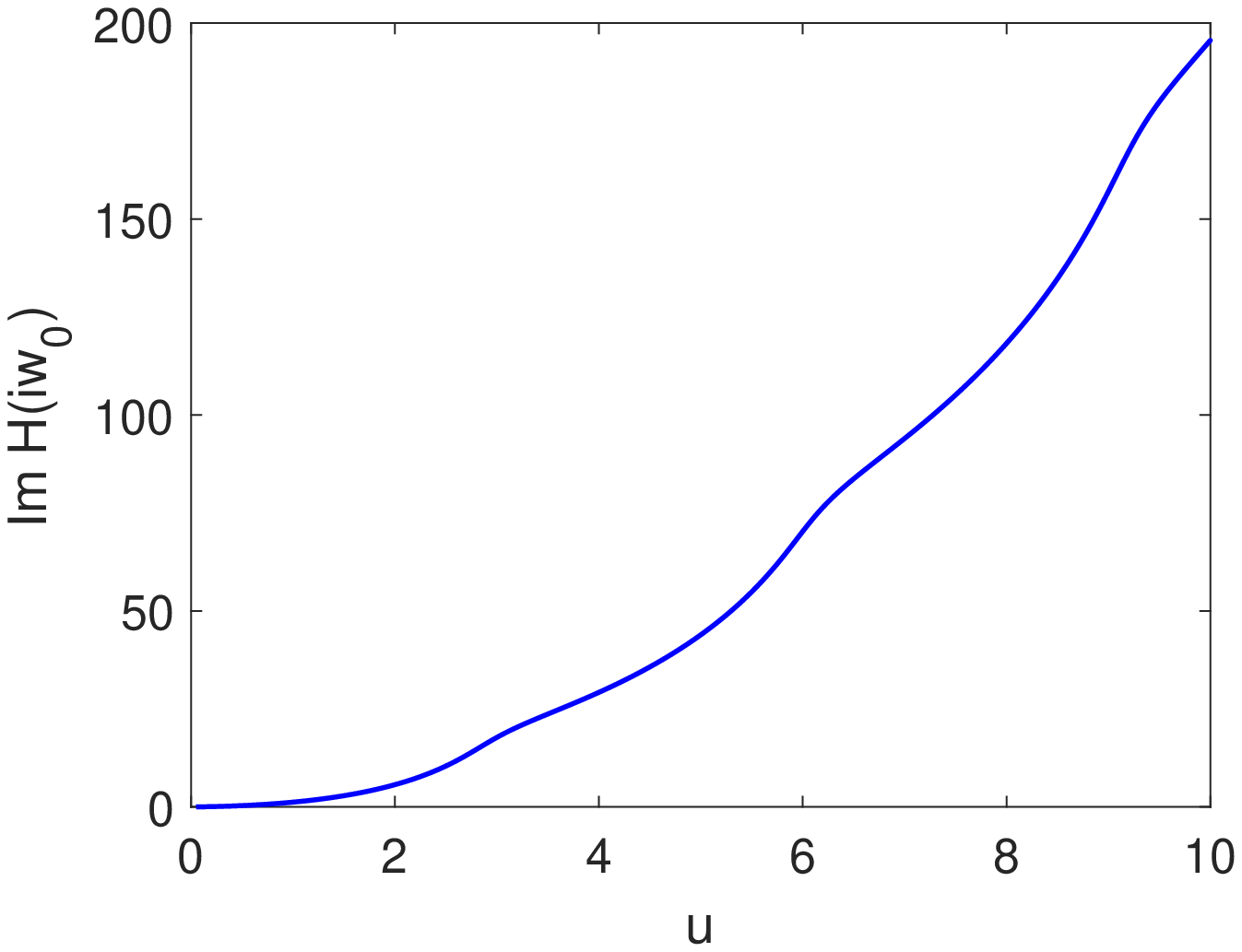}}
 \caption{(a) Control parameter $u$ versus positive zero of $\Real H(i\omega_0)$. (b) Control parameter $u$ versus $\Imag H(iw_0)$ for positive zeros obtained in (a).}
 \label{u_delay}
\end{figure}
The figure \ref{u_delay_a} shows the positive roots of $\Real H(iw_0)=0$ for different values of $u$. The numerical calculations are done using MATLAB. The calculations show that for $u\geq 0.051$, the positive equilibrium point of the controlled system becomes asymptotically stable which was earlier unstable. The parameters used here are $m=1.2$, $p=2$, $c=0.3$, $d=0.4$, $e=0.2$ and $a=0.2$. Time delay for the system is considered as $\varrho=2$.
\begin{figure}[H] 
 \subfloat[$u=0$\label{u_delay}]{%
  \includegraphics[width=0.5\textwidth, height=5cm]{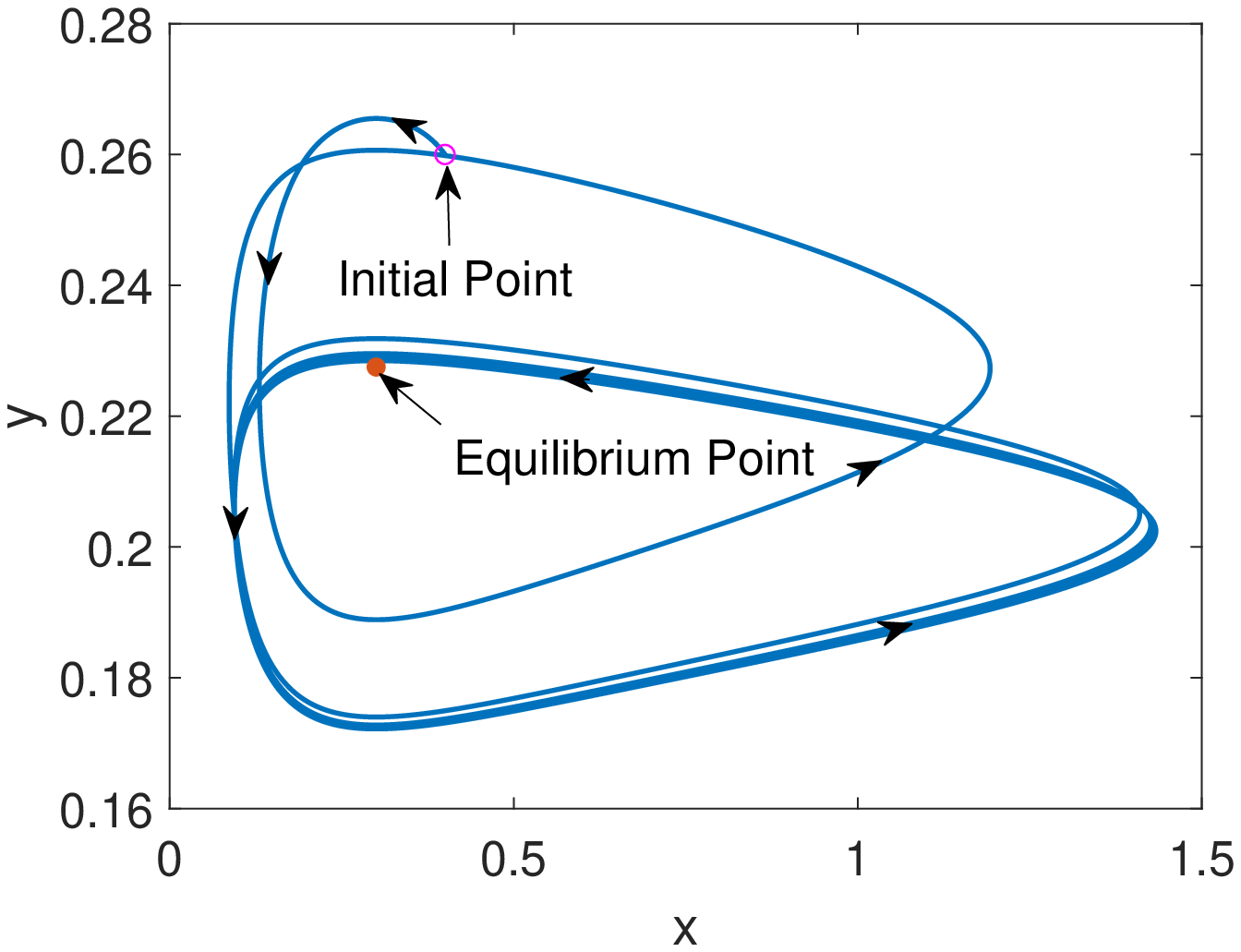}}
 \subfloat[$u=0.04$]{%
  \includegraphics[width=0.5\textwidth, height=5cm]{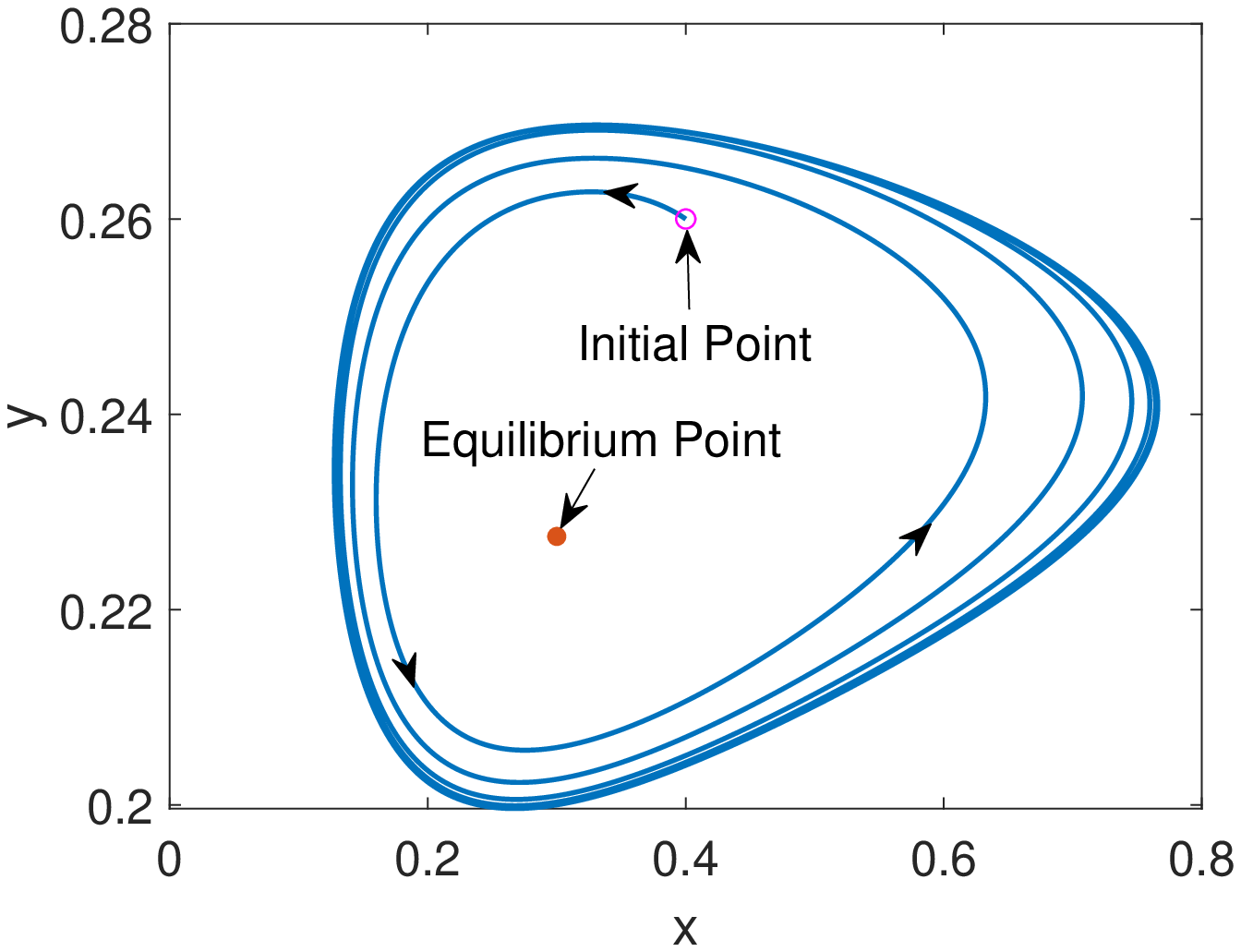}}\\
 \subfloat[$u=1$\label{u_delay}]{%
  \includegraphics[width=0.5\textwidth, height=5cm]{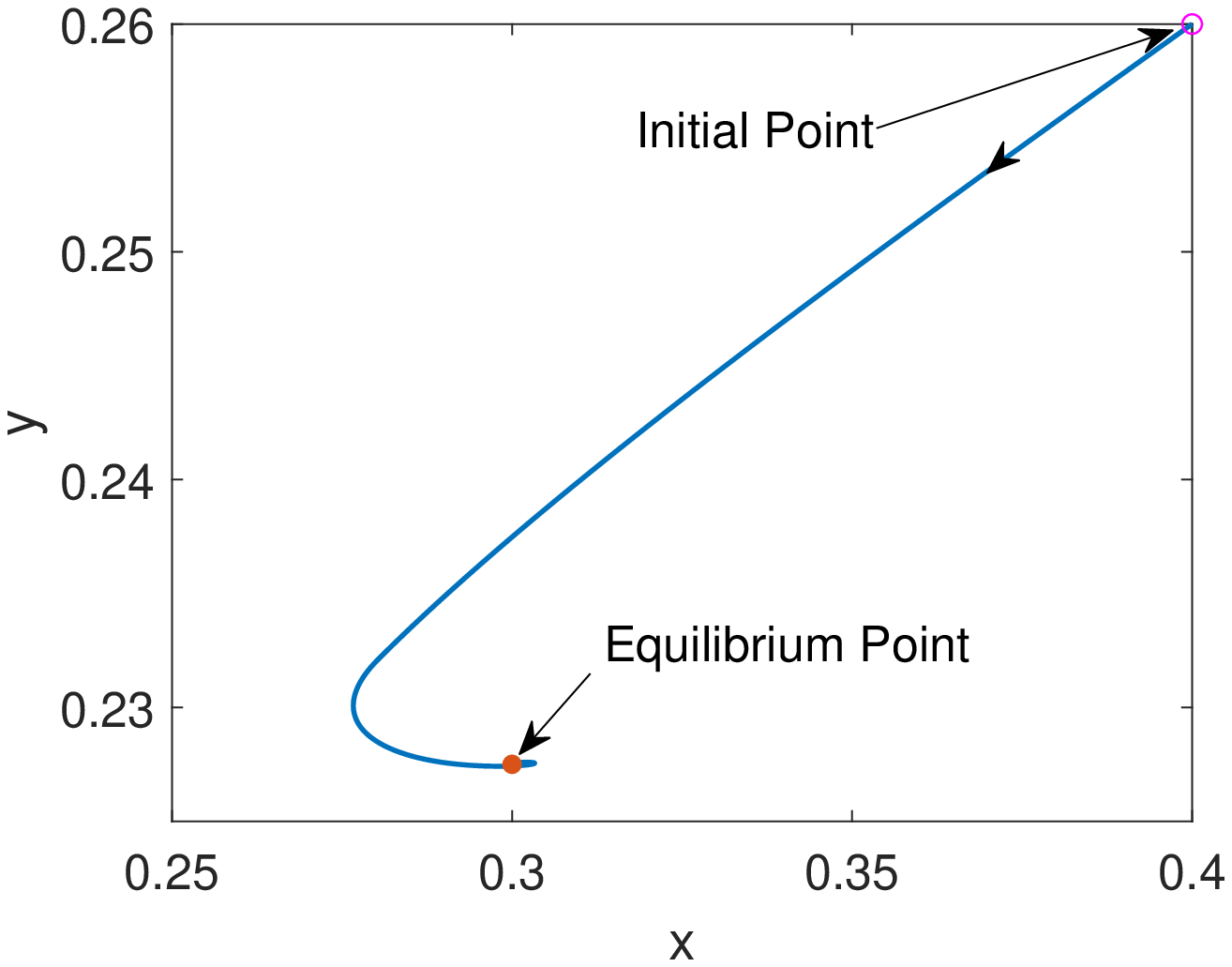}}
 \subfloat[$u=10$]{%
  \includegraphics[width=0.5\textwidth, height=5cm]{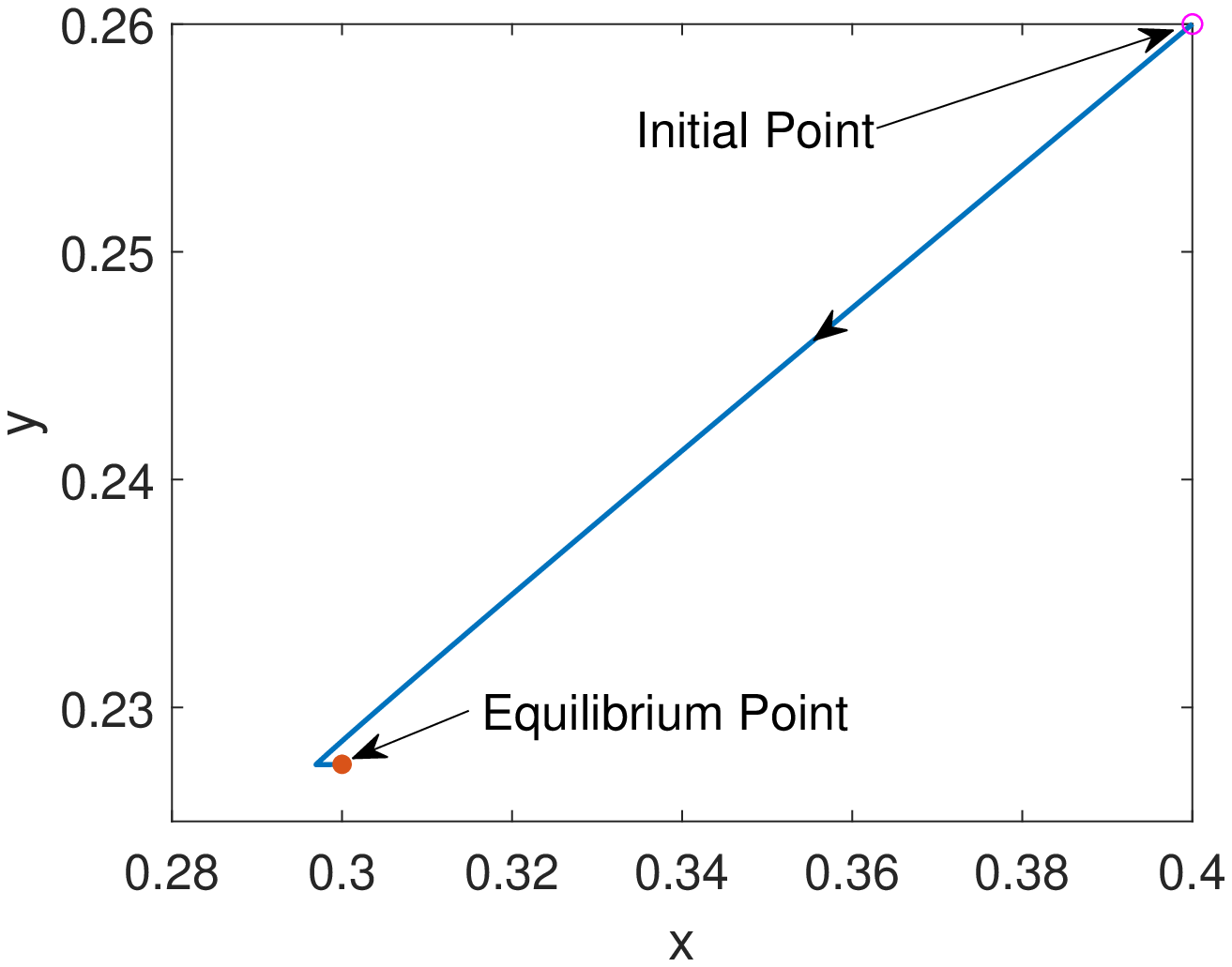}}
 \caption{The initial point of the simulation is $(0.4,0.26)$ in each figure and the equilibrium point is $(0.3,0.1108)$. (a) The equilibrium point is unstable for $u=0$. (b) The equilibrium point is unstable for $u=0.04$. (c) The equilibrium point becomes asymptotically stable for $u=1$. (d) The equilibrium point becomes stable for $u=10$.}
 \label{delay_u_val}
\end{figure}

\section{Conclusion and Discussion} \label{seccon}
The population of generalist predators is greatly affected by the absence of its favourite prey. As an example, in most western and central Europe, the pigeons form the important prey base of most goshawk populations. Due to the scarcity of its preferred prey, it switches to alternative prey. However, in the absence of environmental changes like deforestation or a significant number of haunting, goshawks exhibit stability in breeding numbers\cite{rutz2006}.

In this paper, a two species competition model is considered and investigated along with the assumption that the prey species shows delayed logistic growth and the predator species is of generalist type. Food switching, when prey is scarce, is one of the most common behaviours of the generalist predators. The predator's functional response is modelled by a functional response which shows the grouping behaviour of a prey species in defence against predation. The stability and bifurcation viewpoints may influence the decision related to the management of ecological resources. Here we have obtained the hopf-bifurcation corresponding to the protection provided to the prey species by the environment($c$), the reproduction rate of the generalist predators($d$), maximum rate of death of predators($e$) and residual loss of predator species when there is severe scarcity of prey($a$). If the reproduction rate of the prey($d$) lies between 0.3131 \& 0.4674, a stable limit cycle ensures the oscillation in the numbers of prey and predators. Due to seasonal change and other reasons, $c$(environmental protection provided to the prey species), $ d $(the reproduction rate of the predators), $ a $(residual loss of predator species due to low prey abundance) and $ m $(maximum death rate of the predators) may vary.

Throughout the paper, we have discussed the stability of the positive equilibrium state $E(x^*,y^*)$ of the system (\ref{eq:sys}) as this represents the co-existence of both prey and predator populations. By constructing suitable Lyapunov function $w(v)(t)$, it is observed that the interior equilibrium state is locally asymptotically stable if the conditions
\begin{eqnarray}
\varrho>\pi_o\quad \text{and}\quad \pi_1\varrho^2+\pi_2\varrho+\pi_3>0,
\end{eqnarray}
are satisfied by the parameters in the system (\ref{eq:23}).

The bifurcation point $\varrho=\varrho_0^*$ for Hopf-bifurcation are obtained by considering logistic delay $\varrho$ as the bifurcation parameter.

As the presence of delay has the ability to destabilize a stable ecosystem, so estimation of the length of the delay is very important for system stability. We observed that the delay length is estimated to be within $0\leq \varrho\leq\varrho_+$ for the system (\ref{eq:sys}) to possess local asymptotic stability.

Along with the time series and phase portrait plots for different $\varrho$ values, numerical simulations on bifurcation in the parameters of the system (\ref{eq:sys}) for $\varrho=0$ are also presented in the numerical analysis section in \ref{secnum}. The Hopf-bifurcation diagrams for each parameter are plotted. The system shows Hopf bifurcation at the coexistence state for $c=c_H$ i.e, when the environmental protection provided to the prey species is below a threshold limit($c_H$), the system shows periodic oscillations. The study of the system through a change in the growth rate of the predators($d$), the maximum death rate of the predators($e$) and residual loss of predators when preys are scarce($a$) shows that the system undergoes LPC-bifurcation once and Hopf bifurcation three times. Among three Hopf points, a supercritical Hopf bifurcation occurs for one Hopf point($d=d_{H1}$ or $e=e_{H3}$ or $a=a_{H1}$), a subcritical Hopf bifurcation occurs for another Hopf point($d=d_{H2}$ or $e=e_{H2}$ or $a=a_{H2}$) and the other Hopf bifurcation occurs for the limiting values of the parameters($d\approx 1$ or $e\approx 0.08$ or $a\approx 0.5$). Two limit cycles originate from two Hopf points, which goes on further to collide and disappear at the LPC-point. For certain ranges of the parameters($d_{H2}<d<d_{LPC}$ or $e_{LPC}<e<e_{H2}$ or $a_{H2}<a<a_{LPC}$), two limit cyolorcles appear around the stable equilibrium point among which the inner limit cycle is unstable and the outer limit cycle is stable. In such cases, the rate of convergence of the population densities to the stable states are so small that a long time observation of the populations may seem to be periodic.

The occurrence of limit point bifurcation of cycles gives rise to different stable and unstable states with non-equilibrium states. In such situation, the long term behaviour of the populations depends on the initial population of the species due to the presence of alternate stable states. So, with the same environmental and parametric conditions, the long term population dynamics may show oscillation or may be steady. A small change in the parametric values near the LPC-point can bring severe change in the stability behaviour of the populations as for values lying on one side of the LPC-point will produce steady-state stabilization, and for values on the other side of the LPC-point will show stabilization on alternate stable states(steady-state and limit cycle solution) or on an oscillatory state.

A linear feedback control method is used to to bring back the system to a stable state when the positive equilibrium point of the system becomes unstable. Using local asymptotic stability property, the stability conditions involving the control parameter $u$ are obtained. Using numerical techniques in MATLAB, the range of the control parameter is obtained for which the co-existence equilibrium point becomes locally asymptotically stable.

Time series are designed to examine how the changes in the parameter affect the model output. Therefore, the PRCCs of the model output at a particular instance are obtained with respect to each parameter. With the help of PRCC analysis, we can identify the parameter(s) which is/are most sensitive.
Here as per the Fig \ref{fig_prcc}, it is seen that the parameters d and e are most sensitive for our model, where $d$ and $e$ are the normalized form of parameters $D$ and $E$ respectively representing the reproduction rate and maximum death rate of the generalist predators.

Prey switching is a natural phenomenon in generalist species. Also, prey switching helps in stabilizing prey populations as relatively scarce prey species are freed from predation\cite{jaworski2013} due to the predator's choice of highly abundant prey over prey species of lower abundance.

We assume our system posses a positive equilibrium point. The bifurcation diagrams in section 7, without maturation delay in prey, show that with variations in many system parameters, the populations co-exist due to the behavioural contribution of generalist predators. From PRCC results, we observe that the system is highly sensitive to the growth rate of the generalist predators. The predator species has a stabilizing effect over the prey species at lower growth rates. And for relatively higher growth rates the populations show periodic oscillation.

In case when there is a maturation delay present in prey species, lower delay values give rise to steady-state but higher delay values produce periodic populations. When there is a larger oscillation in population densities, which can pose a risk of prey extinction, the situation can be controlled by the limited removal of both prey and predator biomass from the ecosystem.
\section*{Funding}
Not applicable.
\section*{Conflict of interest/Competing interests}
The authors declare that they have no conflict of interest.
\section*{Availability of data and material (data transparency)}
Not applicable
\section*{Code availability}
Not applicable.


\begin{thebibliography}{99}

\bibitem{van2010} Van Driesche, R. G., Carruthers, R. I., Center, T., Hoddle, M. S., Hough-Goldstein, J., Morin, L., ... \& Van Klinken, R. D. (2010). Classical biological control for the protection of natural ecosystems. Biological control, 54, S2-S33.

\bibitem{causton2001} Causton, C. E. (2001). Dossier on Rodolia cardinalis Mulsant (Coccinellidae: Cocinellinae), a potential biological control agent for the cottony cushion scale, Icerya purchasi Maskell (Margarodidae). Charles Darwin Research Station, Galápagos Islands.

\bibitem{symondson2002} Symondson, W. O. C., Sunderland, K. D., \& Greenstone, M. H. (2002). Can generalist predators be effective biocontrol agents?. Annual review of entomology, 47(1), 561-594.

\bibitem{holling1959} Holling, C. S. (1959). The components of predation as revealed by a study of small-mammal predation of the European pine sawfly.

\bibitem{bedd} Beddington, J.R. (1975). Mutual interference between parasites or predators and its effect on searching efficiency, J. Animal Ecol., 44 331–340.

\bibitem{dea} DeAngelis, D.L., Goldstein, R.A. \& O'Neill, R.V. (1975). A model for trophic interaction Ecology, 56 881–892.

\bibitem{arditi1989} Arditi, R., \& Ginzburg, L. R. (1989). Coupling in predator-prey dynamics: ratio-dependence. Journal of theoretical biology, 139(3), 311-326.

\bibitem{hassell} Hassell, M.P. \& Varley, G.C. (1969). New inductive population model for insect parasites and its bearing on biological control. Nature, 223(5211), pp.1133-1137.

\bibitem{tian2011} Tian, X., \& Xu, R. (2011). Global dynamics of a predator-prey system with Holling type II functional response. Nonlinear Analysis: Modelling and Control, 16(2), 242-253.

\bibitem{liu2019} Liu, H., Zhang, K., Ye, Y., Wei, Y., \& Ma, M. (2019). Dynamic complexity and bifurcation analysis of a host–parasitoid model with Allee effect and Holling type III functional response. Advances in Difference Equations, 2019(1), 1-20.

\bibitem{li2017} Li, T. T., Chen, F. D., Chen, J. H., \& Lin, Q. X. (2017). Stability of a stage-structured plant-pollinator mutualism model with the Beddington-DeAngelis functional response. Nonlinear Funct. Anal, 2017.

\bibitem{gakkhar2002} Gakkhar, S. (2002). Chaos in three species ratio dependent food chain. Chaos, Solitons \& Fractals, 14(5), 771-778.

\bibitem{ioannou2017}
Ioannou, C. C. (2017). Grouping and predation. Encyclopedia of evolutionary psychological science, 1-6.

\bibitem{cosner} Cosner, C., DeAngelis, D.L., Ault, J.S. \& Olson, D.B. (1999). Effects of spatial grouping on the functional response of predators. Theoretical population biology, 56(1), pp.65-75.
\bibitem{tener1965} Tener, J. S. (1965). Muskoxen in Canada: A biological and taxonomic review. Ottawa: Queen's Printer.

\bibitem{davidowicz1988} Davidowicz, P., Gliwicz, Z. M., \& Gulati, R. D. (1988). Can Daphnia prevent a blue-green algal bloom in hypertrophic lakes? Limnologica, 19(2), 21–26.

\bibitem{holmes1972} Holmes, J. C., \& Bethel, W. M. (1972). Modification of intermediate host behaviour by parasites. Supplement I to Zoological Journal of the Linnean Society, 51, pp. 123-149.

\bibitem{andrews1968} Andrews, J. F. (1968). A mathematical model for the continuous culture of microorganisms utilizing inhibitory substrates. Biotechnology and bioengineering, 10(6), 707-723.

\bibitem{edwards1970} Edwards, V. H. (1970). The influence of high substrate concentrations on microbial kinetics. Biotechnology and Bioengineering, 12(5), 679-712.

\bibitem{boon1962} Boon, B., \& Laudelout, H. (1962). Kinetics of nitrite oxidation by Nitrobacter winogradskyi. Biochemical Journal, 85(3), 440-447.

\bibitem{ajraldi2011} Ajraldi, V., Pittavino, M., \& Venturino, E. (2011). Modeling herd behavior in population systems. Nonlinear Analysis: Real World Applications, 12(4), 2319-2338. 

\bibitem{braza2012} Braza, P. A. (2012). Predator–prey dynamics with square root functional responses. Nonlinear Analysis: Real World Applications, 13(4), 1837-1843.

\bibitem{djilali2019}  Djilali, S. (2019). Impact of prey herd shape on the predator-prey interaction. Chaos, Solitons \& Fractals, 120, 139-148.

\bibitem{geritz2013} Geritz, S. A. H., \& Gyllenberg, M. (2013). Group defence and the predator’s functional response. Journal of mathematical biology, 66(4), 705-717.

\bibitem{kumarkumari2021} Kumar, V., \& Kumari, N. (2021). Bifurcation study and pattern formation analysis of a tritrophic food chain model with group defense and Ivlev-like nonmonotonic functional response. Chaos, Solitons \& Fractals, 147, 110964.

\bibitem{sokol1981} Sokol, W., \& Howell, J. A. (1981). Kinetics of phenol oxidation by washed cells. Biotechnology and Bioengineering, 23(9), 2039-2049.

\bibitem{gopalsamy1992} Gopalsamy, K. (1992). Stability and Oscillations in Delay Differential Equations of Population Dynamics (Vol. 74). Springer Science \& Business Media.

\bibitem{ding2017} Ding, D., Qian, X., Hu, W., Wang, N., \& Liang, D. (2017). Chaos and Hopf bifurcation control in a fractional-order memristor-based chaotic system with time delay. The European Physical Journal Plus, 132(11), 1-12.

\bibitem{fowler1986} Fowler, A. C. (1982). An asymptotic analysis of the delayed logistic equation when the delay is large. IMA Journal of Applied Mathematics, 28(1), 41-49.

\bibitem{sun2007} Sun, H., \& Cao, H. (2007). Bifurcations and chaos of a delayed ecological model. Chaos, Solitons \& Fractals, 33(4), 1383-1393.

\bibitem{may1976} May, R. M. (1976). Simple mathematical models with very complicated dynamics. Nature, 261(5560), 459-467.

\bibitem{kundu2018} Kundu, S., \& Maitra, S. (2018). Dynamical behaviour of a delayed three species predator–prey model with cooperation among the prey species. Nonlinear Dynamics, 92(2), 627-643.

\bibitem{hutch1948} Hutchinson, G. E. (1948). Circular causal systems in ecology. Ann. NY Acad. Sci, 50(4), 221-246.

\bibitem{upadhyay2016} Upadhyay, R. K., \& Agrawal, R. (2016). Dynamics and responses of a predator–prey system with competitive interference and time delay. Nonlinear Dynamics, 83(1), 821-837.

\bibitem{murray2002} Murray, J. D. (2002). Mathematical Biology I. An Introduction (Vol. 17). New York: Springer.

\bibitem{alfifi2021} Alfifi, H. Y. (2021). Stability and Hopf bifurcation analysis for the diffusive delay logistic population model with spatially heterogeneous environment. Applied Mathematics and Computation, 408, 126362.

\bibitem{natgeo} National Geograpic. https://www.nationalgeographic.org/encyclopedia/generalist-and-specialist-species/

\bibitem{xiao2001} Xiao, D., \& Ruan, S. (2001). Global analysis in a predator-prey system with nonmonotonic functional response. SIAM Journal on Applied Mathematics, 61(4), 1445-1472.

\bibitem{leslie1948} Leslie, P. H. (1948). Some further notes on the use of matrices in population mathematics. Biometrika, 35(3/4), 213-245.

\bibitem{aziz2003} Aziz-Alaoui, M. A., \& Okiye, M. D. (2003). Boundedness and global stability for a predator-prey model with modified Leslie-Gower and Holling-type II schemes. Applied Mathematics Letters, 16(7), 1069-1075.

\bibitem{batabyal2020} Batabyal, S., Jana, D., Lyu, J. \& Parshad, R.D. (2020). Explosive predator and mutualistic preys: A comparative study. Physica A: Statistical Mechanics and its Applications, 541, p.123348.

\bibitem{morales2021} Morales-Saldaña, J. M., Herman, K. B., Mejía-Falla, P. A., Navia, A. F., Areano, E., Castillo, C. G. A., Espinoza, M., Cevallos, A., Pestana, A. G., González, A., Pérez-Jiménez, J. C., Velez-Zuazo, X., Charvet, P. \& Kyne, P. M. (2021). Eastern Pacific Round Rays, Reference Module in Earth Systems and Environmental Sciences, Elsevier, ISBN 9780124095489. https://doi.org/10.1016/B978-0-12-821139-7.00122-7.

\bibitem{sA3} Kundu, S. \& Maitra, S. (2018). Dynamics of a delayed predator-prey system with stage structure and cooperation for preys, Chaos, Solitons $\&$ Fractals, Vol. 114, pp. 453-460.

\bibitem{kaviya2021} Kaviya, R., \& Muthukumar, P. (2021). Dynamical analysis and optimal harvesting of conformable fractional prey–predator system with predator immigration. The European Physical Journal Plus, 136(5), 1-18.

\bibitem{sA1} Kundu, S. \& Maitra, S. (2016). Stability and delay in a three species predator-prey system. In: AIP Conference Proceedings, vol.1751, p. 020004. AIP Publishing.

\bibitem{prcc1} McKay, M. D., Beckman R. J., \& Conover, W. J. (1979). A comparison of three methods for selecting values of input variables in the analysis of output from a computer code. Technometrics, 21(2):239-245.

\bibitem{prcc2} Hamby, D. M. (1994). A review of techniques for parameter sensitivity analysis of environmental
models. Environmental Monitoring and Assessment, 32(2):135-154.

\bibitem{nadim2018} Nadim, S. S., Samanta, S., Pal, N., ELmojtaba, I. M., Mukhopadhyay, I., \& Chattopadhyay, J. (2018). Impact of predator signals on the stability of a Predator–Prey System: A Z-control approach. Differential Equations and Dynamical Systems, 1-17.

\bibitem{loehle2006} Loehle, C. (2006). Control theory and the management of ecosystems. Journal of applied ecology, 43(5), 957-966.

\bibitem{zhang2016} Zhang, Y., Yan, X., Liao, B., Zhang, Y., \& Ding, Y. (2016). Z-type control of populations for Lotka–Volterra model with exponential convergence. Mathematical biosciences, 272, 15-23.

\bibitem{erbe1986}
Erbe, L. H., Freedman, H. I., \& Rao, V. S. H (1986). Three-species food-chain models with mutual interference and time delays. Mathematical Biosciences, 80(1), 57-80.

\bibitem{bb2006}
Bandyopadhyay, M., \& Banerjee, S. (2006). A stage-structured prey–predator model with discrete time delay. Applied Mathematics and Computation, 182(2), 1385-1398.

\bibitem{rutz2006} Rutz, C., \& Bijlsma, R. G. (2006). Food-limitation in a generalist predator. Proceedings of the Royal Society B: Biological Sciences, 273(1597), 2069-2076.

\bibitem{jaworski2013} Jaworski, C. C., Bompard, A., Genies, L., Amiens-Desneux, E., \& Desneux, N. (2013). Preference and prey switching in a generalist predator attacking local and invasive alien pests. PLoS One, 8(12), e82231.

\end{thebibliography}
\end{document}